\PassOptionsToPackage{usenames,dvipsnames}{xcolor}
\documentclass{article}

\usepackage{arxiv}
\usepackage[utf8]{inputenc}
\usepackage[T1]{fontenc}   
\usepackage{hyperref}      
\usepackage{url}
\usepackage{booktabs}
\usepackage{amsfonts}      
\usepackage{nicefrac}      
\usepackage{microtype}
\usepackage{graphicx}
\usepackage{amsmath}
\usepackage{array}
\usepackage{doi}
\usepackage{nicematrix}
\usepackage{amsthm}
\usepackage{dsfont}
\usepackage{mathtools}
\usepackage{tcolorbox}
\usepackage{appendix}
\usepackage{bbm}
\usepackage{wasysym}
\usepackage{doi}
\usepackage{physics}
\usepackage{subcaption}
\usepackage{caption}
\usepackage{cancel}
\usepackage[normalem]{ulem}
\usepackage{algpseudocode}
\usepackage{algorithm}
\usepackage{rotating} 
\usepackage[customcolors,norndcorners]{hf-tikz}
\usepackage{tikz}
\usetikzlibrary{calc}
\usepackage{pgfplots}
\pgfplotsset{compat=newest}
\usepgfplotslibrary{patchplots}
\usetikzlibrary{pgfplots.patchplots}
\usepgfplotslibrary{groupplots}
\tikzset{set fill color=yellow, set border color=yellow}

\newtheorem{theorem}{Theorem}[section]
\newtheorem{assumption}[theorem]{Assumption}
\newtheorem{lemma}[theorem]{Lemma}
\newtheorem{proposition}[theorem]{Proposition}
\newtheorem{property}[theorem]{Property}

\newtheorem{remark}[theorem]{Remark}
\newtheorem{definition}[theorem]{Definition}

\renewcommand{\div}{\text{div}_{\bx}}
\newcommand{\bn}{\mathbf{n}}
\newcommand{\bx}{\boldsymbol{x}}
\newcommand{\bh}{\boldsymbol{h}}
\newcommand{\bu}{\boldsymbol{u}}
\newcommand{\bv}{\boldsymbol{v}}
\newcommand{\bU}{\boldsymbol{U}}
\newcommand{\bV}{\boldsymbol{V}}
\newcommand{\bff}{\boldsymbol{f}}
\newcommand{\spann}{\text{span}}
\newcommand{\numSchur}{17} 
\newcommand{\numC}{25} 
\DeclareMathOperator{\thh}{th}

\DeclareMathOperator{\maxx}{max}
\newcommand{\eps}{\varepsilon}
\renewcommand{\phi}{\varphi}
\newcommand{\mi}{\mathrm{i}} 
\newcommand{\A}{\mathbf{A}}
\newcommand{\B}{\mathbf{B}}
\newcommand{\bC}{\mathbf{C}}
\newcommand{\C}{\mathbb{C}} 
\newcommand{\N}{\mathbb{N}}
\newcommand{\F}{\mathbf{F}}
\newcommand{\E}{\mathbf{E}}
\newcommand{\J}{\mathbf{J}}
\newcommand{\K}{\mathbf{K}}
\newcommand{\W}{\mathbf{W}}
\renewcommand{\H}{\mathbf{H}}
\newcommand{\M}{\mathbf{M}}
\newcommand{\D}{\mathbf{D}}
\newcommand{\G}{\mathbf{G}}
\renewcommand{\L}{\mathbf{L}}
\newcommand{\Z}{\mathbb{Z}}
\newcommand{\R}{\mathbb{R}}
\renewcommand{\P}{\mathbb{P}}

\definecolor{champagne}{rgb}{0.97, 0.91, 0.81}
\definecolor{dred}{rgb}{0.8, 0.0, 0.0}
\definecolor{dblue}{rgb}{0.0, 0.5, 0.69}
\definecolor{dgreen}{rgb}{0.0, 0.5, 0.0}
\definecolor{dviolet}{rgb}{0.54, 0.17, 0.89}

\newcommand\coolover[2]{\mathrlap{\smash{\overbrace{\phantom{%
    \begin{matrix} #2 \end{matrix}}}^{\mbox{$#1$}}}}#2}
\newcommand\coolunder[2]{\mathrlap{\smash{\underbrace{\phantom{%
    \begin{matrix} #2 \end{matrix}}}_{\mbox{$#1$}}}}#2}
\newcommand\coolleftbrace[2]{%
#1\left\{\vphantom{\begin{matrix} #2 \end{matrix}}\right.}
\newcommand\coolrightbrace[2]{%
\left.\vphantom{\begin{matrix} #1 \end{matrix}}\right\}#2}

\newcommand{\tikzmark}[1]{\tikz[overlay,remember picture] \node (#1) {};}
\newcommand{\DrawBox}[4][]{%
    \tikz[overlay,remember picture]{%
        \coordinate (TopLeft)     at ($(#2)+(-0.2em,0.9em)$);
        \coordinate (BottomRight) at ($(#3)+(0.2em,-0.3em)$);
        \path (TopLeft); \pgfgetlastxy{\XCoord}{\IgnoreCoord};
        \path (BottomRight); \pgfgetlastxy{\IgnoreCoord}{\YCoord};
        \coordinate (LabelPoint) at ($(\XCoord,\YCoord)!0.5!(BottomRight)$);
        \draw [red,#1] (TopLeft) rectangle (BottomRight);
        \node [below, #1, fill=none, fill opacity=1] at (LabelPoint) {#4};
    }
}

\newcommand{\susm}{Unconditionally stable method~\eqref{eq:4}}
\newcommand{\fusm}{Unconditionally stable method in~\cite{FrenchPeterson1996}}
\newcommand{\ssm}{Stabilized method in~\cite{FraschiniLoliMoiolaSangalli2023}}

\newcommand{\circleasterisk}{%
  \begin{tikzpicture}[baseline=-0.5ex]
    \draw (0,0) circle (0.5em);
    \node at (0,0) {$*$};
  \end{tikzpicture}%
}

\title{Unconditionally stable space--time isogeometric \\ discretization  for the wave equation \\ in Hamiltonian formulation}

\author{Matteo ~Ferrari \\
Fakult\"at f\"ur Mathematik \\
Universit\"at Wien\\
1090 Vienna, Austria \\
\texttt{\phantom{1234} matteo.ferrari@univie.ac.at \phantom{567}} \\
\And
Sara ~Fraschini \\
Fakult\"at f\"ur Mathematik \\
Universit\"at Wien\\
1090 Vienna, Austria \\
\texttt{sara.fraschini@univie.ac.at} \\
\And
Gabriele Loli \\
Dipartimento di Matematica "Felice Casorati"\\
Università di Pavia\\
27100 Pavia, Italy \\
\texttt{gabriele.loli@unipv.it} \\
\And
Ilaria Perugia \\
Fakult\"at f\"ur Mathematik \\
Universit\"at Wien\\
1090 Vienna, Austria \\
\texttt{ilaria.perugia@univie.ac.at} \\
}

\renewcommand{\undertitle}{}
\renewcommand{\shorttitle}{}

\hypersetup{
pdftitle={stability FP space-time iso},
pdfsubject={q-bio.NC, q-bio.QM},
pdfauthor={Ferrari M.~Fraschini S.~Loli G.~Perugia I.},
pdfkeywords={spline space-time},
}

\begin{document}
\maketitle

\begin{abstract}
We consider a family of conforming space--time discretizations for the wave equation based on a first-order-in-time formulation employing maximal regularity splines. In contrast with second-order-in-time formulations, which require a CFL condition to guarantee stability, the methods we consider here are unconditionally stable without the need for stabilization terms. Along the lines of the work by M. Ferrari and S. Fraschini (2024), we address the stability analysis by studying the properties of the condition number of a family of matrices associated with the time discretization. Numerical tests validate the performance of the method.
\end{abstract}

\section{Introduction}
We consider the following acoustic wave problem with Dirichlet boundary conditions:
\begin{equation} \label{eq:1}
	\begin{cases}
		\partial_{t}^2 U(\bx,t) - \div(c^2(\bx) \nabla_{\bx} U(\bx,t)) = F(\bx,t) & (\bx,t) \in Q_T := \Omega \times (0, T), 
		\\ U(\bx,t) = g(\bx,t) & (\bx,t) \in \Sigma_T := \Gamma \times [0, T],
		\\ U(\bx,0) = U_0(\bx), \quad \partial_t U(\bx,t)_{|_{t=0}} = V_0(\bx) & \bx \in \Omega,
	\end{cases}
\end{equation}
where~$\Omega \subset \R^d$ ($d = 1, 2, 3$) is a bounded Lipschitz domain with boundary~$\Gamma := \partial \Omega$,~$T > 0$ is a finite time,~$F \in L^2(Q_T)$ is a given source term,~$c \in L^\infty(\Omega)$ is a positive wave velocity independent of~$t$, and the given boundary and initial data satisfy
\begin{equation*}
    g \in H^{\nicefrac{1}{2}}(\Sigma_T), \quad U_0 \in H^1(\Omega), \quad V_0 \in L^2(\Omega), \quad \text{~and~} \quad  {U_0} = g{}_{|_{t=0}} \quad\text{on~$\Gamma$}.
\end{equation*}
Here and throughout the paper we follow standard notation for differential operators, function spaces and norms that can be found, for example, in~\cite{Brezis2010}.

Various discretization techniques are available to compute an approximate numerical solution of problem~\eqref{eq:1}. In contrast to the more standard approaches based on separate discretizations of space and time, space--time methods, introduced in the seminal papers~\cite{French1993, HH1988, Johnson1993}, provide a simultaneous discretization of space and time variables. Thanks to their features (e.g. high-order approximation in space and time, space--time unstructured meshes, and parallelization) and to advances in computer technology, the investigation of space--time methods has increased recently, leading to the development of several space--time discretizations for wave propagation problems. Among them, we mention space--time discontinuous Galerkin methods (see e.g.~\cite{BGL2017,BMPS2021,BCDS2020,DorflerFindeisenWieners2016,MoiolaPerugia2018,MR2005, PerugiaSchoberlStockerWintersteiger2020}), and conforming space--time discretizations (see e.g.~\cite{BalesLasiecka1994, HPSU22, BM2023, FGK2023, SteinbachZank2019, LoscherSteinbachZank2023}). Here, we focus on the latter class of methods. 

A second-order-in-time space--time variational formulation is considered in~\cite{SteinbachZank2019, SteinbachZank2020}. It is obtained by integrating by parts~\eqref{eq:1} both in time and space. Unique solvability of that formulation is proven (see e.g.~\cite[Theorem 5.1]{SteinbachZank2020} for~$c \equiv 1$ and~$V_0 \equiv 0$), but an inf-sup stability in standard Sobolev norms in not satisfied (see~\cite[Theorem 4.2.24]{Zank2020}). Therefore, stability of conforming space--time finite element discretizations may be achieved under suitable CFL conditions; see~\cite{FerrariFraschini2024} for explicit and sharp results on this. In order to recover second-order-in-time unconditionally stable methods, various possibilities have been explored. One consists of testing with optimal test functions written in terms of suitable operators (e.g. the modified Hilbert transform~\cite{LoscherSteinbachZank2023}). Another possibility is to stabilize the corresponding bilinear form by adding appropriate (non-consistent) penalty terms. The different choices of this stabilization depend only on the discretization of the temporal part. In~\cite{SteinbachZank2019}, a stabilization for continuous piecewise linear functions in time has been proposed and analyzed. That idea has been then generalized to higher order continuous piecewise polynomials in~\cite{Zank2021}, and to higher order maximal regularity splines in~\cite{FraschiniLoliMoiolaSangalli2023, FerrariFraschini2024}. 

In this paper, we study a numerical method based on a first-order-in-time formulation of~\eqref{eq:1} obtained by introducing the auxiliary unknown~$V := \partial_t U$. Both~$U$ and~$V$ are discretized in the same discrete space, while test functions are taken in a space with lower polynomial degree and regularity. In~\cite{BalesLasiecka1994, FrenchPeterson1996}, such a space--time method was proposed, with discretization in time performed with continuous piecewise polynomials as trial functions, and discontinuous piecewise polynomials of one degree less as test functions. In~\cite{BalesLasiecka1994} and~\cite{FrenchPeterson1996}, via matricial and variational arguments, respectively, unconditional stability, as well as error estimates, have been proved and numerically verified for~$c \equiv 1$. Note that testing with discontinuous piecewise polynomial functions in time guarantees a time-stepping procedure. The temporal part of these schemes is actually equivalent to a Runge-Kutta Gauss-Legendre method, see~\cite[Section 2]{FrenchSchaeffer1990}. In this work, we analyze a discretization of the first-order-in-time formulation with high order maximal regularity splines in time. We discretize both~$U$ and~$V$ in the same spline spaces. Test functions are taken in the spline space of one less polynomial degree and regularity. Although a complete well-posedness and error analysis is still out of reach at the moment, via a matricial argument along the lines of~\cite{FerrariFraschini2024}, the resulting method is shown to be unconditionally stable without the need of additional stabilization terms. The analysis is focused on a system of ordinary differential equations, which is strongly related to the time part of the wave equation, and stability is derived by exploiting the algebraic structure of the matrices involved. This analysis combines two main tools: properties of the symbols~\cite{GaroniSpeleersEkstromRealiSerraCapizzanoHughes2019} of the matrices associated with spline discretizations, and the behaviour of the condition number of general Toeplitz band matrices characterized in~\cite{AmodioBrugnano1996}. With the same techniques, we also show that a CFL condition is required when both test and trial spaces in time are splines of the same degree and maximal regularity. This CFL condition turns out to be sharp.

The paper is structured as follows. In Section~\ref{sec2}, we introduce the discrete first-order-in-time variational formulation of~\eqref{eq:1}, its discretization with maximal regularity splines in time, and we discuss the properties of the associated Galerkin matrices for the temporal part. Furthermore, we recall results on the conditioning of families of Toeplitz band matrices. In Section~\ref{sec:3}, we present our main results on the stability of proposed method. In Section~\ref{sec:4}, we consider the numerical scheme with equal trial and test spaces in time, and we explain mathematically why it is only conditionally stable. Finally, in Section~\ref{sec:5}, we present an efficient algorithm for the solving the linear system involved, and various numerical tests on the full space--time problem with isogeometric discretization also in the space variables, which demonstrate the performance of the method and its unconditional stability. 

\section{First-order-in-time variational formulation and temporal discretization with maximal regularity splines} \label{sec2}
With the auxiliary unknown~$V := \partial_t U$, problem~\eqref{eq:1} is reformulated as follows:
\begin{equation} \label{eq:2}
	\begin{cases}
		\partial_t U(\bx,t) - V(\bx,t) = 0 & (\bx,t) \in Q_T, 
		\\ \partial_t V(\bx,t) - \div(c^2(\bx) \nabla_{\bx} U(\bx,t)) = F(\bx,t) & (\bx,t) \in Q_T, 
        \\ U(\bx,t) = g(\bx,t) & (\bx,t) \in \Sigma_T,
		\\ U(\bx,0) = U_0(\bx), \quad V(\bx,0) = V_0(\bx) & \bx \in \Omega.
	\end{cases}
\end{equation}
From now, we restrict, for simplicity, to the case of~$U_0 \equiv 0$,~$V_0 \equiv 0$, and~$g \equiv 0$. The general case can be readily considered with a suitable lifting.

For the spatial discretization, let us consider a discrete space~$V_{h_{\bx}}(\Omega) \subset H_0^1(\Omega)$ depending on a spatial parameter~$h_{\bx}$ (e.g. piecewise linear, continuous functions over a triangulation of~$\Omega$ of mesh size~$h_{\bx}$). For the temporal discretization, we introduce the space~$S_{h_t}^{(p,k)}(0,T)$ of splines of polynomial degree~$p \ge 0$ and regularity~$C^k$, with~$k \ge -1$ ($C^{-1}$ allowing for discontinuous functions) over a uniform mesh of~$[0,T]$ made of~$N_t$ intervals with mesh size~$h_t := \nicefrac{T}{N_t}$. We denote the subspaces of~$S_{h_t}^{(p,k)}(0,T)$ incorporating zero initial and final conditions, respectively, by~$S_{h_t,0,\bullet}^{(p,k)}(0,T)$ and~$S_{h_t,\bullet,0}^{(p,k)}(0,T)$. Then, we define the tensor product spaces~$Q_{\bh}^{(p,k)}(Q_T) := V_{h_{\bx}}(\Omega) \otimes S_{h_t}^{(p,k)}(0,T)$,~$Q_{\bh,0,\bullet}^{(p,k)}(Q_T) := V_{h_{\bx}}(\Omega) \otimes S_{h_t,0,\bullet}^{(p,k)}(0,T)$, and~$Q_{\bh,\bullet,0}^{(p,k)}(Q_T) := V_{h_{\bx}}(\Omega) \otimes S_{h_t,\bullet,0}^{(p,k)}(0,T)$.

We discretize~\eqref{eq:2} as follows: find~$(U^p_{\bh}, V^p_{\bh}) \in Q_{\bh,0,\bullet}^{(p,p-1)}(Q_T) \times Q_{\bh,0,\bullet}^{(p,p-1)}(Q_T)$ such that
\begin{equation} \label{eq:3}
	\begin{cases}
		(\partial_t U^p_{\bh}, \chi^{p-1}_{\bh})_{L^2(Q_T)} - (V^p_{\bh}, \chi^{p-1}_{\bh})_{L^2(Q_T)} =0 & \text{for all~} \chi^{p-1}_{\bh} \in Q_{\bh}^{(p-1,p-2)}(Q_T), 
		\\ (\partial_t V^p_{\bh}, \lambda^{p-1}_{\bh})_{L^2(Q_T)} + (c^2 \nabla_{\bx} U^p_{\bh}, \nabla_{\bx} \lambda^{p-1}_{\bh})_{L^2(Q_T)} 
         = (F, \lambda^{p-1}_{\bh})_{L^2(Q_T)} & \text{for all~} \lambda^{p-1}_{\bh} \in Q_{\bh}^{(p-1,p-2)}(Q_T).
	\end{cases}
\end{equation}
\begin{remark}
The scheme proposed in~\cite{BalesLasiecka1994, FrenchPeterson1996} reads as~\eqref{eq:3} but with trial spaces~$Q_{\bh,0,\bullet}^{(p,0)}(Q_T)$ and test spaces~$Q_{\bh}^{(p-1,-1)}(Q_T)$. 
\end{remark}
Alternatively, noticing that for all~$p \ge 1$ and~$k = 0,\ldots, p-1$ it holds~$\partial_t : S_{h_t,\bullet,0}^{(p,k)}(0,T) \to S_{h_t}^{(p-1,k-1)}(0,T)$ is an isomorphism (when~$k=0$, the derivative is intended piecewise), we can rewrite the formulation with maximal regularity splines, after integration by parts, as: find~$(U^p_{\bh}, V^p_{\bh}) \in Q_{\bh,0,\bullet}^{(p,p-1)}(Q_T) \times Q_{\bh,0,\bullet}^{(p,p-1)}(Q_T)$ such that
\begin{equation} \label{eq:4}
	\begin{cases}
		(\partial_t U^p_{\bh}, \partial_t \chi^p_{\bh})_{L^2(Q_T)} + (\partial_t V^p_{\bh}, \chi^p_{\bh})_{L^2(Q_T)} = 0 & \text{for all~} \chi^p_{\bh} \in Q_{\bh,\bullet,0}^{(p,p-1)}(Q_T),
		\\ (\partial_t V^p_{\bh}, \partial_t \lambda^p_{\bh})_{L^2(Q_T)} - (c^2 \nabla_{\bx} \partial_t U^p_{\bh}, \nabla_{\bx} \lambda^p_{\bh})_{L^2(Q_T)} = (F, \partial_t \lambda^p_{\bh})_{L^2(Q_T)} & \text{for all~} \lambda^p_{\bh} \in Q_{\bh,\bullet,0}^{(p,p-1)}(Q_T).
	\end{cases}
\end{equation}
In the following, we will focus on the discrete variational formulation~\eqref{eq:4}. By studying an ODE system in the time variable, which is derived from an eigenfunction expansion in space, we perform a matricial analysis along the lines of~\cite{FerrariFraschini2024} in order to address the stability of scheme~\eqref{eq:4}.

\subsection{Associated ODE and matricial formulation}
Let us consider the following eigenvalue problem: 
\begin{equation*}
	\begin{cases}
		 - \div(c^2(\bx) \nabla \Psi(\bx)) = \mu \Psi(\bx) & \bx \in \Omega, 
        \\ \Psi(\bx) = 0 & \bx \in \Gamma.
	\end{cases}
\end{equation*}
Due to the uniform ellipticity and self-adjointness of the problem, the eigenvalues form an unbounded sequence of positive real numbers (see   e.g.~\cite[Section 9.8]{Brezis2010}),
\begin{equation*}
    0 < \mu_1 \le \mu_2 \le \ldots \le \mu_j \le \ldots \to +\infty.
\end{equation*}
Therefore, exploiting the Fourier expansion of the exact solution, based on analogous considerations as in~\cite{FerrariFraschini2024, Zank2020}, our focus lies in proposing a stable discretization with respect to the parameter~$\mu>0$ for the ODE system
\begin{equation} \label{eq:5}
	\begin{cases}
		\partial_t u(t) - v(t) = 0 & t \in (0,T), \\
		\partial_t v(t) + \mu u(t) = f(t) & t \in (0,T), 
		\\ u(0) = 0, \quad v(0) = 0, &
	\end{cases}
\end{equation}
with~$f \in L^2(0,T)$.
\begin{remark} \label{rem:22}
Neumann boundary conditions in~\eqref{eq:1} can be considered similarly, and our analysis readily extends to this situation. However, problem~\eqref{eq:1} with Robin boundary conditions cannot be recast in our theoretical framework. Indeed, due to the combination of temporal and spatial derivatives in the boundary condition, one cannot perform an eigenfunction expansion in space, and trace the problem back to the study of an ODE system in time. However, a numerical test with Robin's boundary conditions is reported in Section~\ref{subsec6.3} below.
\end{remark}
Let us fix~$N \in \N$ and set~$h:=T/N$. Then, the discrete variational formulation for~\eqref{eq:5} analogous to~\eqref{eq:4} reads: find~$(u_h^p, v_h^p) \in S_{h,0,\bullet}^{(p,p-1)}(0,T) \times S_{h,0,\bullet}^{(p,p-1)}(0,T)$ such that
\begin{equation} \label{eq:6}
	\begin{cases}
		(\partial_t u_h^p, \partial_t \chi_h^p)_{L^2(0,T)} + (\partial_t v_h^p, \chi_h^p)_{L^2(0,T)} = 0 & \text{for all~} \chi_h^p \in S_{h,\bullet,0}^{(p,p-1)}(0,T),
		\\ (\partial_t v_h^p, \partial_t \lambda_h^p)_{L^2(0,T)} - \mu (\partial_t u_h^p, \lambda_h^p)_{L^2(0,T)} = (f,\partial_t \lambda_h^p)_{L^2(0,T)} & \text{for all~} \lambda_h^p \in S_{h,\bullet,0}^{(p,p-1)}(0,T).
	\end{cases}
\end{equation}
Similarly as performed in~\cite{FerrariFraschini2024} for the second-order-in-time variational formulation, we analyze the condition number of a family of matrices associated with~\eqref{eq:6}.

Let consider the B-spline basis~$\{ \phi_j^p\}_{j=0}^{N+p-1}$ of degree~$p$ with maximal regularity~$C^{p-1}$ defined according to the Cox-De Boor recursion formula~\cite{DeBoor1972}. In particular, the basis is defined such that 
\begin{equation*}
    S_{h,0,\bullet}^{(p,p-1)}(0,T) = \spann\{\phi_j^p : j = 1,\ldots, N+p-1\}, \quad S_{h,\bullet,0}^{(p,p-1)}(0,T) = \spann\{\phi_j^p : j = 0,\ldots, N+p-2\}. 
\end{equation*}
The matrices involved in the first-order discrete variational formulation~\eqref{eq:6} are defined as
\begin{equation} \label{eq:7}
	\B^p_h[\ell,j] := (\partial_t \phi_j^p, \partial_t \phi^p_{\ell-1})_{L^2(0,T)}, \quad \bC^p_h[\ell,j] := (\partial_t \phi_j^p, \phi^p_{\ell-1})_{L^2(0,T)},\qquad
 \ell, j = 1,\ldots,N+p-1.
\end{equation}
These matrices have specific structures, due to the properties of B-splines basis functions. Here, we explicitly write the entries of the matrices~$\B_h^2, \bC_h^2 \in \R^{(N+1)\times (N+1)}$ to highlight their structure:
{\small{
\begin{align*}
\B_h^2 = \frac{1}{6h} 
\setcounter{MaxMatrixCols}{20}
    \begin{pmatrix} 
            -6 & -2  \\ 
            8 & -1 & -1 \\
            -1 & 6 & -2 & -1 \\
            -1 & -2 & 6 & -2 & -1 \\
            & \ddots & \ddots & \ddots & \ddots & \ddots \\
            & & -1 & -2 & 6 & -2 & -1 \\
            & & & -1 & -2 & 6 & -1 & -2 \\
            & & & & -1 & -1 & 8 & -6 \\
    \end{pmatrix}, \hspace{0.5cm}
\bC_h^2 & = \frac{1}{24} 
\setcounter{MaxMatrixCols}{20}
    \begin{pmatrix} 
            10 & 2  \\
            0 & 9 & 1 \\
            -9 & 0 & 10 & 1 \\
            -1 & -10 & 0 & 10 & 1 \\
            & \ddots & \ddots & \ddots & \ddots & \ddots \\
            & &  -1 & -10 & 0 & 10 & 1 \\
            & & & -1 & -10 & 0 & 9 & 2 \\
            & & & & -1 & -9 & 0 & 10 \\
    \end{pmatrix}.
\end{align*}}}
With the exception of~$5$ entries at the top left and~$5$ at the bottom right corners, they are Toeplitz band matrices with symmetries. We recall the results obtained in~\cite[Proposition 3.2]{FerrariFraschini2024} for~$\B_h^p$, which can be extended to similar ones for~$\bC_h^p$.
\begin{proposition} \label{prop:23}
Let~$p \ge 1$ and let~$\B_h^p$ and~$\bC_h^p$ be defined in~\eqref{eq:7}. Then, the following properties are valid:
\begin{enumerate} 
\item The entries of the matrices~$h\B_h^p$ and~$\bC_h^p$ are independent of the mesh parameter~$h$.
\item The matrices~$\B_h^p$ and~$\bC_h^p$ are persymmetric, i.e., they are symmetric about their northeast-to-southwest diagonal (anti-diagonal).
\item The matrices~$\B_h^1$ and~$\bC_h^1$ are lower triangular Toeplitz band matrices with three nonzero diagonals. For~$p>1$, except for~$2p^2-3$ entries located at the top left and bottom right corners, the matrices~$\B_h^p$ and~$\bC_h^p$ exhibit a Toeplitz band structure. In particular, in the top left corner, precisely the nonzero entries of the first~$p$ rows and the first~$p-1$ columns, with the exception of the entries in position~$(p,2p-1)$ and~$(2p,p-1)$, do not respect the Toeplitz structure. The precise structure of that block is as follows:
\begin{equation*} 
    \vphantom{
    \begin{matrix}
    \overbrace{XYZ}^{\mbox{$R$}}\\ \\ \\ \\ \\ \\ \\ \\
    \underbrace{pqr}_{\mbox{$S$}} \\
    \end{matrix}}
    \begin{matrix}
        \vphantom{a}
        \coolleftbrace{p+1}{* \\ \vdots \\ * \\ * \vspace{0.5cm}} \vspace{0.1cm}\\ 
        \coolleftbrace{p-1}{* \\ \vdots \\ *}
    \end{matrix}
    \begin{pmatrix}
        \coolover{p-1}{*     & \ldots &     * \hspace{0.1cm}}& \coolover{p}{* & \phantom{***} & \phantom{*}  & \phantom{\circleasterisk}} 
        \\        \vdots &        & \vdots               &         \vdots & \ddots           & 
        \\        *      & \ldots & *                    & *              &  \ldots          &  *           &
        \\         *     & \ldots &     *                &              * &  \ldots          &  *           & \circleasterisk &
        \\        *      & \ldots & *                    &                &                  &              &
        \\         *     & \ldots &     *                &                &                  &              &              &
        \\               & \ddots & \vdots               &                &                  & 
        \\               &        & \circleasterisk      &                &                  &  
    \end{pmatrix}.
\end{equation*}
\item In the purely Toeplitz band part,~$\B_h^p$ and~$\bC_h^p$ exhibit symmetry and skew-symmetry, respectively, with respect to the first lower co-diagonal. In detail, the non-vanishing elements of the purely Toeplitz band parts of~$h \B_h^p$ and~$\bC_h^p$ can be expressed as
\begin{equation*}
    h \B^p_h[\ell,\ell -1 \pm j]  = -\partial_t^2\Phi_{2p+1} (p+1-j), \quad  \bC^p_h[\ell,\ell -1 \pm j] =
    \begin{cases}
        \pm \partial_t \Phi_{2p+1} (p+1-j), & \text{if~} j  \ne 0,
        \\ 0 & \text{if~} j = 0,
    \end{cases}
\end{equation*}
for~$j=0,\ldots,p$ and~$\ell=2p+1,\ldots,N-p$, where~$\Phi_j$ is the cardinal spline of degree~$j$, i.e., the spline function of degree~$j$ and regularity~$C^{j-1}$ defined over the uniform knot sequence~$\{0,\ldots,j+1\}$ (see e.g.~\cite[Section 3]{GaroniManniPelosiSerraCapizzanoSpeleers2014} for precise definition and properties).
\end{enumerate}
\end{proposition}
\begin{proof}
The first three properties readily follow from the definition~\eqref{eq:7}. The fourth one follows from an alternative definition of the entries of~$\B_h^p$ and~$\bC_h^p$ via cardinal splines (see~\cite[Equations (3.2) and (3.3)]{FerrariFraschini2024}), along with the expression for the inner products of derivatives of the cardinal spline~\cite[Lemma 4]{GaroniManniPelosiSerraCapizzanoSpeleers2014}, and the symmetry property of their derivatives~\cite[Lemma 3]{GaroniManniPelosiSerraCapizzanoSpeleers2014}.
\end{proof}
Henceforth, we assume~$N\ge 3p+1$, so there is at least one row in the purely Toeplitz band part of the matrices.

Let us represent the unknown discrete solution~$(u_h^p, v_h^p) \in S_{h,0,\bullet}^{(p,p-1)} \times S_{h,0,\bullet}^{(p,p-1)}$ with respect to the B-spline basis~$\{\phi_j^p\}_{j=1}^{N+p-1}$:
\begin{equation} \label{eq:8}
    u_h^p(t) = \sum_{j=1}^{N+p-1} u_j^p \phi_j^p(t), \quad \quad v_h^p(t) = \sum_{j=1}^{N+p-1} v_j^p \phi_j^p(t),
\end{equation}
and let the vectors~$\bu_h^p, \bv_h^p, \bff_h^p \in \R^{N+p-1}$ be defined, for~$j=1,\ldots,N+p-1$, as
\begin{equation} \label{eq:9}
  \bu_h^p := [\, u_j^p \,]_{j=1}^{N+p-1}, \quad \bv_h^p := [\, v_j^p \,]_{j=1}^{N+p-1}, \quad \text{and} \quad \bff_h^p := [\,f_j^p\,]_{j=1}^{N+p-1}, \quad \text{with~} \, f_j^p := (f,\partial_t \phi_{j-1}^p)_{L^2(0,T)}.
\end{equation}
The linear system representing the discrete variational formulation~\eqref{eq:6} with respect to the B-spline basis is then
\begin{equation} \label{eq:10}
	\begin{bmatrix}
	\B_h^p & \bC_h^p \vspace{0.1cm} \\
	-\mu \bC_h^p & \B_h^p
	\end{bmatrix} \begin{bmatrix}
	\bu_h^p \vspace{0.1cm} \\ \bv_h^p
	\end{bmatrix} = \begin{bmatrix}
	\boldsymbol{0} \vspace{0.1cm} \\ \bff_h^p
	\end{bmatrix}.
\end{equation}
We now make a fundamental assumption, which appears to be true in practice for all~$p\ge 1$
\begin{assumption} \label{assu:1}
We assume that~$p \in \N$ is such that the system matrix in~\eqref{eq:10} is invertible for all~$\mu, h > 0$.
\end{assumption}
This assumption is clearly satisfied for~$p=1$. In fact, in this case,~\eqref{eq:6} coincides with the scheme proposed and analysed in~\cite{FrenchPeterson1996}; see also Remark~\ref{rem:P1} below.

In the following remark, we comment on why, at the state of the art, establishing the uniqueness of the solution of the linear system seems to be out of reach with the matricial analysis proposed in this paper. Therefore, we postpone this, as well as the inf-sup stability and error estimates, to future research.
\begin{remark}[Uniqueness at the continuous level] \label{rem:25}
Consider the following homogeneous problem: find~$(u,v) \in H^1_{0,\bullet}(0,T) \times H^1_{0,\bullet}(0,T)$ such that
\begin{equation} \label{eq:11}
	\begin{cases}
		(\partial_t u, \partial_t \chi)_{L^2(0,T)} + (\partial_t v, \chi)_{L^2(0,T)} = 0 & \text{for all~} \chi \in H^1_{\bullet,0}(0,T),
		\\ (\partial_t v, \partial_t \lambda)_{L^2(0,T)} - \mu (\partial_t u, \lambda)_{L^2(0,T)} = 0 & \text{for all~} \lambda \in H^1_{\bullet,0}(0,T),
	\end{cases}
\end{equation}
where we have used the notation~$H_{0,\bullet}^1(0,T)$ and~$H_{\bullet,0}^1(0,T)$ to indicate the subspaces of~$H^1(0,T)$ with zero initial and final conditions, respectively. Uniqueness of the solution~$(u,v)=(0,0)$ can be proven by taking~$\chi(t) = \mu \mathcal{H}_T u(t)$ and~$\lambda(t) = \mathcal{H}_T v(t)$, being~$\mathcal{H}_T : H_{0,\bullet}^1(0,T) \to H_{\bullet,0}^1(0,T)$ the modified Hilbert transform introduced in~\cite[Section 2.4]{SteinbachZank2020}. Actually,~$\mathcal{H}_T$ is also defined from~$L^2(0,T)$ to~$L^2(0,T)$, and the following properties hold true (see~\cite[Lemma 2.3 and Lemma 2.4]{SteinbachZank2020} and~\cite[Lemma 2.2]{LoscherSteinbachZank2024}):
\begin{equation} \label{eq:12}
\begin{aligned}
    (\partial_t \mathcal{H}_T w_1, w_2)_{L^2(0,T)} & = -(\partial_t w_1, \mathcal{H}_T w_2)_{L^2(0,T)} \quad \phantom{12345} \text{for all~} w_1 \in H^1_{0,\bullet}(0,T), w_2 \in L^2(0,T), \\
    (w, \mathcal{H}_T w)_{L^2(0,T)} & > 0 \quad \phantom{12345678901234567890121} \text{for all~} 0 \ne  w \in H^{\eps}(0,T), \, \, \eps > 0.
\end{aligned}
\end{equation}
Therefore, summing the two equations in~\eqref{eq:11}, integrating by parts, and employing the above mentioned properties, we obtain
\begin{equation*}
    0 = \mu (\partial_t u, \partial_t \mathcal{H}_T u)_{L^2(0,T)} + (\partial_t v, \partial_t \mathcal{H}_T v)_{L^2(0,T)} = -\mu (\partial_t u, \mathcal{H}_T \partial_t u)_{L^2(0,T)} - (\partial_t v, \mathcal{H}_T \partial_t v)_{L^2(0,T)},
\end{equation*}
which implies~$u \equiv v \equiv 0$ if~$u, v \in H_{0,\bullet}^1(0,T)\cap H^{1+\eps}(0,T)$ for some~$\eps>0$. This argument could not be directly applied at the discrete level since~$\mathcal{H}_T (S_{h,0,\bullet}^{(p,p-1)}(0,T)) \not\subseteq
S_{h,\bullet,0}^{(p,p-1)}(0,T)$, and a modified Hilbert transform based projection in spline spaces in the spirit of~\cite{LoscherSteinbachZank2024} still need to be analyzed.
\end{remark}
In the following remark, with variational arguments and without exploiting spline properties, we show that we can establish the uniqueness of the solution of the linear system in~\eqref{eq:10} only for~$\mu$ sufficiently small. 
\begin{remark}[Uniqueness at the discrete level for small~$\mu$]
If
\begin{equation} \label{eq:13}
    4\mu T^2 < \pi^2,
\end{equation} 
then the system in~\eqref{eq:10} admits a unique solution for all~$h >0$ and~$p \in \N$. Indeed, let us assume that 
\begin{equation} \label{eq:14}
    \begin{bmatrix}
        \B_h^p & \bC_h^p \vspace{0.1cm}
        \\ -\mu \bC_h^p & \B_h^p
    \end{bmatrix}
    \begin{bmatrix}
	   \bu_h^p \vspace{0.1cm}
        \\ \bv_h^p 
    \end{bmatrix}
    =
    \begin{bmatrix}
	   \mathbf{0} \vspace{0.1cm}
        \\ \mathbf{0}
    \end{bmatrix},
\end{equation}
with the vectors~$\bu_h^p, \bv_h^p$ as in~\eqref{eq:9}. Introducing the functions~$u_h^p, v_h^p \in S_{h,0,\bullet}^{(p,p-1)}(0,T)$ as in~\eqref{eq:8}, system~\eqref{eq:14} can be written equivalently as
\begin{equation} \label{eq:15}
	\begin{cases}
		(\partial_t u_h^p, \partial_t \chi_h^p)_{L^2(0,T)} + (\partial_t v_h^p, \chi_h^p)_{L^2(0,T)} = 0 & \text{for all~} \chi_h^p \in S_{h,\bullet,0}^{(p,p-1)}(0,T),
		\\ (\partial_t v_h^p, \partial_t \lambda_h^p)_{L^2(0,T)} - \mu (\partial_t u_h^p, \lambda_h^p)_{L^2(0,T)} = 0 & \text{for all~} \lambda_h^p \in S_{h,\bullet,0}^{(p,p-1)}(0,T).
	\end{cases}
\end{equation}
 Taking~$\chi_h^p(t) = v_h^p(t) - v_h^p(T)$ and~$\lambda_h^p(t) = u_h^p(t) - u_h^p(T)$ in~\eqref{eq:15}, and subtracting the two equations, we obtain
\begin{equation} \label{vTuT}
	\begin{aligned}
		0 & = (\partial_t u_h^p, \partial_t v_h^p)_{L^2(0,T)} + (\partial_t v_h^p, v_h^p - v_h^p(T))_{L^2(0,T)} - (\partial_t v_h^p, \partial_t u_h^p)_{L^2(0,T)}
   + \mu (\partial_t u_h^p, u_h^p - u_h^p(T))_{L^2(0,T)}
   \\ & = (\partial_t v_h^p, v_h^p)_{L^2(0,T)} - v_h^p(T) (\partial_t v_h^p, 1 )_{L^2(0,T)} + \mu (\partial_t u_h^p, u_h^p) - \mu u_h^p(T) (\partial_t u_h^p, 1)_{L^2(0,T)}
   \\ & = -\frac{|v_h^p(T)|^2}{2} - \mu \frac{|u_h^p(T)|^2}{2}.
	\end{aligned}
\end{equation}
From the latter, since~$\mu>0$, we deduce ~$v_h^p(T)=u_h^p(T)=0$. In view of this,~$\chi_h^p(t) = \mu u_h^p(t)$ and~$\lambda_h^p(t) = v_h^p(t)$ are admissible test functions in~\eqref{eq:15}. Taking these test functions, and subtracting the two equations, we deduce that
\begin{equation} \label{eq:16}
    | v_h^p |_{H^1(0,T)}^2 = \mu | u_h^p |_{H^1(0,T)}^2.
\end{equation}
Moreover, adding the two equations, we obtain 
\begin{equation} \label{eq:17}
	\begin{aligned}
		0 & = \mu (\partial_t u_h^p, \partial_t u_h^p)_{L^2(0,T)} + \mu (\partial_t v_h^p, u_h^p)_{L^2(0,T)} + (\partial_t v_h^p, \partial_t v_h^p)_{L^2(0,T)} - \mu (\partial_t u_h^p, v_h^p)_{L^2(0,T)}
  \\ & = \mu | u_h^p |^2_{H^1(0,T)} + | v_h^p |^2_{H^1(0,T)} +  2 \mu (\partial_t v_h^p, u_h^p)_{L^2(0,T)}.
	\end{aligned}
\end{equation}%
Combining~\eqref{eq:16},~\eqref{eq:17}, the Cauchy-Schwarz inequality, and a sharp version of Poincar\'e's inequality (see e.g.~\cite[Lemma 3.4.5]{Zank2020}), we deduce
\begin{align*}
    0 = | u_h^p |^2_{H^1(0,T)} + (\partial_t v_h^p , u_h^p)_{L^2(0,T)} & \ge | u_h^p |^2_{H^1(0,T)} - | v_h^p |_{H^1(0,T)} \| u_h^p \|_{L^2(0,T)}
    \\ & = | u_h^p |^2_{H^1(0,T)} - \mu^{\nicefrac{1}{2}} | u_h^p |_{H^1(0,T)}  \| u_h^p \|_{L^2(0,T)}
    \\ & \ge | u_h^p |^2_{H^1(0,T)} - \frac{2T}{\pi} \mu^{\nicefrac{1}{2}} | u_h^p |^2_{H^1(0,T)}.
\end{align*}
This leads to the conclusion that~$u_h^p \equiv v_h^p \equiv 0$ if~\eqref{eq:13} is satisfied. Note that the same conclusion would be valid if instead of~$S_{h,0,\bullet}^{(p,p-1)}$ and~$S_{h,\bullet,0}^{(p,p-1)}$, we had considered generic finite subspaces of~$H^1(0,T)$ with zero initial and final conditions, respectively. We also note that the same condition~\eqref{eq:13} for the uniqueness could have been derived by employing~\cite[Theorem 3.2]{CiarletHuangZou2003}.
\end{remark}
In Propositions~\ref{prop:312} and~\ref{prop:315} below, we prove the invertibility of the matrices~$\B_h^p$ and~$\bC_h^p$. From this property, it follows that both the Schur complements~$\B_h^p + \bC_h^p (\B_h^p)^{-1} \bC_h^p$ and~$\bC_h^p + \B_h^p (\bC_h^p)^{-1} \B_h^p$ of system~\eqref{eq:10} are well defined. Through a standard block~$LDU$ factorization of the system matrix in~\eqref{eq:10}, it follows that the invertibility of any of the Schur complements is equivalent to the invertibility of the system matrix in~\eqref{eq:10}. Therefore, under Assumption~\ref{assu:1}, the linear system~\eqref{eq:10} can be solved in two different ways:
\begin{itemize}
\item write~$\bu_h^p$ from the first equation as~$\bu^p_h = -(\B^p_h)^{-1} \bC^p_h \bv_h^p$,
insert it into the second one and solve for~$\bv_h^p$, then recover~$\bu_h^p$:
\begin{equation} \label{eq:19}
    \begin{aligned}
	   \left(\B^p_h + \mu \bC_h^p (\B_h^p)^{-1} \bC_h^p\right)\bv^p_h & = \bff_h^p,
        \\ \bu^p_h & = - (\B^p_h)^{-1}\bC^p_h \bv_h^p, 
\end{aligned}
\end{equation}
\item or, alternatively, write~$\bv_h^p$ as~$\bv^p_h  = -(\bC^p_h)^{-1} \B^p_h \bu_h^p$ first, solve for~$\bu^p_h$, then recover~$\bv^p_h$:
\begin{equation} \label{eq:20}
    \begin{aligned}
	   \left(\mu \bC_h^p + \B_h^p (\bC_h^p)^{-1} \B_h^p\right)\bu^p_h & = -\bff_h^p,
        \\ \bv^p_h & = - (\bC^p_h)^{-1}\B_h^p\bu_h^p.
    \end{aligned}
\end{equation}
\end{itemize}
\begin{remark} \label{rem:P1}
For~$p=1$, the Schur complement~$\B^1_h + \mu \bC_h^1 (\B_h^1)^{-1} \bC_h^1$ is a lower triangular matrix with entries all equal to~$-\left(h+\mu/(4h)\right)$ on the diagonal. Its invertibility for all~$\mu,h>0$, and thus that of the system matrix in~\eqref{eq:10} for~$p=1$, readily follows.  
\end{remark}
In Section~\ref{sec:3} below, we show that, under Assumption~\ref{assu:1}, from an algebraic point view, both procedures are stable, in the sense that the condition numbers of all the system matrices involved (namely~$\B^p_h$,~$\bC^p_h$, and the two Schur complements) do not grow exponentially when the dimensions of the systems increase. The analysis is based on two main ingredients: the characterization of families of Toeplitz band matrices which are weakly well-conditioned~\cite{AmodioBrugnano1996}, and properties of maximal-regularity splines~\cite{FerrariFraschini2024, GaroniSpeleersEkstromRealiSerraCapizzanoHughes2019}.

\section{Conditioning of the involved matrices} \label{sec:3}
In this section, we present the main theoretical results. We start by introducing the following definition.
\begin{definition}
A family of matrices~$\{\A_n\}_{n}$, with $\A_n \in \R^{n\times n}$,
is \textit{weakly well-conditioned} if, for~$n$ sufficiently large, the matrices~$\A_n$ are invertible and their condition numbers~${\kappa(\A_n)}$ grow only algebrically in~$n$.
\end{definition}
Clearly, this definition does not depend on the chosen matrix norm~$\| \cdot \|_1, \| \cdot \|_2$, or~$\| \cdot \|_\infty$ (see also~\cite[Remark 4.4]{FerrariFraschini2024}). For the weakly well-conditioning of a family of Toeplitz band matrices, it is of crucial importance the location of the zeros of a specific polynomial associated with it. With a sequence of non-singular Toeplitz band matrices~$\{ \widetilde{\A}_n \}_{n}$,~$\widetilde{\A}_n \in \R^{n\times n}$, with structure
\begin{equation}\label{eq:21}
\widetilde{\A}_n =  
\setcounter{MaxMatrixCols}{20}
    \begin{pmatrix} 
            a_0 & \ldots & a_\ell & &  \\ 
            \vdots & \ddots & & \ddots \\
            a_{-m} & & \ddots & &  a_\ell \\
            & \ddots & & \ddots & \vdots \\
            &  & a_{-m} & \ldots & a_0 \\
    \end{pmatrix}_{n \times n} 
    \begin{aligned}
    \hspace{-0.6cm} \quad\quad\quad \text{with} \quad & a_\ell a_{-m} \ne 0, \\ & \ell,m,\ \text{and}\ \{a_i\}_{i=-m}^\ell \, \, \text{independent of}\ n,
    \end{aligned}
\end{equation}
we associate the polynomial~$q^{\A} \in \P_{m+\ell}(\R)$
\begin{equation} \label{eq:22}
    q^{\A}(z) := \sum_{i=-m}^\ell a_i z^{m+i}.
\end{equation}
We recall from~\cite{AmodioBrugnano1996} the following result.
\begin{theorem}{\cite[Theorem 3]{AmodioBrugnano1996}} \label{thm:32}
Let~$\{\widetilde{\A}_n \}_{n}$ be a family of invertible Toeplitz band matrices with structure as in~\eqref{eq:21}, and let~$q^{\A}$ in~\eqref{eq:22} be the associated polynomial. Then, the family~$\{\widetilde{\A}_n\}_{n}$ is weakly well-conditioned if it satisfies the following root property: 
\begin{equation} \label{eq:23}
    \begin{aligned}
        \text{the~polynomial~} & q^{\A} \text{has~exactly~} m \text{~roots~strictly~inside~the~unitary~complex circle~} \\ & \hspace{0.7cm}  \text{~or~exactly~$\ell$~roots~strictly~outside~it}.
    \end{aligned}
\end{equation}
Whenever~$q^{\A}$ has at least one root that is not on the boundary of the unitary complex circle, then the weakly well-conditioning of~$\{\widetilde{\A}_n\}_n$ is equivalent to~\eqref{eq:23}. \\
Furthermore, in case of weakly well-conditioning,~$\kappa_1(\widetilde{\A}_n)={\mathcal O}(n^\mu)$, where~$\mu$ is the highest multiplicity among the roots of unit modulus.
\end{theorem}
\begin{remark} \label{rem:33}
For families of Toeplitz band matrices as in~\eqref{eq:21}, the root property~\eqref{eq:23} actually implies invertibility, see~\cite[Remark~2]{AmodioBrugnano1996}.
\end{remark}
We say that a family~$\{\A_n\}_n$ is \emph{nearly Toeplitz} if~$\A_n = \widetilde{\A}_n+\mathbf{P}_n$, where~$\widetilde{\A}_n$ is a Toeplitz band matrix with structure as in~\eqref{eq:21}, and~$\mathbf{P}_n$ is a perturbation matrix with a number of nonzero entries independent of~$n$ and located only in the top left and/or bottom right corners.

Theorem~\ref{thm:32} applies to nearly-Toeplitz families of matrices~$\{\widetilde{\A}_n+\mathbf{P}_n\}_n$, where~$\widetilde{\A}_n$ are Toeplitz with structure as in~\eqref{eq:21}, and the perturbations~$\mathbf{P}_n$ are \emph{admissible} in the following sense:
\begin{itemize}
\item their nonzero entries are independent of~$n$ and are located only in top left and bottom right~$(m+\ell)\times (m+\ell)$ blocks with structure 
\begin{equation} \label{eq:24}
    \vphantom{
    \begin{matrix}
            \overbrace{XYZ}^{\mbox{$R$}} \\ \\ \\ \\ \\
            \underbrace{pqr}_{\mbox{$S$}} \\
        \end{matrix}}
    \begin{matrix}
        \coolleftbrace{m}{* \\ \vspace{0.15cm} \\ * } \vspace{0.1cm}\\ 
        \coolleftbrace{\ell}{* \\ \vspace{0.15cm} \\ *}
    \end{matrix}%
    \begin{pmatrix}
        \coolover{\ell}{*     & \ldots &     * \hspace{0.2cm}}& \coolover{m}{* & \phantom{**}  & \phantom{**}} 
        \\        \vdots &        & \vdots               &         \vdots & \ddots           & 
        \\         *     & \ldots &     *                &              * &  \ldots          &  *           
        \\         *     & \ldots &     *                &                &                  &                            
        \\               & \ddots & \vdots               &                &                  & 
        \\               &        & *                    &                &                  &  
        \end{pmatrix}, \quad  \quad  \quad  \quad  \quad %
         \vphantom{
    \begin{matrix}
            \overbrace{XYZ}^{\mbox{$R$}}\\ \\ \\ \\ \\ \\ \\ \\
            \underbrace{pqr}_{\mbox{$S$}} \\
        \end{matrix}}
    \begin{pmatrix}
         &        &                &        *  &            & 
        \\              &  &                     &   \vdots            &  \ddots          &             
        \\              &                &                  & *     & \ldots &     *                             
        \\             *  & \ldots & *   &       *     &    \ldots            &                 *  
        \\               &  \ddots & \vdots                    &   \vdots             &              &  \vdots
        \\ \coolunder{\ell}{ \phantom{**}  & \phantom{**} &     * } & \coolunder{m}{* & \ldots & *} 
        \end{pmatrix}
    \begin{matrix}
        \coolrightbrace{* \\ \vspace{0.15cm} \\ * }{\ell} \vspace{0.1cm}\\ 
        \coolrightbrace{* \\ \vspace{0.15cm} \\ * }{m} \vspace{0.1cm}
    \end{matrix},
\end{equation}
where the matrix on the left represents the top left perturbation block, and the matrix on the right the bottom right perturbation block,
\item 
the perturbed matrices~$\widetilde{\A}_n+\mathbf{P}_n$ have all entries on the outer codiagonals (the~$\ell^{\thh}$ and the~$(-m)^{\thh}$) different from zero.
\end{itemize}
More precisely, we recall from~\cite{FerrariFraschini2024} the following result.
\begin{theorem}{\cite[Theorem 4.5 and Remark A.5]{FerrariFraschini2024}} \label{thm:34}
Let~$\{\widetilde{\A}_n+\mathbf{P}_n\}_n$ be a \emph{nearly}-Toeplitz family of matrices with admissible perturbations~$\mathbf{P}_n$.
\begin{itemize}
\item[i)] If the matrices~$\widetilde{\A}_n+\mathbf{P}_n$ are invertible, then~$\{\widetilde{\A}_n+\mathbf{P}_n\}_n$ has the same conditioning behaviour as~$\{\widetilde{\A}_n\}_n$.
\item[ii)] Suppose the polynomial~$q^{\A}$ associated with the family~$\{\widetilde{\A}_n\}_n$ has exactly~$\ell$ roots strictly outside the unitary complex circle, and the perturbed matrices~$\widetilde{\A}_n+\mathbf{P}_n$ are persymmetric. Moreover, assume the invertibility of the~$\ell\times \ell$ matrix
\begin{equation} \label{eq:25}
    (\W^{-1})[m+1:m+\ell,1:m+\ell] (\mathbf{Y}_{2})^{-1} \mathbf{Y}_{1},
\end{equation}
where~$\W$ is the Casorati matrix associated with~$q^{\A}$ (see e.g.~\cite[Section 1.1]{AmodioBrugnano1996} and~\cite[Section 2.1]{LakshmikanthamTrigiante1988}),~$\mathbf{Y}_{1} \in \R^{(m+\ell) \times \ell}$ and~$\mathbf{Y}_{2} \in \R^{(m+\ell) \times (m+\ell)}$ are the following sub-blocks of the top left~$(m+\ell)\times (m+2\ell)$ block of~$\widetilde{\A}_n + \mathbf{P}_n$:
\begin{equation*}
    \begin{matrix}
        \coolleftbrace{m+\ell}{* \\ \\ \vspace{0.15cm} \\ * \\ * \\ \vspace{0.15cm} \\ *}
    \end{matrix}
    \begin{pmatrix}
        \phantom{12} \tikzmark{left1} * & \ldots & * \hspace{0.01cm} & \tikzmark{left2} * & & &
        \\ \phantom{12} \vdots & & \vdots \hspace{0.01cm} & \vdots & \hspace{-1.5cm} \ddots &
        \\\phantom{12} * & \ldots & * \hspace{0.01cm} & * & \hspace{-1cm} \ldots \hspace{0.4cm} * &
        \\\phantom{12} * & \ldots & * \hspace{0.01cm} & \hspace{0.8cm} a_{-m+\ell} & \hspace{-0.2cm} \ldots & \hspace{-0.2cm} a_\ell & 
        \\ & \ddots & \vdots & \vdots \hspace{0.01cm} & & & \hspace{-0.35cm} \ddots 
        \\ & & * \tikzmark{right1} & \phantom{12345} a_{-m+1} & \ldots & \ldots & & a_\ell \tikzmark{right2} \phantom{1}
    \end{pmatrix}. \vspace{0.1cm}
    \DrawBox[thick, black]{left1}{right1}{\textcolor{black}{\footnotesize$\mathbf{Y}_{1}$}}
    \DrawBox[thick, black]{left2}{right2}{\textcolor{black}{\footnotesize$\mathbf{Y}_{2}$}} \vspace{0.2cm}
\end{equation*}
Then, for~$n$ sufficiently large, the matrices~$\widetilde{\A}_n+\mathbf{P}_n$ are invertible.
\end{itemize}
\end{theorem}
\begin{remark} \label{rem:35}
As observed in~\cite[Remark A.3]{FerrariFraschini2024}, part~$i)$ of Theorem~\ref{thm:34} is valid also if the perturbation~$\mathbf{P}_n$ has top left and bottom right blocks of the type
\begin{equation} \label{eq:26}
    \vphantom{
    \begin{matrix}
            \overbrace{XYZ}^{\mbox{$R$}} \\ \\ \\ \\ \\
            \underbrace{pqr}_{\mbox{$S$}} \\
        \end{matrix}}
    \begin{matrix}
        \coolleftbrace{M_1}{* \\ \vspace{0.15cm} \\ * \\  
        * \\ \vspace{0.15cm} \\ *}
    \end{matrix}%
    \begin{pmatrix}
        \coolover{N_1}{*     & \ldots &     * & * & \ldots  & * }
        \\        \vdots &        & \vdots               &         \vdots &            & \vdots
        \\         *     & \ldots &     *                &              * &  \ldots          &  *           
        \\         *     & \ldots &     *                &                &                  &                            
        \\ \vdots              &  & \vdots               &                &                  & 
        \\ \coolunder{\ell}{*  & \ldots &     * } & \phantom{*} & \phantom{\ldots} & \phantom{*}
        \end{pmatrix}
        \begin{matrix}
        \coolrightbrace{* \\ \vspace{0.15cm} \\ *}{m}
        \\\phantom{\coolrightbrace{* \\ \vspace{0.15cm} \\ *}{m}}
    \end{matrix}, \quad   \quad  \quad %
         \vphantom{
    \begin{matrix}
            \overbrace{XYZ}^{\mbox{$R$}}\\ \\ \\ \\ \\ \\ \\ \\
            \underbrace{pqr}_{\mbox{$S$}} \\
        \end{matrix}}
        \begin{matrix}
        \phantom{\coolleftbrace{N_2}{* \\ \vspace{0.15cm} \\ *}}\vspace{-1.4cm}
        \coolleftbrace{\ell}{* \\ \vspace{0.15cm} \\ *} 
    \end{matrix}
    \begin{pmatrix}
         &        &                &  \coolover{m}{      *  &     \ldots       & *}
        \\              &  &                     &   \vdots            &            & \vdots            
        \\              &                &                  & *     & \ldots &     *                             
        \\             *  & \ldots & *   &       *     &    \ldots            &                 *  
        \\              \vdots &   & \vdots                    &   \vdots             &              &  \vdots
        \\ \coolunder{N_2}{ *  & \ldots &     *  & * & \ldots & *} 
        \end{pmatrix}
        \hspace{-0.1cm}
    \begin{matrix}
        \coolrightbrace{* \\ \vspace{0.15cm} \\ * \\ * \\ * 
        * \\ \vspace{0.15cm} \\ * }{M_2} \vspace{0.1cm}
    \end{matrix}
    \hspace{0.1cm}
\end{equation}
with~$M_1, N_1, M_2, N_2$ independent of~$n$, and at least three of the four blocks of the perturbed matrices~$\widetilde{\A}_n+\mathbf{P}_n$ in positions~$[1:m,$~$\ell+1:m+\ell]$,~$[m+1:m+\ell$,~$1:\ell]$,~$[n-(\ell+m-1):n-\ell$,~$n-(m-1):n]$,~$[n-(\ell-1):n$,~$n-(m+\ell-1):n-m]$ are nonsingular. This structure of the perturbations is crucial for studying the conditioning of the Schur complements in Section~\ref{sec:33} below.
\end{remark}
We define the families of matrices
\begin{equation*}
    \{\B_n^p = h\B_h^p\}_n \quad \text{and} \quad \{\bC_n^p = \bC_h^p\}_n, \quad \text{with} \quad n = N+p-1 = T/h+p-1.
\end{equation*}
The families~$\{\B_n^p\}_n$ and~$\{ \bC_n^p\}_n$ are nearly Toeplitz with admissible perturbations. In fact, thanks to properties~\emph{2.} and~\emph{3.} in Proposition~\ref{prop:23}, the families~$\{\B_n^p\}_n$ and~$\{ \bC_n^p\}_n$ are nearly Toeplitz with perturbations having nonzero blocks as in~\eqref{eq:24}, with~$m=p+1$,~$\ell=p-1$, as well as entries, independent of~$n$. 
In the following proposition, we prove that, for all~$p \in \N$, the entries of the~$(p-1)^{\thh}$ codiagonal (upper outer codiagonal) of both~$\bC_n^p$ and~$\B_n^p$ are all different from zero. With similar arguments, one can also prove that all the entries of the~$(-p-1)^{\thh}$ codiagonal (lower outer codiagonal) of these matrices are different from zero.
\begin{proposition} \label{prop:36}
For all~$p \in \N$ and 
$j=1,\ldots, n-p+1$, we have
\begin{equation*}
    \bC_h^p[j,j+p-1] = (\partial_t \phi_{j+p-1}^p,  \phi_{j-1}^p)_{L^2(0,T)} > 0, \quad \quad \B_h^p[j,j+p-1] = (\partial_t \phi_{j+p-1}^p, \partial_t \phi_{j-1}^p)_{L^2(0,T)} < 0.
\end{equation*}
\end{proposition}
\begin{proof}
We prove the statement for~$j=1,\ldots,p$. Then, due to persymmetry and Toeplitz structure of the intermediate part of the matrices, it follows that all entries on the~$(p-1)^{\thh}$ codiagonal of~$\bC_n^p$ and~$\B_n^p$ are strictly positive and negative, respectively.\\
If~$p=1$ we readily compute~$\bC_h^1[1,1] = \tfrac{1}{2}$ and~$\B_h^1[1,1] = -\tfrac{1}{h}$. Then, let suppose~$p \ge 2$. For~$j=1,\ldots,p$, the intersection of the supports of~$\phi^p_{j+p-1}$ and~$\phi_{j-1}^p$ is~$[t_{j-1},t_j]$. Therefore, our statement is equivalent to the two relations
\begin{equation*}
    (\partial_t \phi_{j+p-1}^p, \phi_{j-1}^p)_{L^2(t_{j-1},t_j)} > 0, \quad\quad  (\partial_t \phi_{j+p-1}^p, \partial_t \phi_{j-1}^p)_{L^2(t_{j-1},t_j)} < 0,
\end{equation*}
or also, after integration by parts, using that~$\phi_{j-1}^p(t_j)=\partial_t \phi_{j+p-1}^p(t_{j-1})=0$, to
\begin{equation*}
    (\partial_t \phi_{j+p-1}^p, \phi_{j-1}^p)_{L^2(t_{j-1},t_j)} > 0, \quad\quad  -(\partial_t^2 \phi_{j+p-1}^p, \phi_{j-1}^p)_{L^2(t_{j-1},t_j)} < 0.
\end{equation*}
Recalling that~$\phi_{j-1}^p>0$ in the interior of its support, which includes~$(t_{j-1},t_j)$ (see e.g.~\cite[Theorem 3.3]{BojanovHakopianSahakian1993}), we only need to prove that~$\partial_t \phi_{j+p-1}^p(t) > 0$ and~$\partial^2_t \phi_{j+p-1}^p(t) > 0$ for all~$t \in (t_{j-1},t_j)$. In terms of a rescaled translation of the cardinal spline~$\Phi_p$, we can write $\phi_{j+p-1}^p (t) = \Phi_p( \tfrac{t}{h} - j + 1)$ for~$j=1,\ldots,p$. Therefore, in order to conclude, we have to show that~$\partial_t \Phi_p(t) > 0$ and~$\partial_t^2 \Phi_p(t) > 0$ for~$t \in (0,1)$. The first inequality readily follows from the fact that~$\Phi_p$ attains exactly one maximum value in~$[0,p+1]$, and from the symmetry property~$\Phi_p(\tfrac{p+1}{2} + \cdot) = \Phi_p(\tfrac{p+1}{2} - \cdot)$ (see e.g. \cite[Theorem 4.3 (ix)]{Chui1992}). Combining these properties, we deduce that~$\Phi_p$ attains its maximum in~$\tfrac{p+1}{2}$, and that~$\partial_t \Phi_p(t) > 0$ for all~$t \in (0,\tfrac{p+1}{2})$. Finally, to show that~$\partial_t^2 \Phi_p(t) > 0$ for~$t \in (0,1)$, we employ 
the recursion formula (see e.g. \cite[Theorem (4.3) (vii)]{Chui1992}) 
\begin{equation*}
    \partial_t^2 \Phi_p(t) = \partial_t \Phi_{p-1}(t) - \partial_t \Phi_{p-1}(t-1) \quad \text{for all~} t \in \R,
\end{equation*}
and conclude using that~$\Phi_{p-1}(t-1)=0$ for~$ t \in (0,1)$.
\end{proof}

Let us introduce the families of matrices~$\{\widetilde{\B}_n^p\}_n$ and~$\{\widetilde{\bC}_n^p\}_n$, where~$\widetilde{\B}_n^p$ and~$\widetilde{\bC}_n^p$ extend the purely Toeplitz band parts of~$\B_n^p$ and~$\bC_n^p$, respectively, to~$n \times n$ matrices. Then, to prove the weakly well-conditioning of~$\{\B_n^p\}_n$ and~$\{ \bC_n^p\}_n$, we show that Theorem~\ref{thm:34} applies with~$\widetilde{\A}_n=\widetilde{\B}_n^p$ and~$\mathbf{P}_n=\B_n^p-\widetilde{\B}_n^p$, and with~$\widetilde{\A}_n=\widetilde{\bC}_n^p$ and~$\mathbf{P}_n=\bC_n^p-\widetilde{\bC}_n^p$, respectively. This is performed in Propositions~\ref{prop:312} and~\ref{prop:315} below for~$\{\B_n^p\}_n$ and~$\{ \bC_n^p\}_n$, respectively. We exploit the symmetries and skew-symmetries of~$\widetilde{\B}_n^p$ and~$\widetilde{\bC}_n^p$ in order to characterize the exact number of zeros of the associated polynomials on the boundary of the unitary complex circle, which is sufficient for weakly well-conditioning. The strategy is then to explicitly compute the restriction to the boundary of the unitary circle of these polynomials. As in~\cite{FerrariFraschini2024}, they turn out to be strictly related to the symbols of isogeometric discretizations~\cite{GaroniSpeleersEkstromRealiSerraCapizzanoHughes2019}. A similar strategy, under the assumption of invertibility~\ref{assu:1}, also applies to the Schur complements.
\begin{remark}
For a polynomial with real coefficients, a simple zero in polar coordinates on the unit circle corresponds bijectively to a simple zero in complex coordinates. Indeed, given a complex function~$f : \C \to \C$ such that~$\frac{\partial f}{\partial \bar z} = 0$, we have that~$\widetilde{z} = e^{\mi \widetilde{\theta}}$,~$\widetilde{\theta} \in [-\pi,\pi]$, is a simple zero of~$f$ if and only if~$\widetilde{\theta}$ is a simple zero of the function~$F(\theta) := f(e^{\mi \theta})$. To see this, in polar coordinates~$(\rho,\theta)$, we compute
\begin{equation*}
    0 = \frac{\partial f}{\partial \bar z} = \frac{e^{\mi \theta}}{2} \left( \frac{\partial f}{\partial \rho} + \frac{\mi}{\rho} \frac{\partial f}{\partial \theta}\right),
\end{equation*}
from which we deduce
\begin{equation*}
    \frac{\partial f}{\partial z} = \frac{e^{-\mi \theta}}{2} \left( \frac{\partial f}{\partial \rho} - \frac{\mi}{\rho} \frac{\partial f}{\partial \theta}\right) = - \frac{\mi e^{-\mi \theta}}{\rho} \frac{\partial f}{\partial \theta} = - \frac{e^{-2\mi \theta}}{\rho} F'.
\end{equation*}
\end{remark}
In Sections~\ref{sec:31} and~\ref{sec:32}, we prove the results on the invertibility and the weakly well-conditioning of~$\{\B_n^p\}_n$ and~$\{\bC_n^p\}_n$, respectively, while in Section~\ref{sec:33} those on the conditioning of their Schur's complements. An essential tool is an explicit expression of the restrictions to the unitary complex circle of the polynomials associated with the families~$\{\widetilde{\B}_n^p\}_n$ and~$\{\widetilde{\bC}_n^p\}_n$. Before presenting this result (see Proposition~\ref{prop:39}), we prove a Poisson summation formula.
\begin{lemma} \label{lem:38}
Let~$f \in H^1(\R)$ satisfy the following conditions:
\begin{equation*}
    f \text{~has compact support} \quad \quad \text{and} \quad \widehat f(\omega) = \mathcal{O}(|\omega|^{-\alpha}), \quad \alpha > 1, \quad \text{as~~} |\omega| \to \infty,
\end{equation*}
where~$\widehat f$ is the Fourier transform of~$f$ defined as~$\widehat{f}(\omega) := \int_\R e^{-\mi \omega x} f(x) \, \dd x$ for~$\omega \in \R$. Then,
\begin{equation} \label{eq:27}
    \sum_{j \in \Z} e^{-\mi j x} \left( \int_{\R} f(y+j) \overline{\partial_y f(y)} \dd y\right)  = -\mi \sum_{j \in \Z}(x+2j\pi) | \widehat{f}(x+2j\pi)|^2, \quad \text{for all~} x \in \R.
\end{equation}
\end{lemma}
\begin{proof}
Define
\begin{equation*}
    F(x) := \int_{\R} f(x+y)  \overline{\partial_y f(y)} \dd y.
\end{equation*}
Combining the property of the Fourier transform of the convolution in~\cite[Theorem 2.7]{Chui1992}, and that of the Fourier transform of the derivative in~\cite[Theorem 2.2]{Chui1992}, we obtain
\begin{equation} \label{eq:28}
    \widehat{F}(\omega) = - \mi \omega | \widehat f (\omega) |^2, \quad \omega \in \R.
\end{equation}
Recall that, for~$g \in L^1(\R)$, the classical Poisson summation formula~\cite[Theorem 2.25]{Chui1992} reads
\begin{equation} \label{eq:29}
    \sum_{j \in \Z} \, g(x+2\pi j) = \frac{1}{2\pi} \sum_{j \in \Z} \,  \widehat{g}(j) e^{\mi j x} \quad \text{for all~} x \in \R.
\end{equation}
We apply~\eqref{eq:29} with~$g = \widehat F$ to obtain
\begin{equation*}
    -\mi \sum_{j \in \Z} (x+2\pi j) |\widehat f(x + 2\pi j)|^2  =  \sum_{j \in \Z} F(-j) e^{\mi j x} =\sum_{j \in \Z} F(j) e^{-\mi j x}  \quad \text{for all~} x \in \R,
\end{equation*}
where we used~\eqref{eq:28} and the inversion formula~$\widehat{\widehat F}(x)=2\pi F(-x)$,~\cite[Equation 2.5.15]{Chui1992}. Then~\eqref{eq:27} follows, taking into account the definition of~$F$.
\end{proof}
In the next proposition, we derive explicit expressions for the polynomials~\eqref{eq:22} associated with~$\{\widetilde{\B}_n^p\}_n$ and~$\{\widetilde{\bC}_n^p\}_n$ in terms of the functions~$B_p, C_p : [-\pi,\pi] \to \R$ defined as
\begin{equation} \label{eq:30}
    B_p(\theta) := -(2-2\cos \theta)^{p+1} \sum_{j \in \Z} \frac{1}{(\theta + 2j\pi)^{2p}}, \quad C_p(\theta) := -(2-2\cos \theta)^{p+1} \sum_{j \in \Z} \frac{1}{(\theta + 2j\pi)^{2p+1}}.
\end{equation}
\begin{proposition} \label{prop:39}
For~$p \ge 1$ and~$\theta \in [-\pi,\pi]$, we have
\begin{align}
    e^{-\mi p \theta}  q^{\B^p}(e^{\mi \theta}) = -\partial_t^2 \Phi_{2p+1}(p+1) - \sum_{j=1}^{p}(e^{-\mi j \theta}+e^{\mi j \theta})  \partial_t^2 \Phi_{2p+1}(p+1-j) & = - B_p(\theta), \label{eq:31}
    \\ e^{-\mi p \theta}  q^{\bC^p}(e^{\mi \theta}) = \partial_t \Phi_{2p+1}(p+1) + \sum_{j=1}^{p}(-e^{-\mi j \theta}+e^{\mi j \theta})  \partial_t \Phi_{2p+1}(p+1-j) & = \mi C_p(\theta), \label{eq:32}
\end{align}
where~$q^{\B^p}$ and~$q^{\bC^p}$ are the polynomials associated with~$\{\widetilde{\B}_n^p\}_n$ and~$\{\widetilde{\bC}_n^p\}_n$, respectively.
\end{proposition}
\begin{proof}
The explicit expression for~$q^{\B^p}$ in~\eqref{eq:31}, has been obtained in~\cite[Proposition 5.1]{FerrariFraschini2024}. Here we prove~\eqref{eq:32}.

The Fourier transform of a cardinal spline~$\Phi_p$ is 
\begin{equation*}
	\widehat \Phi_p(\theta) = \left(\frac{1-e^{-\mi \theta}}{\mi \theta} \right)^{p+1}, \quad \bigl| \widehat \Phi_p(\theta) \bigr|^2 = \left(\frac{2-2\cos \theta}{\theta^2} \right)^{p+1}
\end{equation*}
(see~\cite[Example 3.4]{Chui1992}). It clearly satisfies the assumptions of Lemma~\ref{lem:38}. By using the formula for inner products of derivatives of cardinal B-spline~\cite[Lemma 4]{GaroniManniPelosiSerraCapizzanoSpeleers2014}, the symmetry property in~\cite[Lemma 3]{GaroniManniPelosiSerraCapizzanoSpeleers2014}, and the Poisson summation formula~\eqref{eq:27}, we obtain
\begin{align*}
    e^{-\mi p \theta}  q^{\bC^p}(e^{\mi \theta}) & = \partial_t \Phi_{2p+1}(p+1) + \sum_{j=1}^{p}(-e^{-\mi j \theta}+e^{\mi j \theta}) \partial_t \Phi_{2p+1}(p+1-j) \\ & = -\sum_{j = 0}^p e^{-\mi j \theta} \partial_t \Phi_{2p+1}(p+1-j) - \sum_{j=-p }^{-1} e^{-\mi j \theta} \partial_t \Phi_{2p+1}(p+1-j) 
    \\ & = \sum_{j \in \Z} e^{-\mi j \theta} \int_\R \Phi_p(t+j) \partial_t \Phi_p(t) \dd t = - \mi \sum_{j \in \Z} (\theta+2j\pi) | \widehat \Phi_p (\theta+2j\pi)|^2.
\end{align*}
Then~\eqref{eq:32} follows, taking into account the expression of~$| \widehat \Phi_p (\cdot)|^2$.
\end{proof}
According to Proposition~\ref{prop:23}, the purely Toeplitz band matrices in~$\{\widetilde{\B}_n^p\}_n$ and~$\{\widetilde{\bC}_n^p\}_n$ have specific symmetry structures. In view of this, we consider general matrices characterized by these specific Toeplitz structures, and establish their weakly well-conditioning by deriving the required exact number of zeros with modulus one of the associated polynomials. Then, we show that the matrices in~$\{\widetilde{\B}_n^p\}_n$ and~$\{\widetilde{\bC}_n^p\}_n$ satisfy these characterizations for all~$n$ and for all~$p \in \N$.

For later use, we state some properties of the functions~$B_p$ and~$C_p$ defined in~\eqref{eq:30} and for the auxiliary function~$M_p: [-\pi,\pi] \to \R$ defined as
\begin{equation} \label{eq:33}
    M_p(\theta) := (2-2\cos \theta)^{p+1} \sum_{j \in \Z} \frac{1}{(\theta + 2j\pi)^{2p+2}}.
\end{equation}
The function~$M_p$ is actually related to the mass matrix associated with maximal regularity splines, as established in~\cite{FerrariFraschini2024} and recalled in~\eqref{eq:52} below.
\begin{lemma}
The functions~$B_p$,~$C_p$, and~$M_p$ defined in~\eqref{eq:30} and~\eqref{eq:33} satisfy
\begin{align}
    \label{eq:34} &\lim_{\theta\to 0} B_p(\theta) = \lim_{\theta\to 0} B_p'(\theta)=0, \quad \lim_{\theta\to 0} B_p''(\theta)=0,
    \\ \label{eq:35} & \lim_{\theta\to 0} \frac{C_p(\theta)}{\theta} = -1, \quad C_p(\pi)=0, 
    \\ \label{eq:36} & \lim_{\theta\to 0} M_p(\theta) =1, \quad M_p(\pi)>0,
    \\ \label{eq:37} & C_p'(\theta)=(p+1)\frac{\sin\theta}{1-\cos\theta} C_p(\theta)+(2p+1)M_p(\theta),
    \\ \label{eq:38} & \lim_{\theta\to 0} C_p'(\theta)=-1, \quad C_p'(\pi)=(2p+1)M_p(\pi)>0.
\end{align}
\end{lemma}
\begin{proof}
The properties in~\eqref{eq:34} and~\eqref{eq:36}, and the limit in~\eqref{eq:35} are obtained immediately (see also~\cite[Corollary 5.4]{FerrariFraschini2024}). The identity~$C_p(\pi)=0$ in~\eqref{eq:35} follows from taking~$\theta=\pi$ in 
\begin{equation} \label{eq:39}
\begin{aligned}
    \sum_{j \in \Z} \frac{1}{( \theta + 2j\pi)^{2p+1}} & =  \sum_{j=0}^{\infty} \frac{1}{(\theta + 2j\pi)^{2p+1}} - \sum_{j=1}^{\infty} \frac{1}{(2j\pi - \theta)^{2p+1}} 
    \\ & = \sum_{j=0}^{\infty} \left(\frac{1}{(\theta + 2j\pi)^{2p+1}} - \frac{1}{(2(j+1)\pi -  \theta)^{2p+1}}\right).
\end{aligned}
\end{equation} 
The expression in~\eqref{eq:37} is readily obtained from the definitions of~$C_p$ and~$M_p$. Finally, the properties in~\eqref{eq:38} are obtained combining~\eqref{eq:37} with~\eqref{eq:35} and~\eqref{eq:36}.
\end{proof}

\subsection{Conditioning of Toeplitz band matrices with symmetry} \label{sec:31}
We begin by considering matrices with symmetries with respect to the first lower co-diagonal. Fix~$p \in \N$ and~$\{k_j\}_{j=0}^p \subset \R$, and consider the family of Toeplitz band matrices denoted by~$\{\K_n^p\}_n$ with the following structure:
\begin{equation} \label{eq:40}
	\K_n^p =  
	\setcounter{MaxMatrixCols}{20}
	\begin{pmatrix} 
		k_{p-1} & k_{p-2} & \dots & k_1 & k_0 & & &
		\\  k_p & \ & \ & \  & \ & \ddots
		\\ k_{p-1} & \ & \ & \ & \ & \ & \ddots
		\\ \vdots & \ & \ & \ & \ & \ & & k_0
		\\ k_1 & \ & \ & \ & \ & \ & & k_1
		\\ k_0 & \ & \ & \ & \ & \ & & \vdots
		\\ \ & \ddots & \ & \ & \ & \ & & k_{p-2}
		\\ \ & \ & k_0 & k_1 & \dots & k_{p-1} & k_p & k_{p-1}
	\end{pmatrix}_{n \times n} \hspace{-0.6cm}.
\end{equation}
These matrices have already been studied in~\cite{FerrariFraschini2024}. We recall here the main result.
\begin{lemma}{\cite[Lemma 4.4.]{FerrariFraschini2024}} \label{lemma311}
The family of matrices~$\{\K_n^p\}_n$ as in~\eqref{eq:40} is weakly well-conditioned if the associated polynomial
\begin{equation*}
	q^{\K^p}(z) = k_0 + k_1 z + \ldots +k_{p-1}z^{p-1} + k_p z^p +k_{p-1}z^{p+1}\ldots + k_1 z^{2p-1} + k_0 z^{2p}
\end{equation*}
has exactly two zeros of unit modulus. If~$q^{\K^p}$ has at least one root that is not on the boundary of the unitary complex circle, then this requirement is not only sufficient but also necessary.
\end{lemma}
From the previous lemma, we have a precise information on the conditioning of the family of matrices~$\{\widetilde{\B}_n^p\}_n$. 
\begin{proposition} \label{prop:312}
For all~$p \in \N$, the family of matrices~$\{\B_n^p\}_n$ as in~\eqref{eq:7} is weakly well-conditioned. In particular,~$\kappa_1(\B_n^p) = \mathcal{O}(n^2)$.
\end{proposition}
\begin{proof}
First, we prove that each matrix~$\B_n^p$ is invertible. Assuming that~$\B_n^p \bv_h^p = \boldsymbol{0}$ for some~$\bv_h^p = [v_j^p]_{j=1}^{N+p-1}  \in \R^{N+p-1}$, we need to prove that~$\bv_h^p = \boldsymbol{0}$. Define~$v_h^p(t) := \sum_{j=1}^{N+p-1} v_j^p \phi_j^p(t) \in S_{h,0,\bullet}^{(p,p-1)}(0,T)$. Our hypothesis is equivalent to
\begin{equation} \label{eq:41}
    (\partial_t v_h^p, \partial_t \omega_h^p)_{L^2(0,T)} = 0, \quad \text{for all~} \omega_h^p \in S_{h,\bullet,0}^{(p,p-1)}(0,T).
\end{equation}
We test with the function~$\omega_h^p(t) = v_h^p(t) - v_h^p(T)  \in S_{h,\bullet,0}^{(p,p-1)}(0,T)$ in~\eqref{eq:41} and obtain
\begin{equation*}
    (\partial_t v_h^p, \partial_t (v_h^p - v_h^p(T)))_{L^2(0,T)} = | v_h^p |^2_{H^1(0,T)} =0.
\end{equation*}
Since~$v_h^p(0)=0$, we conclude that~$v_h^p \equiv 0$, and therefore~$\bv_h^p = \mathbf{0}$. This proves the invertibilty of the matrices~$\B_n^p$.

We have already observed that, for all~$p \in \N$, $\{\B_n^p\}_n$ is nearly Toeplitz with admissible perturbations. Then, owing to Theorem~\ref{thm:34}, part~\emph{i)}, in order to conclude the proof, it is enough to study the conditioning of the family of purely Toeplitz band matrices~$\{\widetilde{\B}_n^p\}_n$ by applying~Theorem~\ref{thm:32}. The expression of the restriction to the boundary of the complex unit circle of the polynomial associated with~$\{\widetilde{\B}_n^p\}_n$ is obtained explicitly in Proposition~\ref{prop:39}. It is~$q^{\B^p}(e^{\mi \theta}) = -e^{\mi p \theta} B_p(\theta)$, with~$B_p$ as in~\eqref{eq:30}. Thanks to Lemma~\ref{lemma311}, we only need to show that~$B_p : [-\pi,\pi] \to \R$ has exactly two zeros with unit modulus. From~\eqref{eq:34}, we have that~$1$ is a zero of multiplicity~$2$ of~$q^{\B^p}$. To show that there are not other zeros, it suffices to note that~$B_p(\theta) < 0$ for all~$\theta\in (0,\pi]$, then, as~$B_p$ is an even function,~$B_p(\theta) < 0$ for all~$\theta\in [-\pi,0)\cup(0,\pi]$. 
\end{proof}

\subsection{Conditioning of Toeplitz band matrices with skew-symmetry} \label{sec:32}
Let us now consider a family~$\{\K_n^p\}_n$ of Toeplitz band matrices with skew-symmetry with respect to the first lower co-diagonal:
\begin{equation} \label{eq:42}
	\K_n^p =  
	\setcounter{MaxMatrixCols}{20}
	\begin{pmatrix} 
		-k_{p-1} & -k_{p-2} & \dots & -k_1 & -k_0 & & &
		\\  0 & \ & \ & \  & \ & \ddots
		\\ k_{p-1} & \ & \ & \ & \ & \ & \ddots
		\\ \vdots & \ & \ & \ & \ & \ & & -k_0
		\\ k_1 & \ & \ & \ & \ & \ & & -k_1
		\\ k_0 & \ & \ & \ & \ & \ & & \vdots
		\\ \ & \ddots & \ & \ & \ & \ & & -k_{p-2}
		\\ \ & \ & k_0 & k_1 & \dots & k_{p-1} & 0 & -k_{p-1}
	\end{pmatrix}_{n \times n} \hspace{-0.6cm}.
\end{equation}
\begin{lemma} \label{lem:313}
The family of matrices~$\{\K_n^p\}_n$ as in~\eqref{eq:42} is weakly well-conditioned if the associated polynomial
\begin{equation*}
	q^{\K^p}(z) = k_0 + k_1 z + \ldots + k_{p-1} z^{p-1} - k_{p-1} z^{p+1} - \ldots - k_1 z^{2p-1} - k_0 z^{2p}.
\end{equation*}
has no zeros of unit modulus except~$\pm 1$.
\end{lemma}
\begin{proof}
These matrices align with the notation established in Theorem~\ref{thm:32}, with~$m=p+1$ and~$\ell=p-1$. Therefore, it is sufficient for the weakly well-conditioning that the polynomial~$q^{\K^p}$ has~$p-1$ zeros of modulus strictly larger than one, or~$p+1$ with modulus strictly smaller than one. Note that~$q^{\K^p}(\pm 1) =0$, and that if~$\xi$ is a root of~$q^{\K^p}$ then~$\xi^{-1}$ is also a root. Then, we expect the same number of zeros of modulus strictly smaller than one and strictly larger than one. The only possibility for the weakly well-conditioning is that there are exactly~$p-1$ zeros with modulus strictly larger than one. Consequently, the family of matrices~$\{\K_n^p\}_n$ is weakly well-conditioned if and only if~$q^{\K^p}$ has~$2(p-1)$ zeros with modulus different from one, in addition to the zeros~$\pm 1$.
\end{proof}
\begin{remark} \label{rem:314}
According to Theorem~\ref{thm:32}, whenever~$q^{\K^p}$ has at least one zero of modulus different from one, the weakly well-conditioning of~$\{\K^p_n\}_n$ is actually equivalent to having exactly two zeros of unit modulus in Lemma~\ref{lemma311}, and no zeros of unit modulus except~$\pm 1$ in Lemma~\ref{lem:313}.
\end{remark}
We now study the family of matrices~$\{\bC_n^p\}_n$. For~$\{\B_n^p\}_n$, we proved the invertibility of the matrices~$\B_n^p$ by using variational arguments. Here, for~$n$ sufficiently large, we prove the invertibility of the matrices~$\bC_n^p$ by applying Theorem~\ref{thm:34}, part~\emph{ii)}. The argument we use involves a numerical verification that must be performed at each spline degree~$p$. We expect the invertibility of the matrices~$\bC_n^p$ to be true for all~$p \in \N$, but due to stability issues in our numerical verification, we can only establish invertibility up to~$p=\numC$.
\begin{proposition} \label{prop:315}
For~$p = 1,\ldots,\numC$, the matrices~$\bC_n^p$ as in~\eqref{eq:7} are invertible for~$n$ sufficiently large. Moreover, for these values of~$p$, the family of matrices~$\{\bC_n^p\}_n$ is weakly well-conditioned, in particular,~$\kappa_1(\bC_n^p) = \mathcal{O}(n)$.
\end{proposition}
\begin{proof}
We split the proof into three steps.

{\bf Step~1:} 
We study the polynomial~$q^{\bC^p}$ associated with the family of matrices~$\{\widetilde{\bC}_n^p\}_n$. From Proposition~\ref{prop:39}, the restriction of~~$q^{\bC^p}$ to the boundary of the complex unit circle is~$q^{\bC^p}(e^{\mi \theta}) = \mi e^{\mi p \theta} C_p(\theta)$, with~$C_p$ as in~\eqref{eq:30}. We get~$q^{\bC^p}(\pm 1)=0$ from~\eqref{eq:35} and, from~\eqref{eq:38}, it follows that these zeros are simple. We show that no other zeros are present. To this aim, we claim that 
\begin{equation*}
	C_p(\theta) < 0 \quad \text{for all~} \theta\in (0,\pi), \quad C_p(\theta) > 0 \quad \text{for all~} \theta\in (-\pi,0).
\end{equation*}
As the function~$C_p : [-\pi,\pi] \to \R$ is odd with respect to~$\theta=0$, it is enough to prove that, for~$ \theta \in (0,\pi)$,~$C_p(\theta)<0$. This is equivalent to
\begin{equation} \label{eq:43}
    \sum_{j \in \Z} \frac{1}{(\theta + 2j\pi)^{2p+1}} > 0, \qquad \theta \in (0,\pi),
\end{equation}
which follows from~\eqref{eq:39} and the observation that each term of the last sum in~\eqref{eq:39} is positive since, for all~$j \ge 0$, 
\begin{equation*}
    \frac{1}{(\theta + 2j\pi)^{2p+1}} - \frac{1}{(2(j+1)\pi - \theta)^{2p+1}} > 0 \iff \theta < \pi.
\end{equation*}
As in the proof of Lemma~\ref{lem:313}, we deduce that~$q^{\bC^p}$ has he same number of zeros of modulus strictly smaller than one and strictly lager than one, from which we also conclude that~$q^{\bC^p}$ has has exactly~$p-1$ roots strictly outside the unitary complex circle.

{\bf Step~2:} 
We establish the invertibility of the matrices~$\bC_n^p$ by applying Theorem~\ref{thm:34}, part~\emph{ii)}. 
In order to do so, we have numerically verified 
the invertibility of the matrix defined in~\eqref{eq:25} associated with the family~$\{\bC_n^p\}_n$ up to~$p=\numC$ with the code~\cite[Folder \texttt{verifications}]{XTIgAWaves}\footnote{Due to the severe ill-conditioning of the Casorati matrix, this verification requires the availability of the entries of the matrices $\bC_n^p$
with extremely high machine precision. At the moment, we have generated these matrices for $p$ up to $\numC$ using the GeoPDEs toolbox~\cite{defalco} combined with Matlab's \texttt{vpa} function with a precision of 1000 digits. In \cite[Folder \texttt{verifications}]{XTIgAWaves} the code is available for verification, along with the matrices.}.
We have already observed that, for all~$p\in\N$, $\{\bC_n^p\}_n$ is nearly Toeplitz with admissible perturbations. Moreover, according to Proposition~\ref{prop:23}, each~$\bC_n^p$ is persymmetric with components independent of~$n$, and we have proven in Step~1 that~$q^{\bC^p}$ has exactly~$p-1$ roots strictly outside the unitary complex circle. Then, for~$p = 1,\ldots,\numC$, Theorem~\ref{thm:34}, part~\emph{ii)} applies with~$\widetilde{\A}_n^p=\widetilde{\bC}_n^p$ and~$\mathbf{P}_n=\bC_n^p-\widetilde{\bC}_n^p$, giving the invertibilty of~$\bC_n^p$, for~$n$ sufficiently large.

{\bf Step~3:} 
As we have proven in Step~1 that~$q^{\bC^p}$ has no zero of unit modulus except~$\pm1$, Lemma~\ref{lem:313} applies to the family of purely Toeplitz band matrices~$\{\widetilde{\bC}_n^p\}_n$. Then Step~2, Theorem~\ref{thm:34}, part~\emph{i)}, and Theorem~\ref{thm:32} allow us to conclude the weakly well-conditioning of~$\{\bC_n^p\}_n$ with the stated rate, up to~$p=\numC$.
\end{proof}

\subsection{Conditioning of the Schur complements} \label{sec:33}
Recalling the two possibilities of solving the linear system described in~\eqref{eq:19} and~\eqref{eq:20}, all that remains is to study the Schur complements. Thus, under the assumption of invertibility~\ref{assu:1}, we study the behaviour of the conditioning of the families of the (scaled) Schur complements
\begin{equation*}
	\bigl\{\rho \bC_n^p (\B_n^p)^{-1} \bC_n^p + \B^p_n\bigr\}_n \quad \text{and} \quad \bigl\{\rho \bC_n^p + \B_n^p (\bC_n^p)^{-1} \B_n^p\bigr\}_n, \quad \text{with~$\rho:=\mu h^2$.}
\end{equation*}
Due to multiplication by inverses, these families of matrices do not have a nearly-Toeplitz band structure. We define
\begin{equation*}
    \G_n^p(\rho):= \rho (\bC_n^p)^2 + (\B_n^p)^2,
\end{equation*}
and denote by~$\widetilde{\G}_n^p(\rho)$ the extension of the purely Toeplitz part of~$\G_n^p(\rho)$ to an~$n\times n$ matrix. The family~$\{\G_n^p(\rho)\}_n$ is nearly Toeplitz, as~$\{(\B_n^p)^2\}_n$ and~$\{(\bC_n^p)^2\}_n$ are nearly Toeplitz with~$\ell = 2p-2$ and~$m=2p+2$. From Remark~\ref{rem:33}, if~$\{\widetilde{\G}_n^p(\rho)\}_n$ satisfies the root property~\eqref{eq:23} in Theorem~\ref{thm:32}, then, for~$n$ sufficiently large, the matrices~$\widetilde{\G}_n^p(\rho)$ are invertible. We state the following property, which we verify below for~$p=1,\ldots,\numSchur$.
\begin{property} \label{prop:316}
Under Assumption~\ref{assu:1}, for~$\rho >0$, the family of the Schur complements~$\{\rho \bC_n^p (\B_n^p)^{-1} \bC_n^p + \B_n^p\}_n$ is weakly well-conditioned if and only if the family~$\{\widetilde{\G}_n^p(\rho)\}_n$ is weakly well-conditioned.
\end{property} 
To justify this property, we compute
\begin{align*}
    \rho \bC_n^p (\B_n^p)^{-1} \bC_n^p + \B_n^p & = (\B_n^p)^{-1} \bigl(\rho \B_n^p \bC_n^p (\B_n^p)^{-1} \bC_n^p + (\B_n^p)^2 \bigr)
    \\ & = (\B_n^p)^{-1} \bigl(\rho \D_n^p (\B_n^p)^{-1} \bC_n^p + \rho (\bC_n^p)^2 + (\B_n^p)^2 \bigr)
    \\ 
    & = \rho(\B_n^p)^{-1} \bigl( \D_n^p (\B_n^p)^{-1} \bC_n^p + \rho^{-1}\G_n^p(\rho) \bigr),
\end{align*}
where~$\D_n^p := \B_n^p \bC_n^p-\bC_n^p \B_n^p$. By direct calculations, exploiting the Toeplitz band structures of~$\B_n^p$ and~$\bC_n^p$ and their persymmetry (see Proposition~\ref{prop:23}), one shows that, for all~$n$,~$\D_n$ is a matrix with zero entries except for two blocks in the top left and bottom right corners of size~$(2p+1) \times (2p-2)$ and~$(2p-2) \times (2p+1)$, respectively, with entries not depending on~$n$, and such that the one is  minus the transpose of the other, i.e.,
\begin{equation*}
	\sbox4{\text{$\begin{matrix}& 0 & \\ & & \ddots \\ & &\end{matrix}$}}
	\sbox1{\text{$\begin{matrix}& &  \\ 0 & & \\ & \ddots &\end{matrix}$}}
	\sbox2{\text{$\begin{matrix}&\ddots &  \\ & & 0 \\ & &\end{matrix}$}}
	\sbox3{\text{$\begin{matrix} 0 & & \ddots  \\ & \ddots & \\ \ddots & & 0\end{matrix}$}}
	\sbox5{\text{$\begin{matrix}  & \, \, \, \,   \ddots &  \\ & & 0 \\ & &\end{matrix}$}}
	\sbox6{\text{$\begin{matrix}  & &  \\ \ddots & & \\ \vspace{-0.3cm}  & 0 &\end{matrix}$}}
	\D_n^p = \left(
		\begin{array}{c|c|c}
			\makebox[\wd0]{\text{$\large \mathbf{Z}_p$}}&\usebox{1}&\makebox[\wd0]{\text{$\large \phantom{A_1}$}}\\
			\hline
			\usebox{4} & \usebox{3} & \usebox{6}\\
			\hline
			\makebox[\wd0]{\text{$\large \phantom{A_1}$}}&\usebox{5}&\makebox[\wd0]{\text{$\large -\mathbf{Z}_p^\top$}}\\
		\end{array}
				\right)_{n \times n} \hspace{-0.5cm}.
\end{equation*}
We study the conditioning behaviour of the family of matrices
\begin{equation} \label{eq:44}
    \{\D_n^p (\B_n^p)^{-1} \bC_n^p + \rho^{-1} \G_n^p(\rho) \bigr\}_n.
\end{equation}
Firstly, we have numerically verified (with the code available in~\cite[Folder \texttt{verifications}]{XTIgAWaves}) that, for all~$p=1,\dots,\numSchur$, there is no~$\rho>0$ such that more than one of the four blocks associated with the family in~\eqref{eq:44}, with size and position specified in Remark~\ref{rem:35}, is singular. The key fact is now that~$\D_n^p (\B_n^p)^{-1} \bC_n^p$ is an admissible perturbation in the sense of Theorem~\ref{thm:34} and Remark~\ref{rem:35}, with nonzero blocks as in~\eqref{eq:26}. More precisely, it has entries smaller then a given tolerance~$\eps>0$, except for two blocks in the top left and bottom right corners of size~$(2p+1) \times N_1(p,\eps)$ and~$(2p-2) \times N_2(p,\eps)$, respectively, with~$N_1(p,\eps)$ and~$N_2(p,\eps)$, as well as the entries of these two blocks, independent of~$n$. If we show this, then Assumption~\ref{assu:1}, part~\emph{i)} of Theorem~\ref{thm:34}, and Remark~\ref{rem:35} imply Property~\ref{prop:316}. 

Due to the structure of~$\D_n^p$ and~$\bC_n^p$, the matrix~$\D_n^p (\B_n^p)^{-1} \bC_n^p$ has nonzero entries only in the first~$2p+1$ and in the last~$2p-2$ rows. Moreover, as a consequence of~\cite[Remark A.3]{FerrariFraschini2024} and~\cite[Theorem 4]{AmodioBrugnano1996}, the entries of the matrix~$(\B_n^p)^{-1}$ are bounded by a constant independent of~$n$.

In the following lemma, we prove a componentwise bound for the top right block of size~$(n - p -1) \times (n-p-1)$ of the matrix~$(\B_n^p)^{-1}$, which characterizes the decay to zero of its entries as their row and column indices approach~$1$ and~$n$, respectively.
\begin{lemma} \label{lem:317}
For~$n \in \N$ and~$\gamma \in \R$, define the matrices
\begin{align*}
    \mathbf{F}_n := \begin{pmatrix}
    0 &
    \\ 1 & 0
    \\ 1 & 1 & 0
    \\ \vdots &\ddots &\ddots & \ddots
    \\ 1 &\ldots& 1 & 1 & 0
    \end{pmatrix}_{n \times n} \hspace{-0.65cm},
    \quad \quad \boldsymbol{\Delta}_n(\gamma) :=
    \begin{pmatrix}
    0 &
    \\ \gamma & 0
    \\ 2\gamma^2 & \gamma & 0
    \\ \vdots &\ddots &\ddots & \ddots
    \\ \hspace{-0.06cm} (n-1) \gamma^{n-1} &\ldots& 2\gamma^2 & \gamma & 0
    \end{pmatrix}_{n \times n} \hspace{-0.65cm},
\end{align*}
and~$\mathbf{I}_n$ the identity matrix of size~$n$. Then, for all~$p \in \N$ and~$n$ sufficiently large, the top right block of size~$(n - p -1) \times (n-p-1)$ of the matrix~$(\B_n^p)^{-1}$ satisfies the following componentwise bound:
\begin{equation*}
    |(\B_n^p)^{-1}[\ell,j]| \le c_p \big( \mathbf{I}_n + \mathbf{F}_n + \boldsymbol{\Delta}_n^\top(\gamma_p)\big)[\ell,j] \quad \quad \text{for~} \quad \ell = 1,\ldots,n-p-1, \quad j=p+2,\ldots,n,
\end{equation*}
for constants~$c_p>0$ and~$0< \gamma_p<1$ independent of~$n$.
\end{lemma}
\begin{proof}
The proof combines~\cite[Remark A.3]{FerrariFraschini2024},~\cite[Lemma 2]{AmodioBrugnano1996}, the proof of~\cite[Theorem 4.3]{FerrariFraschini2024} and~\cite[Theorem 4]{AmodioBrugnano1996}, with the characterization of the zeros of the polynomial associated with the family of matrices~$\{\widetilde{\B}_n^p\}_n$ obtained in Proposition~\ref{prop:312} (see also~\cite[Lemma 4.4]{FerrariFraschini2024}). Indeed, one can prove that
\begin{equation} \label{eq:45}
    |(\B_n^p)^{-1}| \le | (\mathbf{U}_n^p)^{-1} | | (\L_n^p)^{-1} | + | \H_n^p | + | \J_n (\H_n^p)^\top \J_n |,
\end{equation}
with~$\L_n^p$ and~$\mathbf{U}_n^p$ lower and upper triangular matrices, respectively, satisfying  the following bounds:
\begin{equation*}
        | (\L_n^p)^{-1} | \le \alpha_p (\mathbf{I}_n + \mathbf{F}_n) \quad \quad | (\mathbf{U}_n^p)^{-1} | \le \beta_p (\mathbf{I}_n + \boldsymbol{\Delta}_n^\top(\gamma_p)),
\end{equation*}
for some positive numbers~$\alpha_p$,~$\beta_p$ independent of~$n$. Furthermore, the matrix~$\H_n^p$ is such that the entries of its top right block of size~$(n - p -1) \times (n-p-1)$ satisfies the following componentwise bound:
\begin{equation*}
    | \H_n^p[\ell,j] | \le \omega_p (\mathbf{I}_n + \mathbf{F}_n + \boldsymbol{\Delta}_n^\top(\gamma_p))[\ell,j] \qquad \quad \text{for~} \quad \ell = 1,\ldots,n-p-1, \quad j=p+2,\ldots,n,
\end{equation*}
for~$\omega_p$ > 0 independent of~$n$, and the matrix~$\J_n$ is defined as 
\begin{equation*}
    \J_n := \begin{pmatrix}
    0 & \\
    1 & 0  \\
    0 & 1 & 0 \\
    \vdots &\ddots &\ddots & \ddots \\
    0 &\ldots& 0 & 1 & 0
    \end{pmatrix}_{n \times n} \hspace{-0.6cm}.
\end{equation*}
Note that~$\J_n (\H_n^p)^\top \J_n$ is the flip-transpose of~$\H_n^p$. From~\eqref{eq:45},
using that 
\begin{equation*}
    \boldsymbol{\Delta}_n^\top(\gamma_p) \mathbf{F}_n \le \frac{\gamma_p}{1-\gamma_p} (\mathbf{I}_n + \mathbf{F}_n + \boldsymbol{\Delta}_n^\top(\gamma_p)) \quad \quad \text{and} \quad \quad \boldsymbol{\Delta}_n^\top(\gamma_p) = \J_n \boldsymbol{\Delta}_n(\gamma_p) \J_n
\end{equation*}
one concludes. 
\end{proof}
From Lemma~\ref{lem:317} and the previous observations, we deduce that, given a tolerance~$\eps>0$, there exists~$N_1(p,\eps)$ independent of~$n$ such that, in the first~$2p+1$ rows of~$\D_n^p (\B_n^p)^{-1} \bC_n^p$, only the first~$N_1(p,\eps)$ columns may contain entries with magnitude larger than~$\eps$. We verified numerically that there exists also~$N_2(p,\eps)$ independent of~$n$ such that, in the last~$2p-2$ rows of~$\D_n^p (\B_n^p)^{-1} \bC_n^p$, only the last~$N_2(p,\eps)$ columns may contain entries with magnitude larger than~$\eps$. Fixing~$\eps=10^{-13}$, we obtained for~$N_2(p)$ the values reported in Table~\ref{tab:1} (for each reported~$p$, the same values of~$N_2(p,\eps)$ have been obtained for~$n=2^7+p-1$ and $n=2^8+p-1$). For completeness we also report the values obtained for~$N_1(p,\eps)$. Again, the code is available in~\cite[Folder \texttt{verifications}]{XTIgAWaves}.  With this, we have shown that~$\D_n^p (\B_n^p)^{-1} \bC_n^p$ have entries with magnitude larger than~$\eps$ only in the top left and bottom right blocks of size~$(2p+1) \times N_1(p,\eps)$ and~$(2p-2) \times N_2(p,\eps)$, respectively.

\begin{table}[h!]
  \centering
  \renewcommand{\arraystretch}{1.5} 
  \vspace{0.25cm}
  \begin{tabular}{|c|c|c|c|c|c|c|c|c|c|c|c|c|c|c|c|c|c|}\hline
  ~$p$ & 2 & 3 & 4 & 5 & 6 & 7 & 8 & 9 & 10 & 11 & 12 & 13 & 14 & 15 & 16 & 17 \\\hline
  ~$N_1(p,\eps)$  & 20 & 31 & 39 & 47 & 55 & 61 & 68 & 73 & 79 & 83 & 87 & 91 & 94 & 97 & 99 & 101 \\
   \hline 
  ~$N_2(p,\eps)$  & 23 & 34 & 44 & 53 & 61 & 69 & 76 & 82 & 88 & 93 & 98 & 102 & 106 & 110 & 112 & 115 \\
  \hline
  \end{tabular}
  \vspace{0.15cm}
  \caption{Values of~$N_1(p,\eps)$ and~$N_2(p,\eps)$ for~$p=2,\ldots,\numSchur$ and~$\eps=10^{-13}$ (tested with~$n=2^7+p-1$ and $n=2^8+p-1$). \vspace{0.1cm} By fitting these data, we have obtained~$N_1(p,10^{-13})\sim 13.8\,p^{0.74}$ and~$N_2(p,10^{-13})\sim 15.5\,p^{0.74}$.}
  \label{tab:1}
\end{table}

We checked numerically that the entries of these two blocks are independent of~$n$. For a given~$p$, we computed the difference of the corresponding blocks for~$n = 2^7+p-1$ and~$n=2^8+p-1$. These tests have been performed for~$p=2,\ldots,\numSchur$ (see~\cite[Folder \texttt{verifications}]{XTIgAWaves}). In all these tests, the norm difference resulted to be smaller than~$10^{-13}$.
This completes the justification of Property~\ref{prop:316}.
\begin{remark} \label{rem:318}
Assume that the matrices~$\bC_n^p$ are invertible, which we have verified for~$p=1,\ldots,\numC$ in Proposition~\ref{prop:315}, and that Property~\ref{prop:316} is valid for the families of the Schur complements~$\{\rho \bC_n^p (\B_n^p)^{-1} \bC_n^p + \B_n^p\}_n$, which we have verified above for $p=1,\ldots,17$. Then also the families~$\{\rho \bC_n^p + \B_n^p (\bC_n^p)^{-1} \B_n^p\}_n$ are weakly well-conditioned if and only if~$\{\widetilde{\G}_n^p(\rho)\}_n$ is weakly well-conditioned. This follows from the identity
\begin{equation*}
    \rho \bC_n^p (\B_n^p)^{-1} \bC_n^p + \B_n^p = \bC_n^p(\B_n^p)^{-1} (\rho \bC_n^p + \B_n^p (\bC_n^p)^{-1} \B_n^p).
\end{equation*}
\end{remark}
Based on Property~\ref{prop:316}, we restrict our attention to the family of Toeplitz band matrices~$\{\widetilde{\G}_n^p(\rho)\}_n$. The matrices~$\widetilde{\G}_n^p(\rho)$ are symmetric with respect to the first lower co-diagonal, and we are able to characterize the weakly well-conditioning of~$\{\widetilde{\G}_n^p(\rho)\}_n$ in terms of the number of zeros of unit modulus of the associated polynomial.
\begin{lemma} \label{lem:319} 
The family of Toeplitz band matrices~$\{\widetilde{\G}_n^p(\rho)\}_n$ is weakly well-conditioned if and only if the associated polynomial~$q^{\G^p(\rho)}$ has exactly four zeros of unit modulus.
\end{lemma}
\begin{proof}
Due to the symmetry property of each~$\widetilde{\G}_n^p(\rho)$, the polynomial~$q^{\G^p(\rho)}$, which has degree~$4p$, has the same number~$s$ of zeros strictly less than one and strictly greater than one. According to the notation in~\eqref{eq:21}, for this family,~$m=2p+2$ and~$\ell=2p-2$. From Theorem~\ref{thm:32}, the weakly well-conditioning is equivalent to have either~$2p-2$ smaller than one or~$2p+2$ greater than one. The only  possibility allowed is to have~$2p-2$ zeros of modulus greater than one, four of modulus exactly one and~$2p-2$ with modulus greater than one.
\end{proof}
Before proving the weakly well-conditioning of the family~$\{\widetilde{\G}_n^p(\rho)\}_n$, we state the following lemma on the polynomial associated with the product of Toeplitz matrices. Its proof readily follows from the definition~\eqref{eq:22}.
\begin{lemma} \label{lem:320}
Let~$\{\E_n\}_n, \{\F_n\}_n$ be families of Toeplitz band matrices. Then, except for blocks in top left and bottom right corners, the~$\{\E_n \F_n\}_n$ is also a family of Toeplitz band matrices, and the following statement holds
\begin{equation*}
    q^{\E}(z) q^{\F}(z) = q^{\E \F}(z)
\end{equation*}
with~$q^{\E \F}$ the polynomial associated with the purely Toeplitz part of the family~$\{\E_n \F_n\}_n$.
\end{lemma}
\begin{proposition} \label{prop:321}
For all~$p \in \N$ and for all~$\rho>0$, the family of matrices~$\{\widetilde{\G}_n^p(\rho)\}_n$ is weakly well-conditioned.
\end{proposition}
\begin{proof}
Owing to Lemma~\ref{lem:319}, we restrict our attention to the boundary of the unitary circle. Thanks to Lemma~\ref{lem:320} and Proposition~\ref{prop:39}, we compute
\begin{equation} \label{eq:46}
	q^{\G^p(\rho)}(e^{\mi \theta}) = \rho (q^{\bC^p}(e^{\mi \theta}))^2 + (q^{\B^p}(e^{\mi \theta}))^2 = e^{2 \mi p \theta} \left(-\rho C_p^2(\theta) + B_p^2(\theta)\right).
\end{equation}
According to Lemma~\ref{lem:319} and~\eqref{eq:46}, we need to verify that the function~$G_p : [-\pi,\pi] \times \R^+ \to \R$, defined as 
\begin{equation*}
    G_p(\theta, \rho) := -\rho C_p^2(\theta) + B_p^2(\theta),
\end{equation*}
with~$C_p$ and~$B_p$ as in~\eqref{eq:30}, has exactly four zeros in~$\theta$ for all~$\rho >0$. Note that for all~$p \in \N$ the function~$G_p$ is symmetric in~$\theta$ for all~$\rho>0$ with respect to~$\theta=0$ and, recalling~\eqref{eq:34} and~\eqref{eq:35}, it holds true that
\begin{equation*} 
    \lim_{\theta \to 0} G_p(\theta,\rho) = - \rho \lim_{\theta \to 0} C_p^2(\theta) +  \lim_{\theta \to 0} B_p^2(\theta)  = 0.
\end{equation*}
Similarly, for the first derivative, we have 
\begin{equation*} 
    \lim_{\theta \to 0} \partial_{\theta} G_p(\theta,\rho) = - 2\rho \lim_{\theta \to 0} C_p(\theta) C_p'(\theta) + 2 \lim_{\theta \to 0} B_p(\theta) B_p'(\theta)  = 0.
\end{equation*}
However, for the second derivative we obtain
\begin{align*} 
    \lim_{\theta \to 0} \partial^2_{\theta} G_p(\theta,\rho) & = - 2\rho \lim_{\theta \to 0}  (C_p'(\theta))^2 - 2\rho \lim_{\theta \to 0} C_p(\theta) C_p^{''}(\theta) + 2 \lim_{\theta \to 0} (B_p'(\theta))^2 + 2 \lim_{\theta \to 0} B_p(\theta) B_p^{''}(\theta)
    \\ & = - 2\rho \bigl(\lim_{\theta \to 0}  C_p'(\theta)\bigr)^2 \ne 0,
\end{align*}
recalling~\eqref{eq:38}. Therefore, for any~$\rho>0$,~$G_p(\theta,\rho)$ has a zero of multiplicity exactly~$2$ in~$\theta=0$. It remains to show that, for any~$\rho > 0$, the function~$G_p(\theta,\rho)$ has exactly one zero for~$\theta \in (0,\pi]$. After that, due to the symmetry of the zeros, we conclude that~$G_p(\theta,\rho)$ has exactly four zeros in~$[-\pi,\pi]$, for any~$\rho>0$. To show that, we define
\begin{equation*}
    L_p(\theta,\rho) := \frac{G_p(\theta,\rho)}{M_p(\theta)B_p(\theta)} = -     \rho \frac{C_p^2(\theta)}{M_p(\theta)B_p(\theta)} + \frac{B_p(\theta)}{M_p(\theta)},
\end{equation*}
with the auxiliary function~$M_p$ defined in~\eqref{eq:33}. The function~$L_p(\theta,\rho)$ is well-defined for~$\theta\in(0,\pi]$ since~$M_p(\theta) B_p(\theta)<0$ for~$\theta~\in (0,\pi]$. We aim at showing that~$\partial_{\theta} L_p(\theta,\rho)<0$ for~$\theta \in (0,\pi)$. If this is the case, then 
\begin{equation*}
    G_p(\theta,\rho) \text{~has exactly one zero} \text{~in~} (0,\pi] \text{~if and only if~} \lim_{\theta \to 0} L_p(\theta,\rho)>0 \text{~and~} L_p(\pi,\rho)<0.
\end{equation*}
From the definitions of the functions~$B_p$,~$C_p$, and~$M_p$, we readily compute 
\begin{equation*}
    \lim_{\theta \to 0} L_p(\theta,\rho) = - \rho \lim_{\theta \to 0} \frac{C_p^2(\theta)}{M_p(\theta)B_p(\theta)} =\rho.
\end{equation*}
Recalling~\cite[Corollary 5.4]{FerrariFraschini2024}, we obtain
\begin{equation*}
    L_p(\pi,\rho) = \frac{B_p(\pi)}{M_p(\pi)} = -4 \pi^2 \frac{(2^{2p}-1)}{(2^{2(p+1)}-1)} \frac{\zeta(2p)}{\zeta(2(p+1))} < 0,
\end{equation*}
where~$\zeta$ is the Riemann zeta function. Then, we only need to show that~$\partial_{\theta} L_p(\theta,\rho)<0$ for all~$\theta \in (0,\pi)$ and~$\rho>0$. We compute
\begin{equation*}
    \partial_{\theta} L_p(\theta,\rho) = -\rho \left( \frac{C_p^2(\theta)}{M_p(\theta)B_p(\theta)} \right)' + \left( \frac{B_p(\theta)}{M_p(\theta)} \right)' =: -\rho I_p^1(\theta) + I^2_p(\theta).
\end{equation*}
In~\cite[Theorem 2]{EkstromFurciGaroniManniSerraCapizzano2018}, it has been shown that~$I^2_p(\theta)<0$ for all~$\theta \in (0,\pi)$. Here, we show that~$I^1_p(\theta)>0$ for all~$\theta \in (0,\pi)$. Note that we can factorize~$(2-2\cos \theta)^{2p+2}$ from both the numerator and the denominator, so let us define
\begin{equation} \label{eq:47}
    \widehat{B}_p(\theta) := - \sum_{j \in \Z} \frac{1}{(\theta + 2j\pi)^{2p}}, \quad \widehat{C}_p(\theta) := - \sum_{j \in \Z} \frac{1}{(\theta + 2j\pi)^{2p+1}}, \quad \widehat{M}_p(\theta) := \sum_{j \in \Z} \frac{1}{(\theta + 2j\pi)^{2p+2}}
\end{equation}
and compute
\begin{equation*}
    I^1_p(\theta) = \left( \frac{\widehat{C}_p^2(\theta)}{\widehat{M}_p(\theta)\widehat{B}_p(\theta)} \right)' = \frac{2\widehat{C}_p(\theta)\widehat{C}'_p(\theta)\widehat{M}_p(\theta)\widehat{B}_p(\theta) - \bigl(\widehat{M}_p'(\theta) \widehat{B}_p(\theta) + \widehat{B}_p'(\theta) \widehat{M}_p(\theta)\bigr) \widehat{C}_p^2(\theta)}{(\widehat{M}_p(\theta)\widehat{B}_p(\theta))^2}.
\end{equation*}
From the definitions in~\eqref{eq:47}, we obtain 
\begin{equation*}
    \widehat{C}'_p(\theta) = (2p+1) \widehat{M}_p(\theta), \quad \widehat{B}'_p(\theta) = - 2p \widehat{C}_p(\theta).
\end{equation*}
Therefore,~$I^1_p(\theta) >0$ if and only if 
\begin{equation*}
    2(2p+1)\widehat{C}_p(\theta)\widehat{M}_p^2(\theta)\widehat{B}_p(\theta) - \widehat{C}^2_p(\theta) \widehat{M}_p'(\theta) \widehat{B}_p(\theta) + 2p \widehat{M}_p(\theta) \widehat{C}_p^3(\theta) > 0,
\end{equation*}
and since~$\widehat{C}_p(\theta)<0$ for~$\theta \in (0,\pi)$, see~\eqref{eq:43} in the proof of Proposition~\ref{prop:315}, we can divide by~$\widehat{C}_p(\theta)$ and obtain 
\begin{equation*}
    I^1_p(\theta)>0 \quad \text{~if and only if~} \quad 2(2p+1) \widehat{M}_p^2(\theta)\widehat{B}_p(\theta) - \widehat{C}_p(\theta) \widehat{M}_p'(\theta) \widehat{B}_p(\theta) + 2p \widehat{M}_p(\theta) \widehat{C}_p^2(\theta) < 0.
\end{equation*}
From~\cite[Lemma 1]{EkstromFurciGaroniManniSerraCapizzano2018}, we deduce that~$\theta^2 \widehat{M}_p(\theta) < - \widehat{B}_p(\theta)$ for all~$\theta \in (0,\pi)$ and we compute
\begin{align*}
    & 2(2p+1) \widehat{M}_p^2(\theta)\widehat{B}_p(\theta) - \widehat{C}_p(\theta) \widehat{M}_p'(\theta) \widehat{B}_p(\theta) + 2p \widehat{M}_p(\theta) \widehat{C}_p^2(\theta)
    \\ & \hspace{4cm} < 2(2p+1) \widehat{M}_p^2(\theta)\widehat{B}_p(\theta) - \widehat{C}_p(\theta) \widehat{M}_p'(\theta) \widehat{B}_p(\theta) - 2p \theta^{-2} \widehat{B}_p(\theta) \widehat{C}_p^2(\theta) =: I_p^3(\theta).
\end{align*}
Since~$-\widehat{B}_p(\theta)>0$ for~$\theta \in (0,\pi)$, it holds true that
\begin{equation*}
    I^3_p(\theta)<0 \quad \text{~if and only if~} \quad -2(2p+1) \widehat{M}_p^2(\theta) + \widehat{C}_p(\theta) \widehat{M}_p'(\theta) + 2p \theta^{-2} \widehat{C}_p^2(\theta) < 0.
\end{equation*}
At this point we use that~$-\widehat{C}_p(\theta) < \theta \widehat{M}_p(\theta)$ for all~$\theta \in (0,\pi)$ (see Lemma~\ref{lem:A1} in Appendix~\ref{app:A}), and also 
\begin{equation*}
    \widehat{M}'_p(\theta) = (2p+2)\widehat{C}_{p+1}(\theta)<0
\end{equation*}
(see again the proof of Proposition~\ref{prop:315}). From these, we obtain
\begin{align*}
    -2(2p+1) \widehat{M}_p^2(\theta) + \widehat{C}_p(\theta)\widehat{M}_p'(\theta) + 2p \theta^{-2} \widehat{C}_p^2(\theta) & < -2(2p+1) \widehat{M}_p^2(\theta) - \theta \widehat{M}_p(\theta) \widehat{M}_p'(\theta) + 2p \widehat{M}_p^2(\theta) 
    \\ & = \widehat{M}_p(\theta) \left((-2p-2)\widehat{M}_p(\theta) - \theta \widehat{M}'_p(\theta \right).
\end{align*}
It remains to show that
\begin{equation} \label{eq:48}
    \widehat{M}_p(\theta) > - \frac{\theta}{2p+2} \widehat{M}'_p(\theta),
\end{equation}
which is equivalent to 
\begin{equation*}
    \sum_{j \in \Z} \frac{1}{(\theta+2j\pi)^{2p+2}} > \sum_{j \in \Z} \frac{\theta}{(\theta+2j\pi)^{2p+3}}.
\end{equation*}
The latter readily follows term-by-term. Indeed, for~$j<0$ the addends on the right are negative, while those on the left are positive. For~$j \ge 0$, we have
\begin{equation*}
    \frac{1}{(\theta+ 2j\pi)^{2p+2}} \ge \frac{\theta}{(\theta+2j\pi)^{2p+3}} \iff \theta + 2j\pi \ge \theta \iff j \ge 0.
\end{equation*}
\end{proof}
Combining Property~\ref{prop:316} and Proposition~\ref{prop:321} (see also Remark~\ref{rem:318}), we obtain our main result.
\begin{theorem}
Under Assumption~\ref{assu:1}, for $p =1,\ldots,\numSchur$ and for all~$\rho >0$, the families of Schur complements~$\{\rho \bC_n^p (\B_n^p)^{-1} \bC_n^p + \B_n^p\}_n$ and $\{\rho \bC_n^p + \B_n^p (\bC_n^p)^{-1} \B_n^p\}_n$ are weakly well-conditioned.
\end{theorem}

\section{Why equal trial and test spaces fail} \label{sec:4}
In this section, with the same techniques used to verify the weakly well-conditioning of the family of matrices associated with the variational scheme~\eqref{eq:6}, we show that the numerical discretization of~\eqref{eq:5} obtained with maximal regularity splines of the \emph{same} degree for trial and test functions
leads only to conditional stability, i.e., we have weakly well-conditioning if and only if a CFL condition of the type~$\mu h^2 <\rho_p$ is satisfied, for a positive constant~$\rho_p$ only depending on~$p$.

Consider the following discrete variational formulation: find~$(u_h^p, v_h^p) \in S_{h,0,\bullet}^{(p,p-1)}(0,T) \times S_{h,0,\bullet}^{(p,p-1)}(0,T)$ such that
\begin{equation} \label{eq:49}
	\begin{cases}
		(\partial_t u_h^p, \chi_h^p)_{L^2(0,T)} - (v_h^p, \chi_h^p)_{L^2(0,T)} = 0 & \text{for all~} \chi_h^p \in S_{h,\bullet,0}^{(p,p-1)}(0,T),
		\\ (\partial_t v_h^p, \lambda_h^p)_{L^2(0,T)} + \mu (u_h^p, \lambda_h^p)_{L^2(0,T)} = (f, \lambda_h^p)_{L^2(0,T)} & \text{for all~} \lambda_h^p \in S_{h,\bullet,0}^{(p,p-1)}(0,T).
	\end{cases}
\end{equation}
The associated linear system is
\begin{equation} \label{eq:50}
	\begin{bmatrix}
		\bC_h^p & - \M_h^p \vspace{0.1cm}  \\
        \mu \M_h^p & \bC_h^p 
	\end{bmatrix}
	\begin{bmatrix}
		\bu_h^p \vspace{0.1cm} \\
		\bv_h^p 
	\end{bmatrix} 
	=
	\begin{bmatrix}
		\boldsymbol{0} \vspace{0.1cm} \\
		\bff_h^p
	\end{bmatrix},
\end{equation}
where the matrix~$\bC_h^p$ is introduced in~\eqref{eq:7},~$\bu_h^p, \bv_h^p$ and~$\bff_h^p$ in~\eqref{eq:9}, and the mass matrix~$\M_h^p$ is defined as
\begin{equation} \label{eq:51}
	\M^p_h[\ell,j] := (\phi_j^p, \phi^p_{\ell-1})_{L^2(0,T)}, \quad \ell,j = 1,\ldots,N+p-1.
\end{equation}
\begin{remark} \label{rem:41}
The properties of the matrix~$\M^p_h$ are discussed in~\cite[Proposition 3.2]{FerrariFraschini2024}. It shares the same structure as the matrix~$\B_h^p$ with the following differences: 
\begin{itemize}
\item the entries of~$\frac{1}{h} \M_h^p$ are independent of the mesh parameter~$h$,
\item  the nonzero elements of the purely Toeplitz band part of~$\frac{1}{h} \M_h^p$ can be expressed as
\begin{equation*}
		\frac{1}{h} \M^p_{h}[\ell,\ell -1 \pm j]  = \Phi_{2p+1} (p+1-j),
\end{equation*}
for~$j=0,\ldots,p$ and~$\ell=2p+1,\ldots,N-p$, where~$\Phi_j$ is the cardinal spline of degree~$j$.
\end{itemize}
\end{remark}
The invertibility of the system matrix in~\eqref{eq:50} appears to be true in practice for all~$p \ge 1$, but we have not been able to prove it. Thus, we state the following assumption analogous to~Assumption~\ref{assu:1}.
\begin{assumption} \label{assu:2}
We assume that~$p \in \N$ is such that the system matrix in~\eqref{eq:50} is invertible for all~$\mu, h > 0$.
\end{assumption}
As in Remark~\ref{rem:25}, uniqueness at the continuous level can be established by using the modified Hilbert transform~$\mathcal{H}_T$ with an argument that can not be directly applied at discrete level. Indeed, consider the homogeneous problem: find~$(u,v) \in H^1_{0,\bullet}(0,T) \times H^1_{0,\bullet}(0,T)$ such that
\begin{equation*}
	\begin{cases}
		(\partial_t u, \chi)_{L^2(0,T)} - (v, \chi)_{L^2(0,T)} = 0 & \text{for all~} \chi \in H^1_{\bullet,0}(0,T),
		\\ (\partial_t v, \lambda)_{L^2(0,T)} + \mu (u, \lambda)_{L^2(0,T)} = 0 & \text{for all~} \lambda \in H^1_{\bullet,0}(0,T).
	\end{cases}
\end{equation*}
Uniqueness of the solution~$(u,v)=(0,0)$ is proven by taking~$\chi = - \mathcal{H}_T v$ and~$\lambda = \mathcal{H}_T u$, summing the two equations, integrating by parts, and employing properties~\eqref{eq:12}.

A preliminary difference with respect to system~\eqref{eq:10} is that the linear system~\eqref{eq:50} cannot be solved in two different ways similarly to~\eqref{eq:19} and~\eqref{eq:20}. In fact, the family of matrices~$\{ \M_n^p = \frac{1}{h} \M_h^p\}_n$, with~$n = N+p-1$, is never weakly well-conditioned.
\begin{proposition}
For any~$p \in \N$, the family of matrices~$\{\M_n^p\}_n$ as in~\eqref{eq:51} is not weakly well-conditioned.
\end{proposition}
\begin{proof}
Let us introduce the family of matrices~$\{ \widetilde{\M}_n^p\}_n$, where~$\widetilde{\M}_n^p$ extends the purely Toeplitz band part of~$\M_n^p$ to an~$n\times n$ matrix. In view of Remark~\ref{rem:41}, Lemma~\ref{lemma311}, Remark~\ref{rem:314}, and Theorem~\ref{thm:34}, it is sufficient to show that the polynomial~$q^{\M^p}$ associated with the family~$\{\widetilde{\M}^p_n\}_n$ does not have exactly two zeros of unit modulus nor do all its zero have unit modulus. The restriction of~$q^{\M^p}$ to the boundary of the complex unit circle was obtained in~\cite[Proposition 5.1]{FerrariFraschini2024} and is 
\begin{equation} \label{eq:52}
    q^{\M^p}(e^{\mi \theta}) = e^{\mi p \theta} M_p(\theta),
\end{equation}
with~$M_p$ as in~\eqref{eq:33}. It readily follows~$\lim_{\theta \to 0} M_p(\theta) = 1$, and also~$M_p(\theta) > 0$ for all~$\theta\in [-\pi,0) \cap (0,\pi]$, from which~$q^{\M^p}$ has no zeros of unit modulus and the proof is complete.
\end{proof}
The invertibility of the matrices~$\bC_h^p$ is verified in Proposition~\ref{prop:315} for $p =1,\ldots, \numC$. Then, for~these values of~$p$, the Schur complement~$\bC^p_h + \mu \M_h^p (\bC_h^p)^{-1} \M_h^p$ is well defined and, under Assumption~\ref{assu:2}, is invertible.
\begin{remark}
For~$p=1$, the Schur complement~$\bC^1_h + \mu \M_h^1 (\bC_h^1)^{-1} \M_h^1$ is a lower triangular matrix with entries all equal to~$1/2 +\mu h ^2/18$ on the diagonal. Its invertibility for all~$\mu,h>0$, and thus that of the system matrix in~\eqref{eq:50} for~$p=1$, readily follows.  
\end{remark}
The linear system~\eqref{eq:50} can then solved with the following scheme: write~$\bu_h^p$ from the first equation as~$\bu^p_h = (\bC^p_h)^{-1} \M^p_h \bv_h^p$, insert it into the second one and solve for~$\bv_h^p$, then recover~$\bu_h^p$:
\begin{align*}
	\left( \bC^p_h + \mu \M_h^p (\bC_h^p)^{-1} \M_h^p \right) \bv^p_h & = \bff_h^p,
  \\ \bu^p_h & = (\bC^p_h)^{-1} \M^p_h \bv_h^p.
\end{align*}
We now study the behaviour of the conditioning of the family of the (scaled) Schur complements
\begin{equation} \label{eq:53}
    \bigl\{ \rho \M_n^p (\bC_n^p)^{-1} \M_n^p + \bC_n^p \bigr\}_n, \quad \text{with~} \rho := \mu h^2.
\end{equation}
Similarly to Section~\ref{sec:33}, we introduce the matrices
\begin{equation*}
    \W_n^p(\rho) := \rho (\M_n^p)^2 +  (\bC_n^p)^2,
\end{equation*}
and denote by~$\widetilde{\W}_n^p(\rho)$ the extension of the purely Toeplitz part of~$\W_n^p(\rho)$ to an~$n \times n$ matrix. A componentwise bound can be obtained for~$| (\bC_n^p)^{-1}|$ as in Lemma~\ref{lem:317}. Then, with a similar reasoning as that in Section~\ref{sec:33}, we expect a property analogous to Property~\ref{prop:316} to be true for the families~$\bigl\{ \rho \M_n^p (\bC_n^p)^{-1} \M_n^p + \bC_n^p \bigr\}_n$ and~$\{ \widetilde{\W}_n^p(\rho)\}_n$. In Figures~\ref{fig:1} and~\ref{fig:2}, we report the behaviour of the condition numbers of these two families, and verify that the results are sharp. This justifies the study of the conditioning of the family~$\{ \widetilde{\W}_n^p(\rho)\}_n$ instead of that of the family of the Schur complements.  

Differently from the family~$\{\widetilde{\G}_n^p(\rho)\}_n$, which was shown in Proposition~\ref{prop:321} to always be weakly well-conditioned, we now show that the family~$\{\widetilde{\W}_n^p(\rho)\}_n$ is weakly well-conditioning only provided that~$\rho$ is sufficiently small. In the following proposition, we establish when~$\{\widetilde{\W}_n^p(\rho)\}_n$ is \emph{not} weakly well-conditioned.
\begin{proposition} \label{prop:45}
For any~$p \in \N$ there exists~$\widetilde{\rho}_p > 0$ such that, for~$\rho > \widetilde{\rho}_p$, the family of matrices~$\{\widetilde{\W}_n^p(\rho)\}_n$ is not weakly well-conditioned.
\end{proposition}
\begin{proof}
For any~$n$, the matrix~$\widetilde{\W}_n^p(\rho)$ has the same structure and symmetry of~$\widetilde{\G}_n^p(\rho)$, therefore Lemma~\ref{lem:319} applies also to the family~$\{\widetilde{\W}_n^p(\rho)\}_n$. We study when the polynomial~$q^{\W^p(\rho)}$ has exactly four zeros on the boundary of unitary circle. From Lemma~\ref{lem:320} and the identities~\eqref{eq:32} and~\eqref{eq:52}, we compute 
\begin{equation} \label{eq:54}
	q^{\W^p(\rho)}(e^{\mi \theta}) = e^{2 \mi p \theta} \left(\rho M_p^2(\theta) - C_p^2(\theta)\right) =: e^{2\mi p \theta} W_p(\theta,\rho).
\end{equation}
Here, the functions~$M_p$ and~$C_p$ are defined in~\eqref{eq:33} and~\eqref{eq:30}, respectively. Note that the function~$W_p(\theta,\rho)$ is symmetric in the variable~$\theta$ with respect to~$\theta=0$. Then, recalling~\eqref{eq:35} and~\eqref{eq:36}, we compute
\begin{equation} \label{eq:55}
    \lim_{\theta \to 0} W_p(\theta,\rho) = \rho \lim_{\theta \to 0} M_p^2(\theta) = \rho, \quad W_p(\pi,\rho) = \rho M_p^2(\pi)>0.
\end{equation}
Therefore, the weakly well-conditioning is guaranteed if and only if~$W_p(\theta,\rho)$ has exactly two zeros in~$\theta$ for~$\theta \in (0,\pi)$. For~$\rho$ sufficiently large,~$W_p(\theta,\rho) \approx \rho M_p^2(\theta)$ and, since~$M_p^2(\theta)>0$ for all~$\theta \in [-\pi,\pi]$, we cannot expect two zeros in~$(0,\pi)$. In more detail, since~$M_p'(\theta)<0$ for~$\theta \in (0,\pi)$ (see~\cite[Lemma A.2]{DonatelliGaroniManniSerraCapizzanoSpeleers2017}), then 
\begin{equation} \label{eq:56}
    \min_{\theta \in [0,\pi]} M_p^2(\theta) = M_p^2(\pi).
\end{equation}
Recalling that~$\lim_{\theta \to 0} C_p(\theta)=0$ and~$C_p(\pi)=0$ (see~\eqref{eq:35}), there exists~$\theta^{\text{max}}_p \in (0,\pi)$ such that
\begin{equation} \label{eq:57}
    \max_{\theta \in [0,\pi]} C_p^2(\theta) = C_p^2(\theta^{\text{max}}_p).
\end{equation}
Combining~\eqref{eq:56} and~\eqref{eq:57}, we deduce that, for 
\begin{equation} \label{eq:58}
    \rho > \widetilde{\rho}_p := \frac{C_p^2(\theta^{\text{max}}_p)}{M_p^2(\pi)},
\end{equation}
we have that~$W_p(\theta,\rho)>0$ for all~$\theta \in [-\pi,\pi]$, and therefore the family~$\{ \widetilde \W_n^p(\rho)\}_n$ is not weakly well-conditioned.
\end{proof}
By virtue of Lemma~\ref{lem:B2}, the value~$\theta^{\text{max}}_p \in (0,\pi)$ such that~$C_p^2(\theta^{\text{max}}_p) = \maxx_{\theta \in (0,\pi)} C_p^2(\theta)$ coincides with the unique zero of~$C_p'(\theta)$. It is then possible to use Newton's method to find a suitable approximation of~$\theta^{\text{max}}_p$. We report in Table~\ref{tab:2} some numerical approximations of~$\theta^{\text{max}}_p$ and~$\widetilde{\rho}_p$.

\begin{table}[h!]
  \centering
  \renewcommand{\arraystretch}{1.5} 
  \begin{tabular}{|c|c|c|c|c|c|c|c|c|c|c|c|}\hline
  ~$p$ & 1 & 2 & 3 & 4 & 5 & 6 \\\hline
  ~$\theta^{\text{max}}_p$ &~$\pi/2$ &  1.384  & 1.209 & 1.085 & 0.9917 & 0.9192 \\
  ~$\widetilde{\rho}_p$ &~$9$ & 4.057$e+01$ & 1.871$e+02$ & 9.138$e+02$& 4.644$e+03$ & 2.426$e+04$  \\
  \hline 
  \end{tabular}
  \vspace{0.15cm}
  \caption{Numerical approximations of~$\theta^{\text{max}}_p$ and~$\widetilde{\rho_p}$, defined in~\eqref{eq:57} and~\eqref{eq:58}, respectively.}
  \label{tab:2}
\end{table}

From Proposition~\ref{prop:45}, we obtain that method~\eqref{eq:49} is not stable if~$\mu h^2 > \widetilde{\rho}_p$. However, this result is not sharp. In contrast to~\cite[Theorem 5.9]{FerrariFraschini2024}, there do not seem to be any explicit characterizations of the sharp CFL parameters~$\rho_p$. We recall that, in order for the method~\eqref{eq:49} to be stable, we look for the maximum~$\rho_p$ such that, if~$\mu h^2 = \rho < \rho_p$, then the function~$W_p(\theta,\rho)$ in~\eqref{eq:54} has exactly two zeros in~$(0,\pi)$. Define~$\displaystyle{E_p:=\frac{2(2p+1)^2}{p+1} \frac{M_p(\pi)}{(2p+3)M_{p+1}(\pi)-M_p(\pi)}}$. Then,~$E_p>0$, and we expect the following property to be valid:
\begin{equation}\label{eq:59}
    \text{for any fixed~$\rho$,~$0<\rho < E_p$, the function~$\theta\mapsto \partial_{\theta} W_p(\theta,\rho)$ has exactly one zero in~$(0,\pi)$.}
\end{equation}
We postpone a justification of this, including the proof that~$E_p>0$, to Appendix~\ref{app:B}. For any fixed~$p \ge 1$ and~$0<\rho < E_p$,  let us define~$\theta_p(\rho) \in (0,\pi)$ such that~$\partial_\theta W_p(\theta_p(\rho),\rho)=0$. Recalling~\eqref{eq:55},~$W_p(\theta,\rho)$ has exactly two zeros in~$(0,\pi)$ if and only if ~$W_p(\theta_p(\rho),\rho) \le 0$. The maximum value of~$\rho$ for which there are exactly two zeros is the one for which the two zeros coincide.

We aim at finding the limit quantities
$(\theta_p,\rho_p) \in (0,\pi) \times (0,E_p)$ such that~$W_p(\theta_p,\rho_p)=\partial_\theta W_p(\theta_p,\rho_p) = 0$. To compute them, we need to solve the nonlinear system
\begin{equation} \label{eq:60}
\begin{cases}
    W_p(\theta_p,\rho_p) = \rho_p M_p^2(\theta_p) - C_p^2(\theta_p) = 0, \\
    \partial_\theta W_p(\theta_p,\rho_p) = 2\rho_p M_p(\theta_p) M_p'(\theta_p) - 2 C_p(\theta_p) C_p'(\theta_p) = 0.
\end{cases}
\end{equation}
From the first equation of~\eqref{eq:60}, we obtain~$\rho_p M_p(\theta_p) = C_p^2(\theta_p)/M_p(\theta_p)$ which, inserted into the second one, leads to~$C_p(\theta_p) M_p'(\theta_p) = C_p'(\theta_p)M_p(\theta_p)$. Therefore, we obtain that~$\theta_p$ must be a zero of the function 
\begin{equation*}
    F_p(\theta) := C_p(\theta) M_p'(\theta) - C_p'(\theta)M_p(\theta).
\end{equation*}
From~\eqref{eq:35},~\eqref{eq:36}, and~\eqref{eq:38}, we obtain 
\begin{equation*}
    \lim_{\theta \to 0} F_p(\theta) = 1, \quad F_p(\pi) = - (2p+1) M_p^2(\pi).
\end{equation*}
Moreover, we have numerically validated, and will address the proof in~\cite{Ferrari2025}, that~$F_p'(\theta)<0$ for all~$\theta \in (0,\pi)$, which implies that there exists a unique zero~$\theta_p \in (0,\pi)$ of~$F_p$.

Again with Newton's method, we have computed numerical approximations of~$(\theta_p,\rho_p)$ for several values of~$p$, and we report them in Table~\ref{tab:3}. 

\begin{table}[h]
  \centering
  \renewcommand{\arraystretch}{1.5} 
  \begin{tabular}{|c|c|c|c|c|c|c|c|c|c|c|c|}\hline
  ~$p$ & 1 & 2 & 3 & 4 & 5 & 6 \\\hline
  ~$ \theta_p$ &~$2\pi/3$ &  2.332  & 2.475 & 2.571 & 2.641 & 2.695 \\
  ~$\rho_p$ &~$3$ & 4.318 & 5.204 & 5.834 & 6.305 & 6.671  \\
  ~$E_p$ & 9 &  9.091  & 9.256 & 9.369 & 9.449 & 9.596 \\
  \hline 
  \end{tabular}
  \vspace{0.15cm}
  \caption{Numerical approximations of~$(\theta_p,\rho_p,E_p)$ such that~$W_p(\theta_p,\rho_p) = \partial_\theta W_p(\theta_p,\rho_p) = 0$.}
  \label{tab:3}
\end{table}
\begin{remark}
If~$p=1$, all the calculations simplify. Indeed, we explicitly compute~$F_1(\theta) = -\tfrac{1}{3} (2\cos \theta +1).$
%
%
\end{remark}
In conclusion, we expect the method~\eqref{eq:49} (with equal test and trial spaces and without any additional stabilization) to be stable if and only if~$\rho < \rho_p$, and this result seems to be sharp. In Figures~\ref{fig:1} and~\ref{fig:2}, we show the condition numbers of the Schur complements~$\bigl\{\bC_n^p + \rho \M_n^p (\bC_n^p)^{-1} \M_n^p \bigr\}_n$  for a fixed system dimension~$n=1000$ and varying the parameter~$\rho$, for~$p \in \{1,2,3,4,5,6\}$. We also mark with vertical lines the computed values of~$\rho_p$ reported in Table~\ref{tab:3}. As with the CFL constants obtained in~\cite[Equation 2.7]{FerrariFraschini2024} for the second-order-in-time variational formulation, the sequence~$\{\rho_p\}_p$ is bounded, and it is numerically validated that~$\rho_p \approx 10$ for large~$p$.

\begin{figure}[h!]
    \begin{minipage}{0.485\textwidth}
        \centering
        \includegraphics[width=\linewidth]{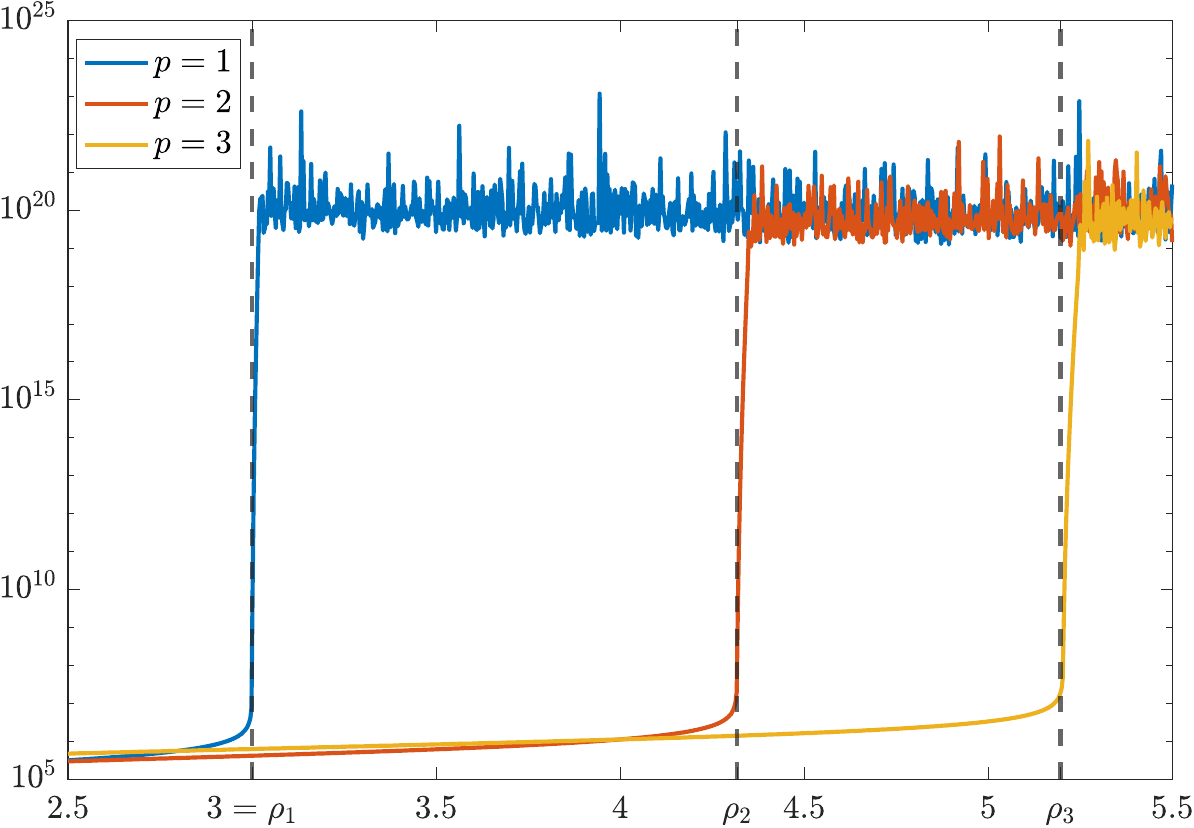}
        \caption{Spectral condition numbers of the Schur complements in~\eqref{eq:53} in semi-logarithmic scale, with~$n=1000$ by varying~$\rho \in [2.5,5.5]$, with~$p \in \{1,2,3\}$.}
        \label{fig:1}
    \end{minipage}%
    \hfill \begin{minipage}{0.485\textwidth}
        \centering
        \includegraphics[width=\linewidth]{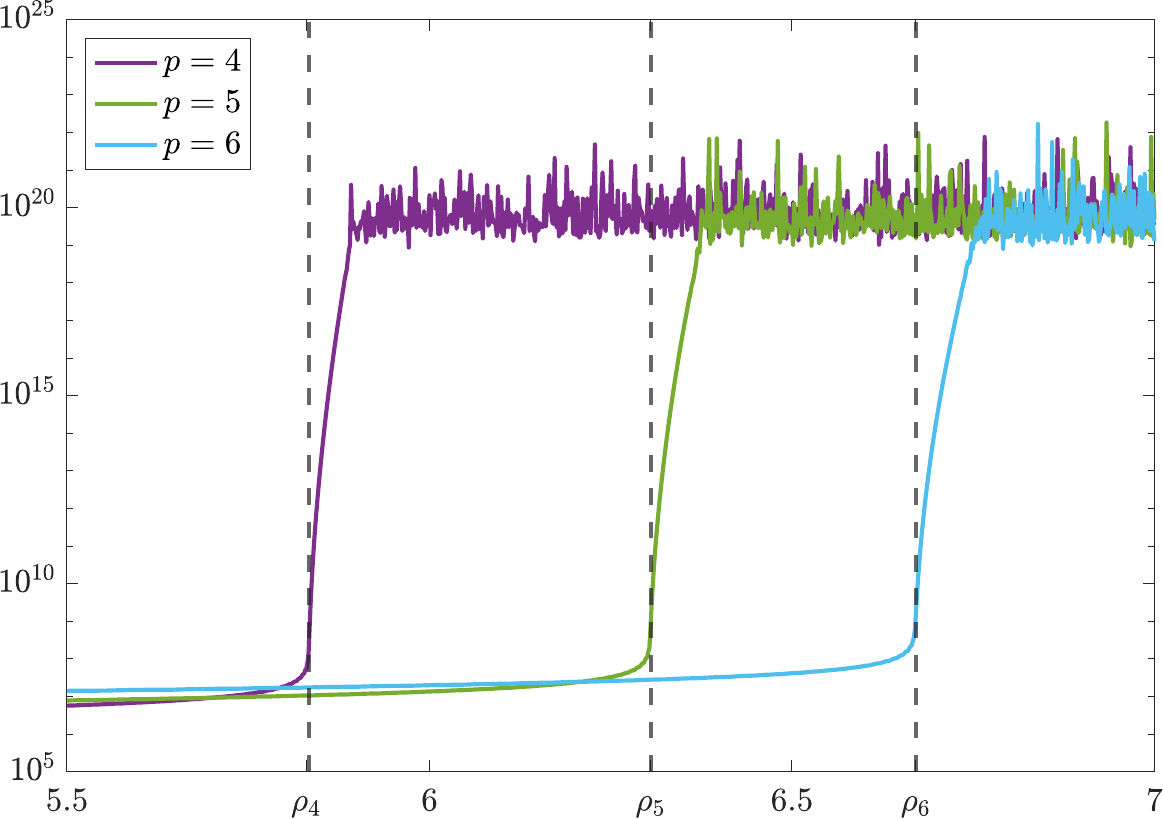}
        \caption{Spectral condition numbers of the Schur complements in~\eqref{eq:53} in semi-logarithmic scale, with~$n=1000$ by varying~$\rho \in [5.5,7]$, with~$p \in \{4,5,6\}$.}
        \label{fig:2}
    \end{minipage}
\end{figure}

\section{Numerical experiments }\label{sec:5}
In this section, we provide details for an efficient implementation of the discretization of the first-order-in-time space--time formulation of the wave equation (Section~\ref{secEff}). Then, using an isogeometric discretization also in space, we present several numerical tests that validate the presented stability results and demonstrate the performance of the complete space--time scheme (Section~\ref{sec6}).

\subsection{Efficient implementation} \label{secEff}
For the implementation aspects and numerical tests, we consider the following extension of~\eqref{eq:2} that includes the case of Neumann and Robin boundary conditions (see Remark~\ref{rem:22}):
\begin{equation} \label{eq:mod_prob_ext}
	\begin{cases}
 \partial_t U(\bx,t) - V(\bx,t) = 0 & (\bx,t) \in Q_T,
		\\ \partial_{t} V(\bx,t) - \div(c^2(\bx) \nabla_{\bx} U(\bx,t)) = F(\bx,t) & (\bx,t) \in Q_T, 
		\\ U(\bx,t) = 0 & (\bx,t) \in \, \, \Sigma_D := \Gamma_D \times [0, T], 
        \\ c^2(\bx) \nabla_{\bx} U(\bx,t) \cdot \bn(\bx) = g_N(\bx,t) & (\bx,t) \in \, \, \Sigma_N := \Gamma_N \times [0, T],
        \\ \vartheta c(\bx) V(\bx,t) + c^2(\bx) \nabla_{\bx} U(\bx,t) \cdot \bn(\bx) = g_R(\bx,t) & (\bx,t) \in \, \, \Sigma_R := \Gamma_R \times [0, T],
		\\ U(\bx,0) = 0, \quad V(\bx,0) = 0 & \bx \in \Omega, 
	\end{cases}
\end{equation}
where~$\partial\Omega$ is partitioned as~$\partial \Omega = \overline{\Gamma_D \cup \Gamma_N \cup \Gamma_R}$, with~$\Gamma_D$,~$\Gamma_N$, and~$\Gamma_R$ having disjoint interiors, and where~$\vartheta>0$ is the impedance parameter.

In order to write the matrix form of
the discrete variational formulation of~\eqref{eq:mod_prob_ext} analogous to~\eqref{eq:4}, let us denote by~$\bC^p_{h_t}$ and~$\B^p_{h_t}$ the temporal matrices already introduced in~\eqref{eq:7}. Additionally, let~${\M}_{h_{\bx}}$,~${\K}_{h_{\bx}}$, and~$\M_{h_{\bx}}^R$ represent the mass matrix, the stiffness matrix, and the mass matrix relative to~$\Gamma_R$, respectively, associated with a basis~$\{ \psi_m \}_{m=1}^{N_{\bx}}$ of the discretization space $V_{h_{\bx}}(\Omega)\subset H^1_{\Gamma_D}(\Omega)$ of finite dimension~$N_x$, where~$H^1_{\Gamma_D}(\Omega)$ is the subspace of~$H^1(\Omega)$ of functions with zero trace on~$\Gamma_D$. Let 
\begin{equation} \label{basis_FFLP}
\{ \Psi^p_{\bh,m,j}(\bx,t):= \psi_m(\bx) \phi_j^p(t), \,  m=1,\ldots, N_{\bx} \ \text{and} \ j=0,\ldots,N_t+p-1 \}
\end{equation}
be a basis for~$Q_{\bh}^{(p,p-1)}(Q_T)$. The components of the right-hand side vector~$\mathbf{F}_{\bh}^p\in \R^{N_{\bx}(N_t+p-1)}$ are defined,
for $m=1,\ldots, N_{\bx}$ \ and \ $j=1,\ldots,N_t+p-1$, as
\begin{equation*}
    \begin{aligned}
    \mathbf{F}_{\bh}^p[m + N_{\bx}(j-1)] := (F, \partial_t \Psi^p_{\bh,m,j-1})_{L^2(Q_T)} + \int_0^T \int_{\Gamma_N} &g_N(\bx,t) \Psi^p_{\bh,m,j-1}(\bx,t) \, \dd \bx \, \dd t \\
    &+ \int_0^T \int_{\Gamma_R} g_R(\bx,t) \Psi^p_{\bh,m,j-1}(\bx,t) \, \dd \bx \, \dd t.
\end{aligned}
\end{equation*}
Finally, the discrete solution~$(U_{\bh}^p, V_{\bh}^p) \in Q_{\bh,0,\bullet}^{(p,p-1)}(Q_T) \times Q_{\bh,0,\bullet,}^{(p,p-1)}(Q_T)$ is represented in terms of the space--time basis~\eqref{basis_FFLP} as
\begin{equation*}
U_{\bh}^p(\bx,t) = \sum_{m=1}^{N_{\bx}} \sum_{j=1}^{N_t+p-1} U_{m+N_{\bx}(j-1)}^p \Psi_{\bh,m,j}^p(\bx,t), \quad \quad  V_{\bh}^p(\bx,t) = \sum_{m=1}^{N_{\bx}} \sum_{j=1}^{N_t+p-1} V_{m+N_{\bx}(j-1)}^p \Psi_{\bh,m,j}^p(\bx,t),
\end{equation*}
with coefficients collected in the unknown vectors~$\bU_{\bh}^p, \bV_{\bh}^p \in \R^{N_{\bx}(N_t+p-1)}$:
\begin{equation*}
\bU_{\bh}^p := [\, U_\ell^p \,]_{\ell=1}^{N_{\bx}(N_t+p-1)}, \quad \bV_{\bh}^p := [\, V_\ell^p \,]_{\ell=1}^{N_{\bx}(N_t+p-1)}.
\end{equation*}
The matrix form of the discrete variational formulation reads as follows:
\begin{equation}\label{eq:matrix_system}
\begin{cases}
    \hfil ( \B^p_{h_t} \otimes \M_{h_{\bx}} ) \, \bU_{\bh}^p + ( \bC^p_{h_t} \otimes \M_{h_{\bx}} ) \, \bV_{\bh}^p &= \mathbf{0}, \vspace{0.05cm} \\
     (-\bC^p_{h_t} \otimes \K_{h_{\bx}} + \B^p_{h_t} \otimes\M_{h_{\bx}}^R)  \, \bU_{\bh}^p + ( \B^p_{h_t} \otimes \M_{h_{\bx}} )  \, \bV_{\bh}^p &= \mathbf{F}_{\bh}^p,
\end{cases}
\end{equation}
where~$\otimes$ denotes the Kronecker product. 
Using the properties of the Kronecker product, \eqref{eq:matrix_system} can be solved efficiently without the need to assemble space--time matrices. Specifically, for matrices~$\A$,~$\B$, and~$\mathbf{X}$ of appropriate dimensions, we have
\begin{equation*}
    (\A \otimes \B) \mathrm{vec}(\mathbf{X}) = \mathrm{vec}(\B \mathbf{X} \A^{\top}),
\end{equation*}
where~$\mathrm{vec}(\mathbf{X})$ denotes the vector obtained by stacking the entries of~$\mathbf{X}$ into a column vector. Moreover, if~$\A$ and~$\B$ are nonsingular square matrices, we have
\begin{equation*}
    (\A \otimes \B)^{-1} = \A^{-1} \otimes \B^{-1}.
\end{equation*}
Computing~$\bV_{\bh}^p$ from the first equation of~\eqref{eq:matrix_system} and plugging it into the second one, system~\eqref{eq:matrix_system} can be solved with the following procedure:
\begin{itemize}
    \item solve $\mathbf{A}_{\bh}^p \mathbf{U}_h^p = -\mathbf{F}_h^p$ where 
    \begin{equation*}
        \A_{\bh}^p := \bC^p_{h_t} \otimes \K_{h_{\bx}} + \B^p_{h_t} ({\bC}^{p}_{h_t})^{-1} \B^p_{h_t} \otimes \M_{h_{\bx}} - \B^p_{h_t} \otimes \M_{h_{\bx}}^R,
    \end{equation*}
    \item solve~$( \bC^p_{h_t} \otimes \mathbf{I}_{h_{\bx}} ) \, \bV_{\bh}^p =  - ( \B^p_{h_t} \otimes \mathbf{I}_{h_{\bx}} ) \, \bU_{\bh}^p$.
\end{itemize}
Here and in the following,~$\mathbf{I}_{h_{\bx}}$ and~$\mathbf{I}^p_{h_t}$ denote the identity matrices with the same dimensions as~$\K_{h_{\bx}}$ and~$\B^p_{h_t}$, respectively.

We observe that both steps can be performed without explicitly computing Kronecker products. In fact, the matrix~$\A_{\bh}^p$ can be rewritten as:
\begin{equation*}
    \A_{\bh}^p =(\B^p_{h_t} \otimes \mathbf{I}_{h_{\bx}}) \left(({\B}^{p}_{h_t})^{-1} \bC^p_{h_t} \otimes \K_{h_{\bx}} + ({\bC}^{p}_{h_t})^{-1} \B^p_{h_t} \otimes \M_{h_{\bx}} - \mathbf{I}^p_{h_t} \otimes \M_{h_{\bx}}^R \right).
\end{equation*}
Using the standard \emph{complex} Schur decomposition applied to~$(\B^p_{h_t})^{-1} \bC^p_{h_t}$, we find a unitary matrix~$\mathbf{Q}_{h_t}^p$ and an upper triangular matrix~$\mathbf{R}_{h_t}^p$ such that
\begin{equation*}
    (\mathbf{Q}_{h_t}^p)^{H} (\B^p_{h_t})^{-1} \bC^p_{h_t} \mathbf{Q}_{h_t}^p = \mathbf{R}_{h_t}^p,
\end{equation*}
where the superscript~${}^H$ denotes the conjugate transpose.
At this point, we can express $\A_{\bh}^p$ as
\begin{equation*}
    \A_{\bh}^p = \left(\B^p_{h_t} (\mathbf{Q}_{h_t}^p)^{-H} \otimes \mathbf{I}_{h_{\bx}}\right) \left(\mathbf{R}_{h_t}^p \otimes \K_{h_{\bx}} + (\mathbf{R}_{h_t}^p)^{-1} \otimes \M_{h_{\bx}} - \mathbf{I}_{h_t}^p \otimes\M_{h_{\bx}}^R\right) \left((\mathbf{Q}_{h_t}^p)^{-1} \otimes \mathbf{I}_{h_{\bx}}\right),
\end{equation*}
and its inverse becomes
\begin{equation*}
    (\A_{\bh}^p)^{-1} =  (\mathbf{Q}^p_{h_t} \otimes \mathbf{I}_{h_{\bx}})\left(\mathbf{R}^p_{h_t} \otimes \K_{h_{\bx}} + (\mathbf{R}^p_{h_t})^{-1} \otimes \M_{h_{\bx}} - \mathbf{I}^p_{h_t} \otimes \M_{h_{\bx}}^R\right)^{-1} \left((\mathbf{Q}^p_{h_t})^{H}({\B}^{p}_{h_t})^{-1}  \otimes \mathbf{I}_{h_{\bx}} \right),
\end{equation*}
where the middle term has a block upper triangular structure. The procedure is summarized in Algorithm \ref{alg:sol1}. Computing~$\bU_{\bh}^p$ requires solving $N_{\bx}$ independent linear systems associated with ${\B}^{p}_{h_t}$ (Step 2), performing~$N_{\bx}$ matrix-vector products involving~$(\mathbf{Q}_{h_t}^p)^H$ and~$\mathbf{Q}_{h_t}^p$ (Steps 3 and 5), and solving a block upper triangular system, where each block has dimensions $N_{\bx} \times N_{\bx}$ (Step 4). Finally, to compute $\bV_{\bh}^p$, we need to perform $N_{\bx}$ matrix-vector products involving~$\B^p_{h_t}$, and solve $N_{\bx}$ independent linear systems associated with ${\bC}^{p}_{h_t}$ (Step 6).
\begin{algorithm}[ht!]
\caption{efficient implementation for solving the system \eqref{eq:matrix_system}} \label{alg:sol1}
\begin{algorithmic}[1]
    \State Compute the complex Schur decomposition of $(\B_{h_t}^p)^{-1} \bC_{h_t}^p$ obtaining $\mathbf{Q}_{h_t}^p$ and $\mathbf{R}_{h_t}^p$, and compute $(\mathbf{R}_{h_t}^p)^{-1}$
    \vspace{0.1cm}
    \State Solve for $\boldsymbol{Y}_{\bh}^p$ the system $\left({\B}^{p}_{h_t} \otimes \mathbf{I}_{h_{\bx}} \right) \boldsymbol{Y}_{\bh}^p = \mathbf{F}_{\bh}^p$
    \vspace{0.1cm}
    \State Update $\boldsymbol{Y}_{\bh}^p \leftarrow \left((\mathbf{Q}_{h_t}^p)^H \otimes \mathbf{I}_{h_{\bx}}\right) \boldsymbol{Y}_{\bh}^p$
    \vspace{0.1cm}
    \State Solve for $\boldsymbol{Z}_{\bh}^p$ the system $\left(\mathbf{R}^p_{h_t} \otimes \K_{h_{\bx}} + (\mathbf{R}^p_{h_t})^{-1} \otimes \M_{h_{\bx}} - \mathbf{I}^p_{h_t} \otimes \M_{h_{\bx}}^R\right) \boldsymbol{Z}_{\bh}^p =\boldsymbol{Y}_{\bh}^p$
    \vspace{0.1cm}
    \State Compute~$\boldsymbol{U}_{\bh}^p = -(\mathbf{Q}_{h_t}^p \otimes \mathbf{I}_{h_{\bx}}) \boldsymbol{Z}_{\bh}^p$ 
    \vspace{0.1cm}
    \State Solve~$( \bC^p_{h_t} \otimes \mathbf{I}_{h_{\bx}} ) \, \bV_{\bh}^p =  - ( \B^p_{h_t} \otimes \mathbf{I}_{h_{\bx}} ) \, \bU_{\bh}^p$
\end{algorithmic}
\end{algorithm}
\begin{remark}
Algorithm~\ref{alg:sol1} can be extended to handle the general case of non-homogeneous Dirichlet and initial conditions while maintaining a similar computational cost.
\end{remark}

\subsection{Numerical tests} \label{sec6}
In this section, we present numerical tests validating the accuracy and unconditional stability of the first-order-in-time space--time method using maximal regularity splines in time (of degree~$p$ for the trial functions and~$p-1$ for the test functions), and isogeometric discretization in space.

From now on, let us assume the discrete space~$V_{h_{\bx}}(\Omega)$ to be the isoparametric push-forward on~$\Omega$ of the multivariate B-spline space on the reference domain~$(0,1)^d$ ($d=1,2$), with maximal regularity splines, and the spline degree in all space directions being equal to the spline degree~$p$ of the trial spaces in time. For more details on the construction of this space see, e.g.~\cite[Section 3]{FraschiniLoliMoiolaSangalli2023}.

To compare the performance of this scheme with the second-order-in-time \emph{stabilized} space--time isogeometric method devised in~\cite{FraschiniLoliMoiolaSangalli2023}, from Section~\ref{subsec6.1} to Section~\ref{subsec6.7}, we apply the method to the same test cases considered in~\cite{FraschiniLoliMoiolaSangalli2023}. Additionally, in Section~\ref{subsec6.8}, we apply our space--time method to solve a wave propagation problem in a heterogeneous material. All numerical test are performed with Matlab R2024a and the GeoPDEs toolbox~\cite{defalco}. Matlab's direct solver is employed for all the experiments except for
Section~\ref{subsec6.3}, where an iterative solver is employed. The codes used for the numerical tests are available in the GitHub repository \cite{XTIgAWaves}.

\subsubsection{Example 1. Unconditional stability and optimal convergence rates} \label{subsec6.1}
In this section, we validate the unconditional stability of the considered space--time method, and we compare its accuracy with the ones of the discretizations devised in~\cite{FrenchPeterson1996, FraschiniLoliMoiolaSangalli2023}. We actually discretize in space with maximal regularity splines of degree~$p$. For the time discretization, the method of~\cite{FrenchPeterson1996} uses continuous piecewise  polynomials of degree~$p$ for the trial functions, and discontinuous piecewise polynomials of degree~$p-1$ for the test functions. Both our method and that of~\cite{FraschiniLoliMoiolaSangalli2023} use maximal regularity splines of degree~$p$ for trial functions and of degree~$p-1$ for test functions. 

We solve the wave propagation problem~\eqref{eq:2} with one-dimensional space domain~$\Omega = (0,1)$, wave velocity~$c=1$, and exact solution
\begin{equation} \label{u1_ex} 
    U(x,t)= \sin(\pi x)\sin^2\left(\tfrac{5}{4} \pi t\right) \quad \text{for } (x,t) \in Q_T= \Omega \times (0,10).
\end{equation}
The same setting has been considered in~\cite[Section 5.1.1]{FraschiniLoliMoiolaSangalli2023} and~\cite[p. 367]{SteinbachZank2019}.

As shown in Figure~\ref{fig:StabRobust}, the method is unconditionally stable. By reducing the mesh size in space, the errors remain bounded without the need to satisfy any CFL condition. This differs from what has been observed in~\cite{SteinbachZank2019, FraschiniLoliMoiolaSangalli2023}, where a non-consistent penalty term was required to stabilize the corresponding space--time method. Figure~\ref{fig:StabComparisonHP} shows the optimal convergence rates of our method for different choice of the spline degree~$p \in \{1,\ldots,4\}$. Additionally, this figure compares the errors of our method against the ones of~\cite{FrenchPeterson1996, FraschiniLoliMoiolaSangalli2023}. For the same number of degrees of freedom, and except for the error in the position~$U$ in the case~$p=1$, our method is as accurate as the stabilized space--time isogeometric method of~\cite{FraschiniLoliMoiolaSangalli2023}, providing more accurate approximations than the ones obtained by the method of~\cite{FrenchPeterson1996}. This fact confirms the remarkable approximation properties of high-order maximal regularity spline spaces.   

\tikzstyle{Linea1}=[dashed]
\tikzstyle{Linea2}=[thick]
\pgfplotscreateplotcyclelist{Lista1}{%
	{Linea2,Red,mark=*},
	{Linea2,Green,mark=triangle*},
	{Linea2,Cyan,mark=square*},
	{Linea2,Violet,mark=diamond*},
}

\def \DATAFILE {tabErrorStabFO.csv}

\begin{figure}[htbp]
	\centering
	\hspace*{\fill}
	\begin{subfigure}[t]{0.4\linewidth}
		\centering
		\begin{tikzpicture}[
			trim axis left
			]
			\begin{loglogaxis}[
				cycle list name=Lista1,
				width=\linewidth,
				height=\linewidth,
				xlabel={$h_t/h_x$},
				ymin=0.00001,
				ymax=0.5,
				xminorticks=false,
				yminorticks=false,
							ylabel={$\| U - U_{\bh}^p \|_{L^2(Q_T)}/\| U \|_{L^2(Q_T)}$},
				legend columns=2,
				legend style={nodes={scale=0.85, transform shape},at={(0.5,1.05)},anchor= south},
				xmajorgrids=true,
				ymajorgrids=true,
				legend entries={$p=1$,$p=2$,$p=3$,$p=4$}]
				\addplot table [x=htfhs, y=L2P1U, col sep=comma] {\DATAFILE};
				\addplot table [x=htfhs, y=L2P2U, col sep=comma] {\DATAFILE};
				\addplot table [x=htfhs, y=L2P3U, col sep=comma] {\DATAFILE};
				\addplot table [x=htfhs, y=L2P4U, col sep=comma] {\DATAFILE};
			\end{loglogaxis}
		\end{tikzpicture}
	\end{subfigure}
	\hfill
	\begin{subfigure}[t]{0.4\linewidth}
		\centering
		\begin{tikzpicture}[
			trim axis right
			]
			\begin{loglogaxis}[
				cycle list name=Lista1,
				width=\linewidth,
				height=\linewidth,
				xlabel={$h_t/h_x$},
				ymin=0.0001,
				ymax=0.5,
				xminorticks=false,
				yminorticks=false,
							ylabel={$| U - U_{\bh}^p |_{H^1(Q_T)}/| U |_{H^1(Q_T)}$},
				legend columns=2,
				legend style={nodes={scale=0.85, transform shape},at={(0.5,1.05)},anchor= south},
				xmajorgrids=true,
				ymajorgrids=true,
				legend entries={$p=1$,$p=2$,$p=3$,$p=4$}]
				\addplot table [x=htfhs, y=H1sP1U, col sep=comma] {\DATAFILE};
				\addplot table [x=htfhs, y=H1sP2U, col sep=comma] {\DATAFILE};
				\addplot table [x=htfhs, y=H1sP3U, col sep=comma] {\DATAFILE};
				\addplot table [x=htfhs, y=H1sP4U, col sep=comma] {\DATAFILE};
			\end{loglogaxis}
		\end{tikzpicture}
	\end{subfigure}
    \hspace*{\fill}
    \\
    \vspace{0.3cm}
    \hspace*{\fill}
	\begin{subfigure}[t]{0.4\linewidth}
		\centering
		\begin{tikzpicture}[
			trim axis left
			]
			\begin{loglogaxis}[
				cycle list name=Lista1,
				width=\linewidth,
				height=\linewidth,
				xlabel={$h_t/h_x$},
				ymin=0.00001,
				ymax=0.5,
				xminorticks=false,
				yminorticks=false,
				ylabel={$\| V - V_{\bh}^p \|_{L^2(Q_T)}/\| V \|_{L^2(Q_T)}$},
				legend columns=2,
				legend pos=south east,
				legend style={nodes={scale=0.85, transform shape},at={(0.01,0.85)},anchor= north west, draw=none},
				xmajorgrids=true,
				ymajorgrids=true,
                ]
				\addplot table [x=htfhs, y=L2P1V, col sep=comma] {\DATAFILE};
				\addplot table [x=htfhs, y=L2P2V, col sep=comma] {\DATAFILE};
				\addplot table [x=htfhs, y=L2P3V, col sep=comma] {\DATAFILE};
				\addplot table [x=htfhs, y=L2P4V, col sep=comma] {\DATAFILE};
			\end{loglogaxis}
		\end{tikzpicture}
	\end{subfigure}
	\hfill
	\begin{subfigure}[t]{0.4\linewidth}
		\centering
		\begin{tikzpicture}[
			trim axis right
			]
			\begin{loglogaxis}[
				cycle list name=Lista1,
				width=\linewidth,
				height=\linewidth,
				xlabel={$h_t/h_x$},
				ymin=0.0001,
				ymax=0.5,
				xminorticks=false,
				yminorticks=false,
				ylabel={$| V - V_{\bh}^p |_{H^1(Q_T)}/| V |_{H^1(Q_T)}$},
				legend columns=2,
				legend style={nodes={scale=0.85, transform shape},at={(0.01,0.01)},anchor= south west, draw=none},
				xmajorgrids=true,
				ymajorgrids=true,
                ]
				\addplot table [x=htfhs, y=H1sP1V, col sep=comma] {\DATAFILE};
				\addplot table [x=htfhs, y=H1sP2V, col sep=comma] {\DATAFILE};
				\addplot table [x=htfhs, y=H1sP3V, col sep=comma] {\DATAFILE};
				\addplot table [x=htfhs, y=H1sP4V, col sep=comma] {\DATAFILE};
			\end{loglogaxis}
		\end{tikzpicture}
	\end{subfigure}
	\hspace*{\fill}
	\caption{Example 1. Relative errors plotted against the ratio $h_t/h_x$ with fixed $h_t=0.1562$. The exact solution is defined in \eqref{u1_ex}.}
	\label{fig:StabRobust}
\end{figure}
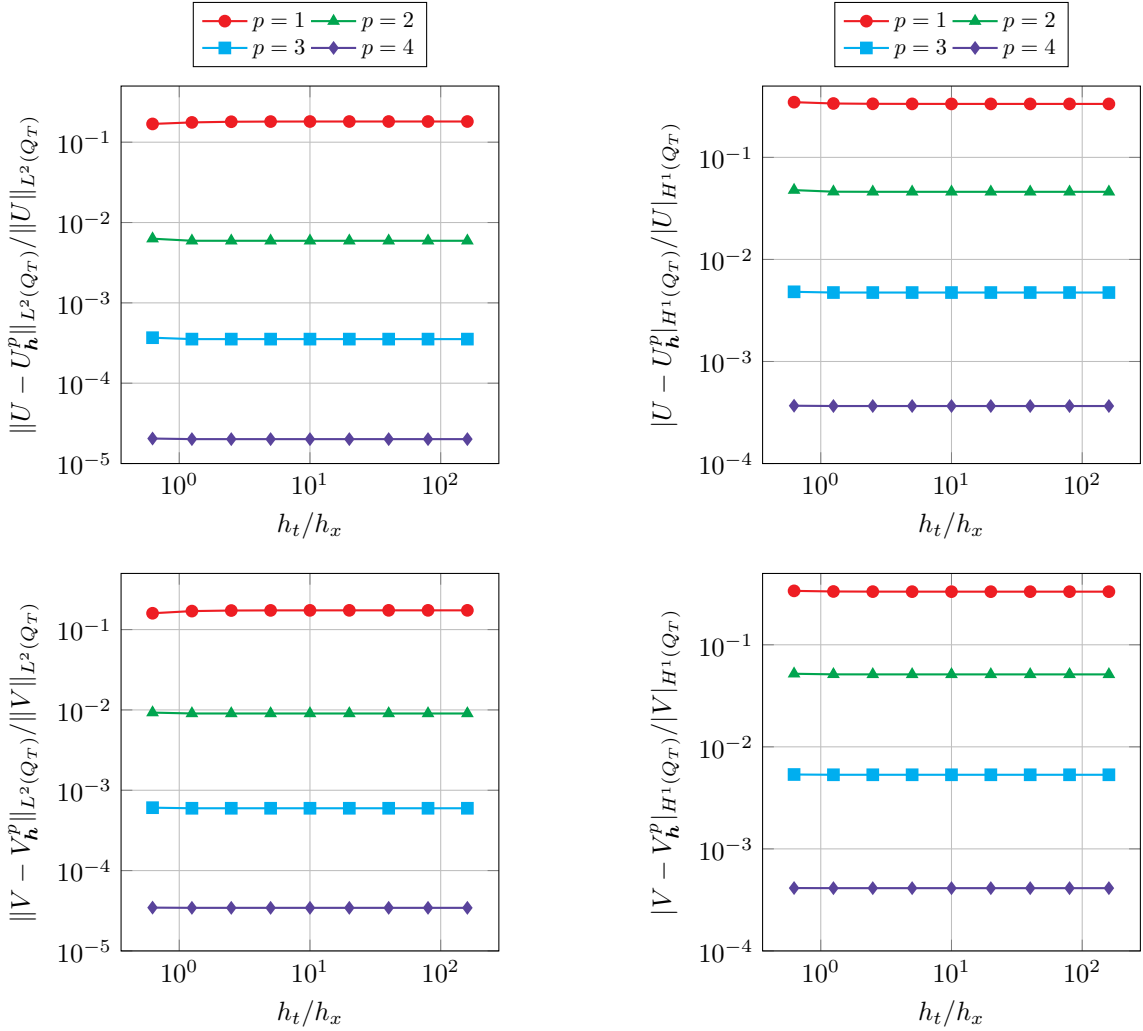

\tikzstyle{Linea1}=[thick,dashed, mark options={solid}]
\tikzstyle{Linea2}=[thick]
\tikzstyle{Linea3}=[thick,dotted]
\tikzstyle{Linea4}=[thick,dashdotted]

\pgfplotscreateplotcyclelist{Lista1}{
{Linea1,Red,mark=*, mark size=1.35pt},
{Linea2,Red,mark=*, mark size=1.35pt},
{Linea4,Red,mark=*, mark size=1.35pt},
{Linea3,Red},
{Linea1,Green,mark=triangle*, mark size=2pt},
{Linea2,Green,mark=triangle*, mark size=2pt},
{Linea4,Green,mark=triangle*, mark size=2pt},
{Linea3,Green},
{Linea1,Cyan,mark=square*, mark size=1.35pt},
{Linea2,Cyan,mark=square*, mark size=1.35pt},
{Linea4,Cyan,mark=square*, mark size=1.35pt},
{Linea3,Cyan},
{Linea1,Violet,mark=diamond*, mark size=2pt},
{Linea2,Violet,mark=diamond*, mark size=2pt},
{Linea4,Violet,mark=diamond*, mark size=2pt},
{Linea3,Violet},
}

\pgfplotscreateplotcyclelist{Lista2}{
{Linea1,Red,mark=*, mark size=1.35pt},
{Linea2,Red,mark=*, mark size=1.35pt},
{Linea3,Red},
{Linea1,Green,mark=triangle*, mark size=2pt},
{Linea2,Green,mark=triangle*, mark size=2pt},
{Linea3,Green},
{Linea1,Cyan,mark=square*, mark size=1.35pt},
{Linea2,Cyan,mark=square*, mark size=1.35pt},
{Linea3,Cyan},
{Linea1,Violet,mark=diamond*, mark size=2pt},
{Linea2,Violet,mark=diamond*, mark size=2pt},
{Linea3,Violet},
}

\pgfplotscreateplotcyclelist{Lista1P}{
{Linea2,Red,mark=*, mark size=1.35pt},
{Linea1,Red,mark=*, mark size=1.35pt},
{Linea2,Green,mark=triangle*, mark size=2pt},
{Linea1,Green,mark=triangle*, mark size=2pt},
{Linea2,Cyan,mark=square*, mark size=1.35pt},
{Linea1,Cyan,mark=square*, mark size=1.35pt},
{Linea2,Violet,mark=diamond*, mark size=2pt},
{Linea1,Violet,mark=diamond*, mark size=2pt},
}

\def \DATAFILE {tabErrorConvNdof.csv}

\begin{figure}[htbp]
\centering
\hspace*{\fill}
\begin{subfigure}[t]{0.4\linewidth}
\centering
\begin{tikzpicture}[
trim axis left
]
\begin{loglogaxis}[
cycle list name=Lista1,
yticklabel style={text width=2.125em, align=right},
width=\linewidth,
height=\linewidth,
xlabel={$N_{\mathrm{dof}}$},
ymax=10,
ymin=0.5e-11,
xminorticks=false,
yminorticks=false,
ylabel={$\| U - U_{\bh}^p \|_{L^2(Q_T)}/\| U \|_{L^2(Q_T)}$},
legend columns=2,
legend style={nodes={scale=0.85, transform shape},at={(0.5,1.05)},anchor=south},
xmajorgrids=true,
ymajorgrids=true,
legend entries={,$p=1$,,$O(N_{\mathrm{dof}}^{-1})$,,$p=2$,,\phantom{1.}$O(N_{\mathrm{dof}}^{-3/2})$,,$p=3$,,$O(N_{\mathrm{dof}}^{-2})$,,$p=4$,,\phantom{1.}$O(N_{\mathrm{dof}}^{-5/2})$}]
\addplot table [x=NdofP1, y=L2P1Ufp, col sep=comma] {\DATAFILE};
\addplot table [x=NdofP1, y=L2P1Ufo, col sep=comma] {\DATAFILE};
\addplot table [x=NdofP1, y=L2P1UIs, col sep=comma] {\DATAFILE};
\addplot table [x=NdofP1, y expr=0.3*\thisrow{NdofL2P1U}, col sep=comma] {\DATAFILE};

\addplot table [x=NdofP2, y=L2P2Ufp, col sep=comma] {\DATAFILE};
\addplot table [x=NdofP2, y=L2P2Ufo, col sep=comma] {\DATAFILE};
\addplot table [x=NdofP2, y=L2P2UIs, col sep=comma] {\DATAFILE};
\addplot table [x=NdofP2, y expr=0.3*\thisrow{NdofL2P2U}, col sep=comma] {\DATAFILE};

\addplot table [x=NdofP3, y=L2P3Ufp, col sep=comma] {\DATAFILE};
\addplot table [x=NdofP3, y=L2P3Ufo, col sep=comma] {\DATAFILE};
\addplot table [x=NdofP3, y=L2P3UIs, col sep=comma] {\DATAFILE};
\addplot table [x=NdofP3, y expr=0.3*\thisrow{NdofL2P3U}, col sep=comma] {\DATAFILE};

\addplot table [x=NdofP4, y=L2P4Ufp, col sep=comma] {\DATAFILE};
\addplot table [x=NdofP4, y=L2P4Ufo, col sep=comma] {\DATAFILE};
\addplot table [x=NdofP4, y=L2P4UIs, col sep=comma] {\DATAFILE};
\addplot table [x=NdofP4, y expr=0.2*\thisrow{NdofL2P4U}, col sep=comma] {\DATAFILE};
\end{loglogaxis}
\end{tikzpicture}
\end{subfigure}
\hfill
\begin{subfigure}[t]{0.4\linewidth}
\centering
\begin{tikzpicture}[
trim axis right
]
\begin{loglogaxis}[
cycle list name=Lista1,
yticklabel style={text width=2.125em, align=right},
width=\linewidth,
height=\linewidth,
xlabel={$N_{\mathrm{dof}}$},
ymax=10,
ymin=0.5e-11,
xminorticks=false,
yminorticks=false,
ylabel={$| U - U_{\bh}^p |_{H^1(Q_T)}/| U |_{H^1(Q_T)}$},
legend columns=2,
legend style={nodes={scale=0.85, transform shape},at={(0.5,1.05)},anchor=south},
xmajorgrids=true,
ymajorgrids=true,
legend entries=
{,$p=1$,,\phantom{1.}$O(N_{\mathrm{dof}}^{-1/2})$,,$p=2$,,\phantom{}$O(N_{\mathrm{dof}}^{-1})$,,$p=3$,,\phantom{1.}$O(N_{\mathrm{dof}}^{-3/2})$,,$p=4$,,\phantom{}$O(N_{\mathrm{dof}}^{-2})$}]
\addplot table [x=NdofP1, y=H1sP1Ufp, col sep=comma] {\DATAFILE};
\addplot table [x=NdofP1, y=H1sP1Ufo, col sep=comma] {\DATAFILE};
\addplot table [x=NdofP1, y=H1sP1UIs, col sep=comma] {\DATAFILE};
\addplot table [x=NdofP1, y expr=0.3*\thisrow{NdofH1sP1U}, col sep=comma] {\DATAFILE};

\addplot table [x=NdofP2, y=H1sP2Ufp, col sep=comma] {\DATAFILE};
\addplot table [x=NdofP2, y=H1sP2Ufo, col sep=comma] {\DATAFILE};
\addplot table [x=NdofP2, y=H1sP2UIs, col sep=comma] {\DATAFILE};
\addplot table [x=NdofP2, y expr=0.3*\thisrow{NdofH1sP2U}, col sep=comma] {\DATAFILE};

\addplot table [x=NdofP3, y=H1sP3Ufp, col sep=comma] {\DATAFILE};
\addplot table [x=NdofP3, y=H1sP3Ufo, col sep=comma] {\DATAFILE};
\addplot table [x=NdofP3, y=H1sP3UIs, col sep=comma] {\DATAFILE};
\addplot table [x=NdofP3, y expr=0.3*\thisrow{NdofH1sP3U}, col sep=comma] {\DATAFILE};

\addplot table [x=NdofP4, y=H1sP4Ufp, col sep=comma] {\DATAFILE};
\addplot table [x=NdofP4, y=H1sP4Ufo, col sep=comma] {\DATAFILE};
\addplot table [x=NdofP4, y=H1sP4UIs, col sep=comma] {\DATAFILE};
\addplot table [x=NdofP4, y expr=0.3*\thisrow{NdofH1sP4U}, col sep=comma] {\DATAFILE};
\end{loglogaxis}
\end{tikzpicture}
\end{subfigure}
\hspace*{\fill}
\\
\hspace*{\fill}
\begin{subfigure}[t]{0.4\linewidth}
\centering
\begin{tikzpicture}[
trim axis left
]
\begin{loglogaxis}[
cycle list name=Lista2,
yticklabel style={text width=2.125em, align=right},
width=\linewidth,
height=\linewidth,
xlabel={$N_{\mathrm{dof}}$},
ymax=10,
ymin=0.5e-11,
xminorticks=false,
yminorticks=false,
ylabel={$\| V - V_{\bh}^p \|_{L^2(Q_T)}/\| V \|_{L^2(Q_T)}$},
xmajorgrids=true,
ymajorgrids=true
]
\addplot table [x=NdofP1, y=L2P1Vfp, col sep=comma] {\DATAFILE};
\addplot table [x=NdofP1, y=L2P1Vfo, col sep=comma] {\DATAFILE};
\addplot table [x=NdofP1, y expr=0.3*\thisrow{NdofL2P1V},, col sep=comma] {\DATAFILE};

\addplot table [x=NdofP2, y=L2P2Vfp, col sep=comma] {\DATAFILE};
\addplot table [x=NdofP2, y=L2P2Vfo, col sep=comma] {\DATAFILE};
\addplot table [x=NdofP2, y expr=0.3*\thisrow{NdofL2P2V}, col sep=comma] {\DATAFILE};

\addplot table [x=NdofP3, y=L2P3Vfp, col sep=comma] {\DATAFILE};
\addplot table [x=NdofP3, y=L2P3Vfo, col sep=comma] {\DATAFILE};
\addplot table [x=NdofP3, y expr=0.3*\thisrow{NdofL2P3V}, col sep=comma] {\DATAFILE};

\addplot table [x=NdofP4, y=L2P4Vfp, col sep=comma] {\DATAFILE};
\addplot table [x=NdofP4, y=L2P4Vfo, col sep=comma] {\DATAFILE};
\addplot table [x=NdofP4, y expr=0.3*\thisrow{NdofL2P4V}, col sep=comma] {\DATAFILE};
\end{loglogaxis}
\end{tikzpicture}
\end{subfigure}
\hfill
\begin{subfigure}[t]{0.4\linewidth}
\centering
\begin{tikzpicture}[
trim axis right
]
\begin{loglogaxis}[
cycle list name=Lista2,
yticklabel style={text width=2.125em, align=right},
width=\linewidth,
height=\linewidth,
xlabel={$N_{\mathrm{dof}}$},
ymax=10,
ymin=0.5e-11,
xminorticks=false,
yminorticks=false,
ylabel={$| V - V_{\bh}^p |_{H^1(Q_T)}/| V |_{H^1(Q_T)}$},
xmajorgrids=true,
ymajorgrids=true
]
\addplot table [x=NdofP1, y=H1sP1Vfp, col sep=comma] {\DATAFILE};
\addplot table [x=NdofP1, y=H1sP1Vfo, col sep=comma] {\DATAFILE};
\addplot table [x=NdofP1, y expr=0.3*\thisrow{NdofH1sP1V}, col sep=comma] {\DATAFILE};

\addplot table [x=NdofP2, y=H1sP2Vfp, col sep=comma] {\DATAFILE};
\addplot table [x=NdofP2, y=H1sP2Vfo, col sep=comma] {\DATAFILE};
\addplot table [x=NdofP2, y expr=0.3*\thisrow{NdofH1sP2V}, col sep=comma] {\DATAFILE};

\addplot table [x=NdofP3, y=H1sP3Vfp, col sep=comma] {\DATAFILE};
\addplot table [x=NdofP3, y=H1sP3Vfo, col sep=comma] {\DATAFILE};
\addplot table [x=NdofP3, y expr=0.3*\thisrow{NdofH1sP3V}, col sep=comma] {\DATAFILE};

\addplot table [x=NdofP4, y=H1sP4Vfp, col sep=comma] {\DATAFILE};
\addplot table [x=NdofP4, y=H1sP4Vfo, col sep=comma] {\DATAFILE};
\addplot table [x=NdofP4, y expr=0.3*\thisrow{NdofH1sP4V}, col sep=comma] {\DATAFILE};
\end{loglogaxis}
\end{tikzpicture}
\end{subfigure}
\hspace*{\fill}
\caption{Example 1. First row: relative errors between the exact position $U$ and the discrete one $U_{\bh}^p$ provided by the unconditionally stable method~\eqref{eq:4} (continuous lines \rule[0.5ex]{0.45cm}{0.5pt}~), the unconditionally stable method in~\cite{FrenchPeterson1996} (dashed lines \rule[0.5ex]{0.15cm}{0.5pt}~\rule[0.5ex]{0.15cm}{0.5pt}~\rule[0.5ex]{0.15cm}{0.4pt}~), and the stabilized method devised in~\cite{FraschiniLoliMoiolaSangalli2023} (dash-dotted lines \rule[0.5ex]{0.15cm}{0.5pt}~$\cdot$~\rule[0.5ex]{0.15cm}{0.4pt}~). \\ Second row: relative errors between the exact velocity $V$ and the discrete one $V_{\bh}^p$ provided by method~\eqref{eq:4} (continuous lines \rule[0.5ex]{0.45cm}{0.5pt}~) and~\cite{FrenchPeterson1996} (dashed lines \rule[0.5ex]{0.15cm}{0.5pt}~\rule[0.5ex]{0.15cm}{0.5pt}~\rule[0.5ex]{0.15cm}{0.4pt}~).
\\The errors are plotted against the total number of DOFs $N_{\mathrm{dof}}$, and the mesh sizes satisfy $h_t = 5 h_x$.}
\label{fig:StabComparisonHP}
\end{figure}

\subsubsection{Example 2. Highly oscillatory solutions} \label{subsec6.2}
In this example, we investigate the robustness of our method for high-frequency oscillations. As in~\cite[Section 5.1.2]{FraschiniLoliMoiolaSangalli2023}, we consider the following exact solution of the acoustic wave equation in one space dimension:
\begin{equation}  \label{u2_ex}
    U(x,t)= \sin(k \pi x) \sin(k \pi t) \quad \text{for } (x,t) \in Q_T=(0,1) \times (0,2)
\end{equation}
for different wave numbers~$k \in \N$, and constant wave velocity~$c=1$. 

Let~$\sharp \lambda$ be the number of space wave lengths contained in the space domain~$\Omega = (0,1)$, i.e.,~$\sharp \lambda := \tfrac{k}{2}$. Figure~\ref{fig:Ldof1234} shows the relative errors of our discretization~\eqref{eq:4} in the usual space--time~$L^2$ norm and~$H^1$ seminorm plotted against the number of space DOFs per wave length~$N_\text{dof}/\sharp \lambda$ for different wave numbers~$k \in \{1,2,4,8,16\}$. As in~\cite[Section 5.1.2]{FraschiniLoliMoiolaSangalli2023}, for~$p>1$, the number of DOFs per wave length that is required to reach a given accuracy is independent of the wave number~$k$. 

\tikzstyle{Linea1}=[thick,dashed]
\tikzstyle{Linea2}=[thick,mark=*]

\def\HeiFact{.78}

\pgfplotscreateplotcyclelist{Lista0}{
	{Linea2,Red},
	{Linea2,Green},
	{Linea2,Cyan},
	{Linea2,Violet},
	{Linea2,BurntOrange},
	{Linea2,Brown},
	{Linea2,Magenta},
	{Linea2,Blue},
}

\pgfplotscreateplotcyclelist{Lista1}{
	{thick,mark=*,Red},
	{thick,mark=triangle*,Green},
	{thick,mark=square*,Cyan},
	{thick,mark=diamond*,Violet},
	{thick,mark=pentagon*,BurntOrange},
	{dashed,black,thick},
	{Linea2,Brown},
	{Linea2,Magenta},
	{Linea2,Blue},
}

\def \DATAFILEa {tabErrorConvLdofP1.csv}
\def \DATAFILEb {tabErrorConvLdofP2.csv}
\def \DATAFILEc {tabErrorConvLdofP3.csv}
\def \DATAFILEd {tabErrorConvLdofP4.csv}

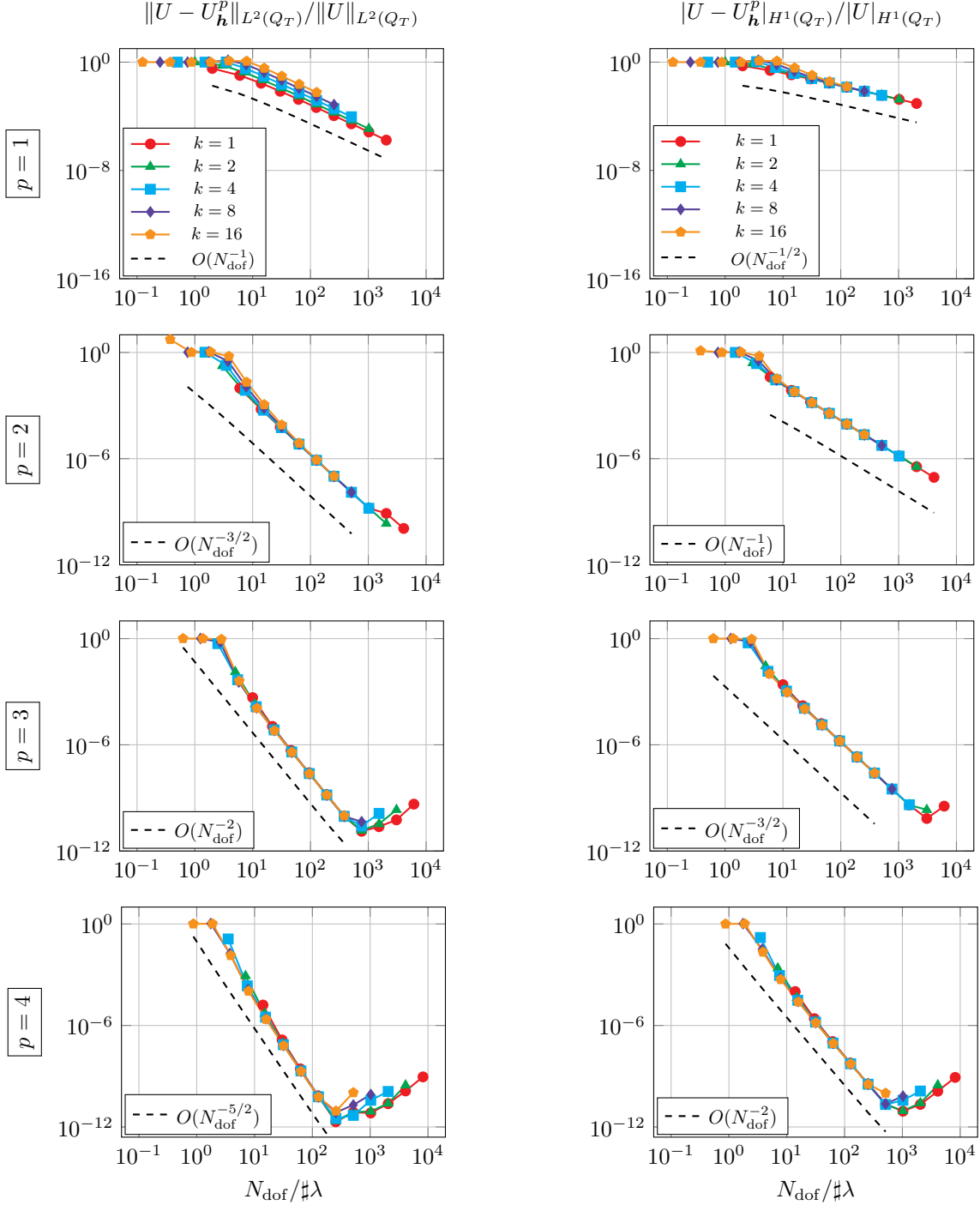
\begin{figure}[p]
	\centering 
\hspace*{\fill}
	\begin{subfigure}[t]{0.4\linewidth}
		\centering
		\begin{tikzpicture}[trim axis left]
			\begin{loglogaxis}
				[cycle list name=Lista1,
				width=\linewidth,
				height=\HeiFact\linewidth,
				xminorticks=false, yminorticks=false,
				ylabel = {\fbox{$p=1$}},
				title={$\| U - U_{\bh}^p \|_{L^2(Q_T)}/\| U \|_{L^2(Q_T)}$},
				legend columns=1,
				legend pos=south east,
				legend style={nodes={scale=0.85, transform shape},at={(0.01,0.01)},anchor=south west},
				xmajorgrids=true, ymajorgrids=true,
				ymin={0.0000000000000001}, ymax=10, 
				xmin=0.05,
				xmax=20000,
                legend style={font=\small},
				legend entries={$k=1$,$k=2$,$k=4$,$k=8$,\phantom{.}$k=16$,\phantom{1..}$O(N_{\mathrm{dof}}^{-1})$}]
				\addplot table [x=nDofsK1, y=L2P1K1, col sep=comma] {\DATAFILEa};
				\addplot table [x=nDofsK2, y=L2P1K2, col sep=comma] {\DATAFILEa};
				\addplot table [x=nDofsK4, y=L2P1K4, col sep=comma] {\DATAFILEa};
				\addplot table [x=nDofsK8, y=L2P1K8, col sep=comma] {\DATAFILEa};
				\addplot table [x=nDofsK16, y=L2P1K16, col sep=comma] {\DATAFILEa};
				\addplot[dashed,black,thick] table [x=nDofsK1, y expr=0.04*\thisrow{h2L2}, col sep=comma] {\DATAFILEa};
			\end{loglogaxis}
		\end{tikzpicture}
	\end{subfigure}\hfill
	\begin{subfigure}[t]{0.4\linewidth}
		\centering
		\begin{tikzpicture}[
			trim axis right]
			\begin{loglogaxis}[
				cycle list name=Lista1,
				width=\linewidth, 
				height=\HeiFact\linewidth,
				xminorticks=false, yminorticks=false,
				title={$| U - U_{\bh}^p |_{H^1(Q_T)}/| U |_{H^1(Q_T)}$ },
				legend columns=1,
				legend pos=south east,
				legend style={nodes={scale=0.85, transform shape},at={(0.01,0.01)},anchor=south west},
				xmajorgrids=true, ymajorgrids=true,
				ymin={0.0000000000000001}, ymax=10,
				xmin=0.05,
				xmax=20000,
       legend style={font=\small},
				legend entries={$k=1$,$k=2$,$k=4$,$k=8$,\phantom{..}$k=16$,\phantom{..12}$O(N_{\mathrm{dof}}^{-1/2})$}]
				\addplot table [x=nDofsK1, y=H1sP1K1, col sep=comma] {\DATAFILEa};
				\addplot table [x=nDofsK2, y=H1sP1K2, col sep=comma] {\DATAFILEa};
				\addplot table [x=nDofsK4, y=H1sP1K4, col sep=comma] {\DATAFILEa};
				\addplot table [x=nDofsK8, y=H1sP1K8, col sep=comma] {\DATAFILEa};
				\addplot table [x=nDofsK16, y=H1sP1K16, col sep=comma] {\DATAFILEa};
				\addplot[dashed,black,thick] table [x=nDofsK1, y expr=0.04*\thisrow{h1H1s}, col sep=comma] {\DATAFILEa};
			\end{loglogaxis}
		\end{tikzpicture}
	\end{subfigure}
	\hspace*{\fill}
 \vspace{0.3cm}
	\\
	\hspace*{\fill}
	\begin{subfigure}[t]{0.4\linewidth}
		\centering
		\begin{tikzpicture}[
			trim axis left]
			\begin{loglogaxis}[
				cycle list name=Lista1,
				width=\linewidth,
				height=\HeiFact\linewidth,
				xminorticks=false, yminorticks=false,
				ylabel = {\fbox{$p=2$}},
				legend columns=1,
				legend pos=south east,
				legend style={nodes={scale=0.85, transform shape},at={(0.01,0.01)},anchor=south west},
				xmajorgrids=true, ymajorgrids=true,
				ymin={0.000000000001}, ymax=10,
				xmin=0.05,
				xmax=20000, 
				legend entries={,,,,,$O(N_{\mathrm{dof}}^{-3/2})$}]
				\addplot table [x=nDofsK1, y=L2P2K1, col sep=comma] {\DATAFILEb};
				\addplot table [x=nDofsK2, y=L2P2K2, col sep=comma] {\DATAFILEb};
				\addplot table [x=nDofsK4, y=L2P2K4, col sep=comma] {\DATAFILEb};
				\addplot table [x=nDofsK8, y=L2P2K8, col sep=comma] {\DATAFILEb};
				\addplot table [x=nDofsK16, y=L2P2K16, col sep=comma] {\DATAFILEb};
				\addplot[dashed,black,thick] table [x=nDofsK8, y expr=0.000000000001+0.5*\thisrow{h3L2}, col sep=comma] {\DATAFILEb};
			\end{loglogaxis}
		\end{tikzpicture}
	\end{subfigure}\hfill
	\begin{subfigure}[t]{0.4\linewidth}
		\centering
		\begin{tikzpicture}[
			trim axis right]
			\begin{loglogaxis}[
				cycle list name=Lista1,
				width=\linewidth,
				height=\HeiFact\linewidth,
				xminorticks=false, yminorticks=false,
				legend columns=1,
				legend pos=south east,
				legend style={nodes={scale=0.85, transform shape},at={(0.01,0.01)},anchor=south west},
				xmajorgrids=true, ymajorgrids=true,
				ymin={0.000000000001}, ymax=10,
				xmin=0.05,
				xmax=20000,
				legend entries={,,,,,$O(N_{\mathrm{dof}}^{-1})$}]
				\addplot table [x=nDofsK1, y=H1sP2K1, col sep=comma] {\DATAFILEb};
				\addplot table [x=nDofsK2, y=H1sP2K2, col sep=comma] {\DATAFILEb};
				\addplot table [x=nDofsK4, y=H1sP2K4, col sep=comma] {\DATAFILEb};
				\addplot table [x=nDofsK8, y=H1sP2K8, col sep=comma] {\DATAFILEb};
				\addplot table [x=nDofsK16, y=H1sP2K16, col sep=comma] {\DATAFILEb};
				\addplot[dashed,black,thick] table [x=nDofsK1, y expr=0.01*\thisrow{h2H1s}, col sep=comma] {\DATAFILEb};
			\end{loglogaxis}
		\end{tikzpicture}
	\end{subfigure}
\hspace*{\fill}
\vspace{0.3cm}
\\
\hspace*{\fill}
	\begin{subfigure}[t]{0.4\linewidth}
		\centering
		\begin{tikzpicture}[
			trim axis left]
			\begin{loglogaxis}[
				cycle list name=Lista1,
				width=\linewidth,
				height=\HeiFact\linewidth,
				xminorticks=false, yminorticks=false,
				xmin=0.05,
				xmax=20000,
				ylabel = {\fbox{$p=3$}},
				legend columns=1,
				legend pos=south east,
				legend style={nodes={scale=0.85, transform shape},at={(0.01,0.01)},anchor=south west},
				xmajorgrids=true, ymajorgrids=true,
				ymin={0.000000000001}, ymax=10,
				legend entries={,,,,,$O(N_{\mathrm{dof}}^{-2})$}]
				\addplot table [x=nDofsK1, y=L2P3K1, col sep=comma] {\DATAFILEc};
				\addplot table [x=nDofsK2, y=L2P3K2, col sep=comma] {\DATAFILEc};
				\addplot table [x=nDofsK4, y=L2P3K4, col sep=comma] {\DATAFILEc};
				\addplot table [x=nDofsK8, y=L2P3K8, col sep=comma] {\DATAFILEc};
				\addplot table [x=nDofsK16, y=L2P3K16, col sep=comma] {\DATAFILEc};
				\addplot[dashed,black,thick] table [x=nDofsK16, y expr=0.0000000000005+0.005*\thisrow{h4L2}, col sep=comma] {\DATAFILEc};
			\end{loglogaxis}
		\end{tikzpicture}
	\end{subfigure}\hfill
	\begin{subfigure}[t]{0.4\linewidth}
		\centering
		\begin{tikzpicture}[
			trim axis right]
			\begin{loglogaxis}[
				cycle list name=Lista1,
				width=\linewidth, height=\HeiFact\linewidth,
				xminorticks=false, yminorticks=false,
				legend columns=1,
				legend pos=south east,
				legend style={nodes={scale=0.85, transform shape},at={(0.01,0.01)},anchor=south west},
				xmajorgrids=true, ymajorgrids=true,
				xmin=0.05,
				xmax=20000,
				ymin={0.000000000001}, ymax=10,
				legend entries={,,,,,$O(N_{\mathrm{dof}}^{-3/2})$}]
				\addplot table [x=nDofsK1, y=H1sP3K1, col sep=comma] {\DATAFILEc};
				\addplot table [x=nDofsK2, y=H1sP3K2, col sep=comma] {\DATAFILEc};
				\addplot table [x=nDofsK4, y=H1sP3K4, col sep=comma] {\DATAFILEc};
				\addplot table [x=nDofsK8, y=H1sP3K8, col sep=comma] {\DATAFILEc};
				\addplot table [x=nDofsK16, y=H1sP3K16, col sep=comma] {\DATAFILEc};
				\addplot[dashed,black,thick] table [x=nDofsK16, y expr=0.00000000000001+0.1*\thisrow{h3H1s}, col sep=comma] {\DATAFILEc};
			\end{loglogaxis}
		\end{tikzpicture}
	\end{subfigure}
	\hspace*{\fill}
 \vspace{0.3cm}
	\\
	\hspace*{\fill}
	\begin{subfigure}[t]{0.4\linewidth}
		\centering
		\begin{tikzpicture}[
			trim axis left]
			\begin{loglogaxis}[
				cycle list name=Lista1,
				width=\linewidth, height=\HeiFact\linewidth,
				xlabel={$N_{\mathrm{dof}}/\sharp \lambda$},
				xminorticks=false, yminorticks=false,
				ylabel = {\fbox{$p=4$}},
				legend columns=1,
				legend pos=south east,
				legend style={nodes={scale=0.85, transform shape},at={(0.01,0.01)},anchor=south west},
				xmajorgrids=true, ymajorgrids=true,
				ymin={2.5e-13}, ymax=10,
				xmin=0.05,
				xmax=20000,
				legend entries={,,,,,$O(N_{\mathrm{dof}}^{-5/2})$}]
				\addplot table [x=nDofsK1, y=L2P4K1, col sep=comma] {\DATAFILEd};
				\addplot table [x=nDofsK2, y=L2P4K2, col sep=comma] {\DATAFILEd};
				\addplot table [x=nDofsK4, y=L2P4K4, col sep=comma] {\DATAFILEd};
				\addplot table [x=nDofsK8, y=L2P4K8, col sep=comma] {\DATAFILEd};
				\addplot table [x=nDofsK16, y=L2P4K16, col sep=comma] {\DATAFILEd};
				\addplot[dashed,black,thick] table [x=nDofsK16, y expr=0.00000000000001+0.000002*\thisrow{h5L2}, col sep=comma] {\DATAFILEd};
			\end{loglogaxis}
		\end{tikzpicture}
	\end{subfigure}\hfill
	\begin{subfigure}[t]{0.4\linewidth}
		\centering
		\begin{tikzpicture}[
			trim axis right]
			\begin{loglogaxis}[
				cycle list name=Lista1,
				width=\linewidth,
				height=\HeiFact\linewidth,
				xlabel={$N_{\mathrm{dof}}/\sharp \lambda$},
				xminorticks=false,yminorticks=false,
				legend columns=1,
				legend pos=south east,
				legend style={nodes={scale=0.85, transform shape},at={(0.01,0.01)},anchor=south west},
				xmajorgrids=true, ymajorgrids=true,
				ymin={2.5e-13}, ymax=10, 
				xmin=0.05,
				xmax=20000,
				legend entries={,,,,,$O(N_{\mathrm{dof}}^{-2})$}]
				\addplot table [x=nDofsK1, y=H1sP4K1, col sep=comma] {\DATAFILEd};
				\addplot table [x=nDofsK2, y=H1sP4K2, col sep=comma] {\DATAFILEd};
				\addplot table [x=nDofsK4, y=H1sP4K4, col sep=comma] {\DATAFILEd};
				\addplot table [x=nDofsK8, y=H1sP4K8, col sep=comma] {\DATAFILEd};
				\addplot table [x=nDofsK16, y=H1sP4K16, col sep=comma] {\DATAFILEd};
				\addplot[dashed,black,thick] table [x=nDofsK16, y expr=0.0000000000001+0.0005*\thisrow{h4H1s}, col sep=comma] {\DATAFILEd};
			\end{loglogaxis}
		\end{tikzpicture}
	\end{subfigure}
	\hspace*{\fill}
	\caption{Example 2. Relative errors of~\eqref{eq:4} plotted against the number of space DOFs per wave length $N_{\mathrm{dof}}/\sharp \lambda$, at different wave numbers $k$. $L^2$ norms are shown on the left, $H^1$ seminorms on the right. Rows 1 to 4 correspond to $p=1$ to $p=4$.
		The exact solution is defined in \eqref{u2_ex}.}
	\label{fig:Ldof1234}
\end{figure}

\subsubsection{Example 3. More general boundary conditions: a scattering problem} \label{subsec6.3}
To verify the effectiveness of our discretization for impedance boundary conditions, we employ method~\eqref{eq:matrix_system} to solve the scattering problem that has been addressed in~\cite[Section 5.1.3]{FraschiniLoliMoiolaSangalli2023}. Let~$\Omega \subset{\R^2}$ be the bi-dimensional space domain illustrated in Figure~\ref{fig::scattering_space-domain}. The wave problem under consideration is problem~\eqref{eq:mod_prob_ext} with constant wave velocity~$c=1$, space--time domain~$Q_T=\Omega \times (0,6)$, impedance parameter~$\vartheta =1$, and homogeneous Dirichlet, Neumann, and impedance, boundary conditions, imposed, respectively, on the boundaries~$\Gamma_D$,~$\Gamma_N$, and~$\Gamma_R$ specified in Figure~\ref{fig::scattering_space-domain}. The initial conditions are homogeneous, and the source term reads as
\begin{equation*}
    F(\bx,t) =\cos(2\pi t) \Psi(t) \Psi\left(\frac{\| \bx - \bx_{C}\|}{0.4}\right) \quad \text{for } (\bx,t) \in Q_T,
\end{equation*}
where~$\bx_C = (2, 0)^\top$, and ~$\Psi:\R \rightarrow \R$ denotes the bump function defined as
\begin{equation}\label{eq::bump}
	\Psi(s)= \begin{cases}
		\begin{aligned}
			&e^{1+1/(s^2-1)} & & s \in (-1,1),\\
			&0 & & \text{otherwise}.
		\end{aligned}
	\end{cases}
\end{equation}
In Figure \ref{fig:ErrorScattering} we report the relative $L^2$ and $H^1$ error of the numerical solution $\bU_h^p$ obtained with various mesh sizes and splines degrees, compared with a reference solution obtained with the stabilized method proposed in \cite{FraschiniLoliMoiolaSangalli2023}. Also in this case, not included in our theoretical analysis, we observe optimal order of convergence.
\tikzstyle{LineaD}=[ultra thick,blue]
            \tikzstyle{LineaN}=[ultra thick,dgreen]
            \tikzstyle{LineaR}=[ultra thick,dred]
            \tikzstyle{Linea2}=[gray]
        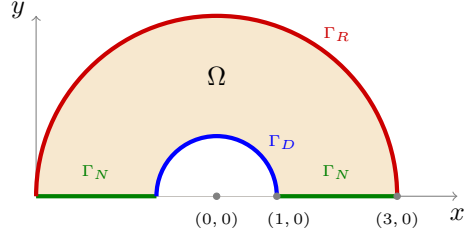
\begin{figure}[ht]
        \centering
        \begin{tikzpicture}[scale=0.8]
        \draw[Linea2,->] (-3,0) -- (4,0);
        \node at (4,-0.3) {$x$};
        \draw[Linea2,->] (-3,0) -- (-3,3);
        \node at (-3.3,3.1) {$y$};

        \fill[champagne] (3,0) arc(0:180:3);
        \fill[white] (1,0) arc(0:180:1);
        \draw[LineaN,-] (1,0) -- (3,0);
        \draw[LineaN,-] (-3,0) -- (-1,0);
        \draw[LineaD,-] (1,0) arc(0:180:1);
        \draw[LineaR,-] (3,0) arc(0:180:3);
        \filldraw [gray] (0,0) circle (1.5pt);
        \node at (0,-0.4) {\tiny{${(0,0)}$}};
        \filldraw [gray] (3,0) circle (1.5pt);
        \node at (3,-0.4) {\tiny{${(3,0)}$}};
        \filldraw [gray] (1,0) circle (1.5pt);
        \node at (1.2,-0.4) {\tiny{${(1,0)}$}};

        \node at (0,2) {$\Omega$};
        \node at (1.1,0.9) {\textcolor{blue}{\tiny{$\Gamma_D$}}};
        \node at (2,0.4) {\textcolor{dgreen}{\tiny{$\Gamma_N$}}};
        \node at (-2,0.4) {\textcolor{dgreen}{\tiny{$\Gamma_N$}}};
        \node at (2,2.7) {\textcolor{dred}{\tiny{$\Gamma_R$}}};
        \end{tikzpicture}
        \caption{Example 3. Space domain of the scattering problem of Section~\ref{subsec6.3}.} \label{fig::scattering_space-domain}
        \end{figure}

\tikzstyle{Linea1}=[thick,dashed]
\tikzstyle{Linea2}=[thick]

\pgfplotscreateplotcyclelist{Lista1}{
{Linea2,Red,mark=*},
{Linea1,Red},
{Linea2,Green,mark=triangle*},
{Linea2,Cyan,mark=square*},
{Linea2,Violet,mark=diamond*},
{Linea1,Violet},
}
\pgfplotscreateplotcyclelist{Lista2}{
{Linea2,Red,mark=*},
{Linea1,Red},
{Linea2,Green,mark=triangle*},
{Linea2,Cyan,mark=square*},
{Linea2,Violet,mark=diamond*},
}

\pgfplotscreateplotcyclelist{Lista1rs}{
{Linea2,Red,mark=*},
{Linea1,Red},
{Linea2,Green,mark=triangle*},
{Linea1,Green},
{Linea2,Cyan,mark=square*},
{Linea1,Cyan},
{Linea2,Violet,mark=diamond*},
{Linea1,Violet},
}

\def \DATAFILE {tabErrorScatteringConv.csv}

\begin{figure}[htbp]
\centering
\hspace*{\fill}
\begin{subfigure}[t]{0.4\linewidth}
\centering
\begin{tikzpicture}[
trim axis left
]
\begin{loglogaxis}[
yticklabel style={text width=2.5em, align=right},
cycle list name=Lista1rs,
width=\linewidth,
height=\linewidth,
xlabel={$h_t$},
ymin=1e-6,
ymax=5e0,
xminorticks=false,
yminorticks=false,
ylabel={$\| U - U_{\bh}^p \|_{L^2(Q_T)}/\| U \|_{L^2(Q_T)}$},
legend columns=2,
legend style={nodes={scale=0.85, transform shape},at={(0.5,1.05)},anchor=south},
xmajorgrids=true,
ymajorgrids=true,
legend entries={$p=1$,$O(h_t^2)$,$p=2$,$O(h_t^3)$,$p=3$,$O(h_t^4)$,$p=4$,$O(h_t^5)$}]
\addplot table [x=ht, y=L2P1U, col sep=comma] {\DATAFILE};
\addplot table [x=ht, y=h2L2U, col sep=comma] {\DATAFILE};
\addplot table [x=ht, y=L2P2U, col sep=comma] {\DATAFILE};
\addplot table [x=ht, y=h3L2U, col sep=comma] {\DATAFILE};
\addplot table [x=ht, y=L2P3U, col sep=comma] {\DATAFILE};
\addplot table [x=ht, y=h4L2U, col sep=comma] {\DATAFILE};
\addplot table [x=ht, y=L2P4U, col sep=comma] {\DATAFILE};
\addplot table [x=ht, y=h5L2U, col sep=comma] {\DATAFILE};
\end{loglogaxis}
\end{tikzpicture}
\end{subfigure}
\hfill
\begin{subfigure}[t]{0.4\linewidth}
\centering
\begin{tikzpicture}[
trim axis right
]
\begin{loglogaxis}[
yticklabel style={text width=2.5em, align=right},
cycle list name=Lista1rs,
width=\linewidth,
height=\linewidth,
xlabel={$h_t$},
ymin=1e-4,
ymax=5e0,
xminorticks=false,
yminorticks=false,
ylabel={$| U - U_{\bh}^p |_{H^1(Q_T)}/| U |_{H^1(Q_T)}$},
legend columns=2,
legend style={nodes={scale=0.85, transform shape},at={(0.5,1.05)},anchor=south},
xmajorgrids=true,
ymajorgrids=true,
legend entries={$p=1$,$O(h_t)$,$p=2$,$O(h_t^2)$,$p=3$,$O(h_t^3)$,$p=4$,$O(h_t^4)$}]
\addplot table [x=ht, y=H1sP1U, col sep=comma] {\DATAFILE};
\addplot table [x=ht, y=h1H1sU, col sep=comma] {\DATAFILE};
\addplot table [x=ht, y=H1sP2U, col sep=comma] {\DATAFILE};
\addplot table [x=ht, y=h2H1sU, col sep=comma] {\DATAFILE};
\addplot table [x=ht, y=H1sP3U, col sep=comma] {\DATAFILE};
\addplot table [x=ht, y=h3H1sU, col sep=comma] {\DATAFILE};
\addplot table [x=ht, y=H1sP4U, col sep=comma] {\DATAFILE};
\addplot table [x=ht, y=h4H1sU, col sep=comma] {\DATAFILE};
\end{loglogaxis}
\end{tikzpicture}
\end{subfigure}
\hspace*{\fill}
\caption{Example 3. Relative errors of method~\eqref{eq:matrix_system} solving the scattering problem presented in Section~\ref{subsec6.3}. These errors correspond to the position $U$ and are plotted against the time mesh size $h_t$, where $h_t \approx h_{\bx}$.}
\label{fig:ErrorScattering}
\end{figure}
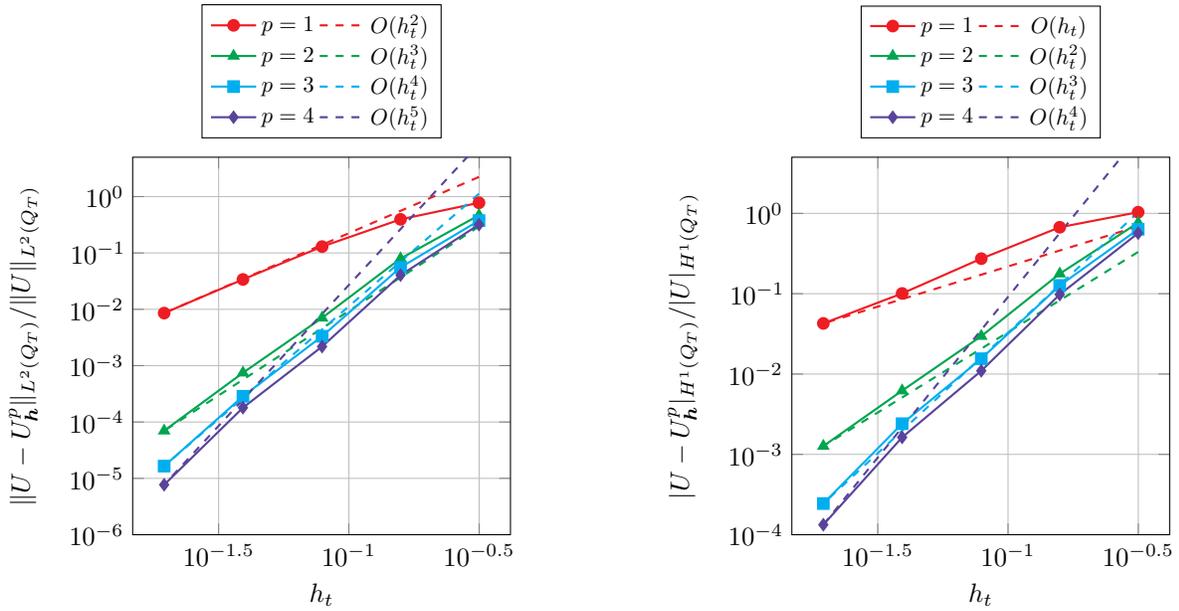

\subsubsection{Example 4. Singular solution} \label{subsec6.4}
As in~\cite[Section 5.1.5]{FraschiniLoliMoiolaSangalli2023}, we test the accuracy of our space--time method~\eqref{eq:matrix_system} approximating the singular solution of the acoustic wave equation~\eqref{eq:mod_prob_ext} with the following piecewise constant wave velocity:
\begin{equation*}
	c(x,t)= \begin{cases}
	1 & 0\le x<\displaystyle{\frac12}, \\
	2 & \displaystyle{\frac12} \le x\le 1,
	\end{cases}
        \quad \quad \text{for } (x,t) \in Q_T=(0,1) \times (0,1),
\end{equation*}
homogeneous Neumann boundary conditions ($\partial\Omega=\Gamma_N$), homogeneous source term, and initial data~$U_0(x)=\Psi(5x-1)$ and~$V_0(x)=-5 \Psi'(5x-1)$, where~$\Psi$ is the smooth bump defined in~\eqref{eq::bump}. For the explicit expression of the exact solution of this problem, we refer to~\cite[Eq. (5.5)]{FraschiniLoliMoiolaSangalli2023}. 

Let~$|\cdot|_{c,H^1(Q_T)}$ be the weighted~$H^1(Q_T)$ seminorm 
\begin{equation*}
    |w|^{2}_{c,H^1(Q_T)} := \int_{Q_T} \left( \left|\partial_t w(x, t)\right|^2 + c^2(x)\left|\partial_x w(x,t)\right|^2 \right) \dd x \, \dd t \quad \text{for } w \in H^1(Q_T).
\end{equation*}
Let us consider a discretization with space--time maximal regularity splines except at~$x=\nicefrac{1}{2}$, where we impose only~$C^0$-continuity. Figure~\ref{fig:CdiscBumpStabConv2in1} shows the relative errors in the~$L^2(Q_T)$ norm and the weighted~$H^1(Q_T)$ seminorm of the space--time method~\eqref{eq:matrix_system} based on these spaces, against the time mesh size~$h_t = h_{x}$. As one can observe, optimal convergence rates are achieved.

\tikzstyle{Linea1}=[thick,dashed]
\tikzstyle{Linea2}=[thick]

\pgfplotscreateplotcyclelist{Lista1}{
{Linea2,Red,mark=*},
{Linea1,Red},
{Linea2,Green,mark=triangle*},
{Linea2,Cyan,mark=square*},
{Linea2,Violet,mark=diamond*},
{Linea1,Violet},
}
\pgfplotscreateplotcyclelist{Lista2}{
{Linea2,Red,mark=*},
{Linea1,Red},
{Linea2,Green,mark=triangle*},
{Linea2,Cyan,mark=square*},
{Linea2,Violet,mark=diamond*},
}

\pgfplotscreateplotcyclelist{Lista1rs}{
{Linea2,Red,mark=*},
{Linea1,Red},
{Linea2,Green,mark=triangle*},
{Linea1,Green},
{Linea2,Cyan,mark=square*},
{Linea1,Cyan},
{Linea2,Violet,mark=diamond*},
{Linea1,Violet},
}

\def \DATAFILE {tabErrorCdiscConv.csv}

\begin{figure}[htbp]
\centering
\hspace*{\fill}
\begin{subfigure}[t]{0.4\linewidth}
\centering
\begin{tikzpicture}[
trim axis left
]
\begin{loglogaxis}[
yticklabel style={text width=2.5em, align=right},
cycle list name=Lista1rs,
width=\linewidth,
height=\linewidth,
xlabel={$h_t$},
ymin=1e-10,
ymax=5e0,
xminorticks=false,
yminorticks=false,
ylabel={$\| U - U_{\bh}^p \|_{L^2(Q_T)}/\| U \|_{L^2(Q_T)}$},
legend columns=2,
legend style={nodes={scale=0.85, transform shape},at={(0.5,1.05)},anchor=south},
xmajorgrids=true,
ymajorgrids=true,
legend entries={$p=1$,$O(h_t^2)$,$p=2$,$O(h_t^3)$,$p=3$,$O(h_t^4)$,$p=4$,$O(h_t^5)$},
skip coords between index={5}{6}]
\addplot table [x=ht, y=L2P1U, col sep=comma] {\DATAFILE};
\addplot table [x=ht, y expr=0.1*\thisrow{h2L2U}, col sep=comma] {\DATAFILE};
\addplot table [x=ht, y=L2P2U, col sep=comma] {\DATAFILE};
\addplot table [x=ht, y expr=0.5*\thisrow{h3L2U}, col sep=comma] {\DATAFILE};
\addplot table [x=ht, y=L2P3U, col sep=comma] {\DATAFILE};
\addplot table [x=ht, y expr=0.005*\thisrow{h4L2U}, col sep=comma] {\DATAFILE};
\addplot table [x=ht, y=L2P4U, col sep=comma] {\DATAFILE};
\addplot table [x=ht, y expr=0.0001*\thisrow{h5L2U}, col sep=comma] {\DATAFILE};
\end{loglogaxis}
\end{tikzpicture}
\end{subfigure}
\hfill
\begin{subfigure}[t]{0.4\linewidth}
\centering
\begin{tikzpicture}[
trim axis right
]
\begin{loglogaxis}[
yticklabel style={text width=2.5em, align=right},
cycle list name=Lista1rs,
width=\linewidth,
height=\linewidth,
xlabel={$h_t$},
ymin=1e-10,
ymax=5e0,
xminorticks=false,
yminorticks=false,
ylabel={$| U - U_{\bh}^p |_{c,H^1(Q_T)}/| U |_{c,H^1(Q_T)}$},
legend columns=2,
legend style={nodes={scale=0.85, transform shape},at={(0.5,1.05)},anchor=south},
xmajorgrids=true,
ymajorgrids=true,
legend entries={$p=1$,$O(h_t)$,$p=2$,$O(h_t^2)$,$p=3$,$O(h_t^3)$,$p=4$,$O(h_t^4)$},
skip coords between index={5}{6}]
\addplot table [x=ht, y=H1sP1U, col sep=comma] {\DATAFILE};
\addplot table [x=ht, y expr=0.5*\thisrow{h1H1sU}, col sep=comma] {\DATAFILE};
\addplot table [x=ht, y=H1sP2U, col sep=comma] {\DATAFILE};
\addplot table [x=ht, y expr=0.5*\thisrow{h2H1sU}, col sep=comma] {\DATAFILE};
\addplot table [x=ht, y=H1sP3U, col sep=comma] {\DATAFILE};
\addplot table [x=ht, y expr=0.5*\thisrow{h3H1sU},, col sep=comma] {\DATAFILE};
\addplot table [x=ht, y=H1sP4U, col sep=comma] {\DATAFILE};
\addplot table [x=ht, y expr=0.05*\thisrow{h4H1sU},, col sep=comma] {\DATAFILE};
\end{loglogaxis}
\end{tikzpicture}
\end{subfigure}
\hspace*{\fill}
\\
\hspace*{\fill}
\begin{subfigure}[t]{0.4\linewidth}
\centering
\begin{tikzpicture}[
trim axis left
]
\begin{loglogaxis}[
yticklabel style={text width=2.5em, align=right},
cycle list name=Lista1rs,
width=\linewidth,
height=\linewidth,
xlabel={$h_t$},
ymin=1e-10,
ymax=5e0,
xminorticks=false,
yminorticks=false,
ylabel={$\| V - V_{\bh}^p \|_{L^2(Q_T)}/\| V \|_{L^2(Q_T)}$},
xmajorgrids=true,
ymajorgrids=true,
skip coords between index={5}{6}]
\addplot table [x=ht, y=L2P1V, col sep=comma] {\DATAFILE};
\addplot table [x=ht, y expr=0.1*\thisrow{h2L2V}, col sep=comma] {\DATAFILE};
\addplot table [x=ht, y=L2P2V, col sep=comma] {\DATAFILE};
\addplot table [x=ht, y expr=\thisrow{h3L2V}, col sep=comma] {\DATAFILE};
\addplot table [x=ht, y=L2P3V, col sep=comma] {\DATAFILE};
\addplot table [x=ht, y expr=0.5*\thisrow{h4L2V}, col sep=comma] {\DATAFILE};
\addplot table [x=ht, y=L2P4V, col sep=comma] {\DATAFILE};
\addplot table [x=ht, y expr=0.01*\thisrow{h5L2V}, col sep=comma] {\DATAFILE};
\end{loglogaxis}
\end{tikzpicture}
\end{subfigure}
\hfill
\vspace{0.2cm}
\begin{subfigure}[t]{0.4\linewidth}
\centering
\begin{tikzpicture}[
trim axis right
]
\begin{loglogaxis}[
yticklabel style={text width=2.5em, align=right},
cycle list name=Lista1rs,
width=\linewidth,
height=\linewidth,
xlabel={$h_t$},
ymin=1e-10,
ymax=5e0,
xminorticks=false,
yminorticks=false,
ylabel={$| V - V_{\bh}^p |_{c,H^1(Q_T)}/| V |_{c,H^1(Q_T)}$},
xmajorgrids=true,
ymajorgrids=true,
skip coords between index={5}{6}]
\addplot table [x=ht, y=H1sP1V, col sep=comma] {\DATAFILE};
\addplot table [x=ht, y expr=0.6*\thisrow{h1H1sV}, col sep=comma] {\DATAFILE};
\addplot table [x=ht, y=H1sP2V, col sep=comma] {\DATAFILE};
\addplot table [x=ht, y expr=0.8*\thisrow{h2H1sV}, col sep=comma] {\DATAFILE};
\addplot table [x=ht, y=H1sP3V, col sep=comma] {\DATAFILE};
\addplot table [x=ht, y expr=\thisrow{h3H1sV}, col sep=comma] {\DATAFILE};
\addplot table [x=ht, y=H1sP4V, col sep=comma] {\DATAFILE};
\addplot table [x=ht, y expr=0.1*\thisrow{h4H1sV}, col sep=comma] {\DATAFILE};
\end{loglogaxis}
\end{tikzpicture}
\end{subfigure}
\hspace*{\fill}
\caption{Example 4. Relative errors of method~\eqref{eq:matrix_system} solving the wave problem with piecewise-constant velocity presented in Section~\ref{subsec6.4}. The discretization is based on space--time maximal regularity splines except at $x=\nicefrac{1}{2}$, where only $C^0$-continuity is imposed. First row: relative errors between the exact position $U$ and the discrete one $U_{\bh}^p$. Second row: relative errors between the exact velocity $V$ and the discrete one $V_{\bh}^p$. These errors are plotted against the time mesh size $h_t$, which satisfies $h_t = h_x$.}
\label{fig:CdiscBumpStabConv2in1}
\end{figure}
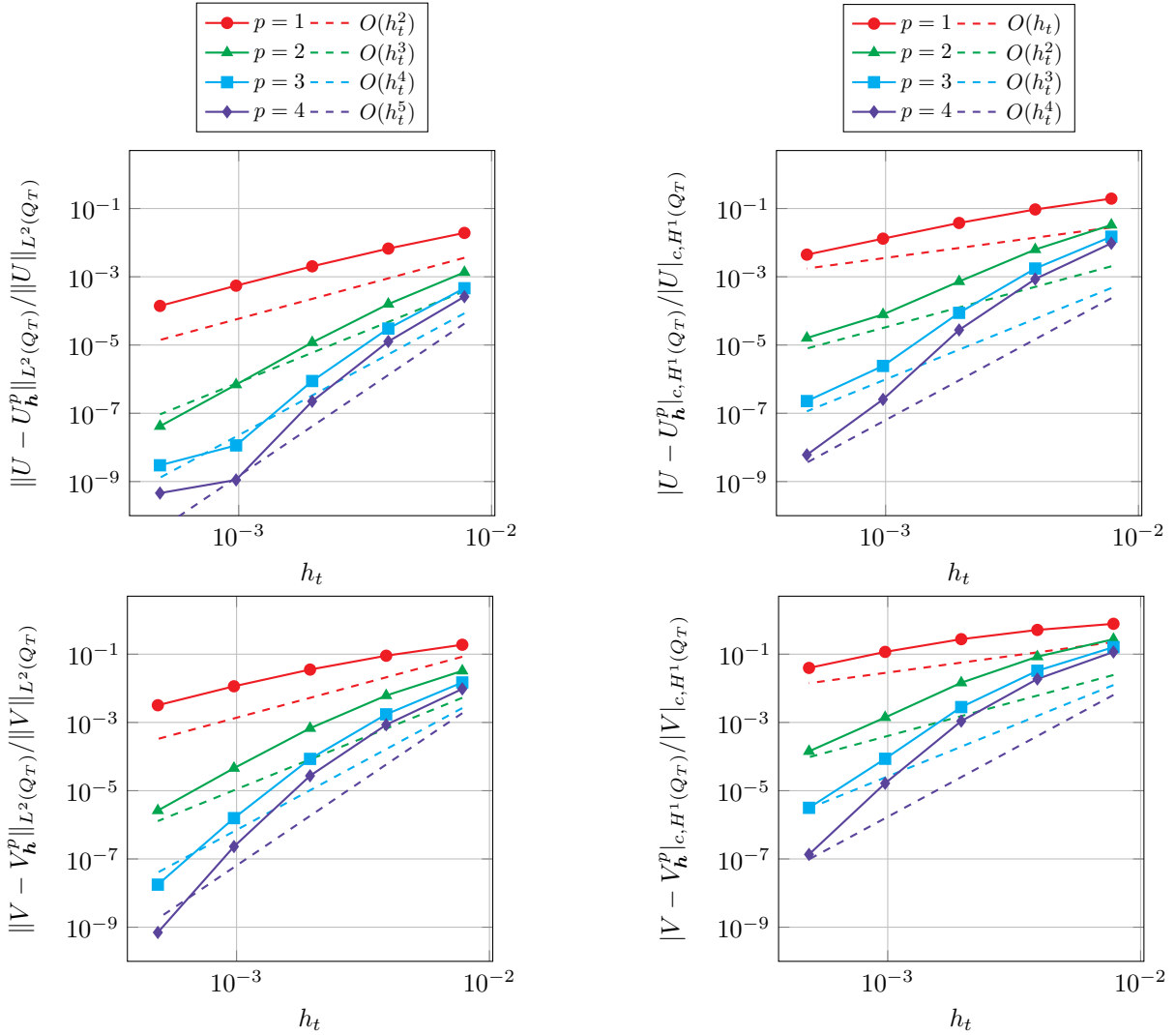

\subsubsection{Example 5. Energy conservation}\label{subsec6.5}
We test the accuracy of the discrete energy associated with our method by solving the same problem that has been addressed in~\cite[Remark 4.2.36]{Zank2020} and~\cite[Section 5.2]{FraschiniLoliMoiolaSangalli2023}. Specifically, we solve wave problem~\eqref{eq:2} with one-dimensional space domain~$\Omega = (0,1)$, wave velocity~$c=1$, and exact solution
\begin{equation}\label{uex_energy}
	U(x,t)=\left( \cos(\pi t) + \sin(\pi t) \right) \sin(\pi x) \quad \text{for } (x,t) \in Q_T = \Omega \times (0,10),
\end{equation}
whose constant energy~$E(t)\equiv\displaystyle{\frac{\pi^2}{2}}$, where
\begin{equation*}
    E(t) := \tfrac12 \|V(\cdot,t)\|^2_{L^2(\Omega)} + \tfrac12 \|\nabla_x U(\cdot,t)\|^2_{L^2(\Omega)} \quad \text{for } t \in [0,10].
\end{equation*}
Let~$U_{\bh}^p$ and~$V_{\bh}^p$ be, respectively, the discrete position and velocity provided by method~\eqref{eq:matrix_system}. The discrete energy associated with these solutions is 
\begin{equation*}
    E_{\bh}^p(t) := \tfrac12 \|V_{\bh}^p(\cdot,t)\|^2_{L^2(\Omega)} + \tfrac12 \|\nabla_x U_{\bh}^p(\cdot,t)\|^2_{L^2(\Omega)} \quad \text{for } t \in [0,10].
\end{equation*}
Figure~\ref{fig:EnergyConservation} shows the time evolution of the relative errors between the exact and discrete energy with space mesh size~$h_x = 2^{-7}$, and time mesh size~$h_t = h_x$. The relative error does not grow with time and is bounded by~$10^{-2p}$, where~$p$ is the spline degree in space and time.

\tikzstyle{Linea1}=[thick,dashed]
\tikzstyle{Linea2}=[thick, only marks, mark=oplus]
\tikzstyle{Linea3}=[thick, only marks, mark=halfcircle]

\pgfplotscreateplotcyclelist{Lista1}{
	{Linea1,Red},
	{Linea1,Green},
	{Linea1,Cyan},
	{Linea1,Violet},
	{Linea2,Red},
	{Linea2,Green},
	{Linea2,Cyan},
	{Linea2,Violet},
	{Linea3,Red},
	{Linea3,Green},
	{Linea3,Cyan},
	{Linea3,Violet},
}

\def \DATAFILE {tabEnergyFO.csv}

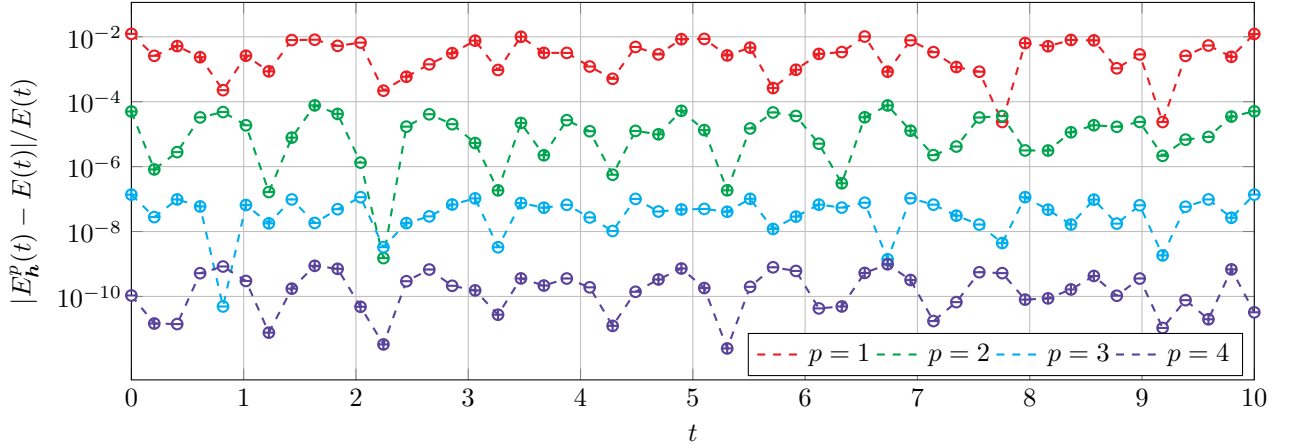
\begin{figure}[htbp]
	\centering 
	\begin{tikzpicture}[
        ]
		\begin{semilogyaxis}[
			cycle list name=Lista1,
			width=\textwidth,
			height=0.4\textwidth,
			xlabel={$t$},
			xmin=0,
			xmax=10,
			xminorticks=false,
			yminorticks=false,
			ytick={0.0000000001,0.00000001,0.000001,0.0001,0.01}, 
			ylabel={$|E_{\bh}^p(t)-E(t)|/E(t)$},
			legend columns=4,
			legend pos=south east,
			legend style={at={(0.99,0.01)},anchor=south east},
			grid=major,
			legend entries={$p=1$,$p=2$,$p=3$,$p=4$}]
			\addplot table [x=t, y=Ep1, col sep=comma] {\DATAFILE};
			\addplot table [x=t, y=Ep2, col sep=comma] {\DATAFILE};
			\addplot table [x=t, y=Ep3, col sep=comma] {\DATAFILE};
			\addplot table [x=t, y=Ep4, col sep=comma] {\DATAFILE};
			\addplot table [x=t, y=Ep1, col sep=comma,restrict expr to domain={\thisrow{Ep1Sign}}{1:1}] {\DATAFILE};
			\addplot table [x=t, y=Ep2, col sep=comma,restrict expr to domain={\thisrow{Ep2Sign}}{1:1}] {\DATAFILE};
			\addplot table [x=t, y=Ep3, col sep=comma,restrict expr to domain={\thisrow{Ep3Sign}}{1:1}] {\DATAFILE};
			\addplot table [x=t, y=Ep4, col sep=comma,restrict expr to domain={\thisrow{Ep4Sign}}{1:1}] {\DATAFILE};
			\addplot table [x=t, y=Ep1, col sep=comma,restrict expr to domain={\thisrow{Ep1Sign}}{-1:-1}] {\DATAFILE};
			\addplot table [x=t, y=Ep2, col sep=comma,restrict expr to domain={\thisrow{Ep2Sign}}{-1:-1}] {\DATAFILE};
			\addplot table [x=t, y=Ep3, col sep=comma,restrict expr to domain={\thisrow{Ep3Sign}}{-1:-1}] {\DATAFILE};
			\addplot table [x=t, y=Ep4, col sep=comma,restrict expr to domain={\thisrow{Ep4Sign}}{-1:-1}] {\DATAFILE};
		\end{semilogyaxis}
	\end{tikzpicture}
	\caption{Example 5. Time evolution of the energy relative error for the problem with solution \eqref{uex_energy}. The marker ``$\oplus$'' denotes time instants when $E_{\bh}^p \geq E$, while ``$\ominus$'' stands for $E_{\bh}^p \leq E$.}
	\label{fig:EnergyConservation}
\end{figure}

\subsubsection{Example 6. Dispersion properties}\label{subsec6.7}
To investigate the numerical dispersion of our space--time method, we approximate the~$C^0$ tent profile and~$C^\infty$ bump profile that has been considered in~\cite[Section 5.3]{FraschiniLoliMoiolaSangalli2023}. Given the space--time cylinder~$Q_T = (0,1) \times (0,2)$, the wave velocity~$c=1$, and the source~$F=0$, we solve two wave propagation problems with periodic boundary conditions, and initial data, respectively,
\begin{equation}\label{eq:tent}
    U_0(x) = (1-|4x-1|)\chi_{[0,\nicefrac{1}{2}]}(x), \quad V_0(x) = -4\chi_{[0,\nicefrac{1}{4}]}(x) + 4 \chi_{[\nicefrac{1}{4},\nicefrac{1}{2}]}(x) \quad \text{for } x \in [0,1],
\end{equation}
and
\begin{equation}\label{eq:bump}
    U_0(x) = \Psi(4x-1)\chi_{[0,\nicefrac{1}{2}]}(x), \quad V_0(x) = -4\Psi'(4x-1)\chi_{[0,\nicefrac{1}{2}]}(x) \quad  \text{for~} x \in [0,1].
\end{equation}
The former describes the~$C^0$ tent profile, the latter the~$C^\infty$ bump profile, where~$\Psi$ is the smooth bump defined in~\eqref{eq::bump}.

We compare how our spline-based, unconditionally stable method, the FEM-based unconditionally stable method of~\cite{FrenchPeterson1996}, and the spline-based stabilized method of~\cite{FraschiniLoliMoiolaSangalli2023} deform a periodic wave. Figure~\ref{fig:dispersionTentErrP} shows how the spatial~$L^2$ norm and~$H^1$ seminorm errors at the final time depend on the polynomial degree~$p$. The method proposed in this paper performs slightly better than the methods of \cite{FraschiniLoliMoiolaSangalli2023} and \cite{FrenchPeterson1996}. Finally, following~\cite[Section 5.3]{FraschiniLoliMoiolaSangalli2023}, Figure~\ref{fig:dispersionBumpFourier} shows the time evolution of the phase error of the largest (in magnitude) Fourier coefficients of  the solution, defined as
\begin{equation} \label{phase_error}
	\Phi_{\bh,n}^p(t) := \left|\arg\left(\frac{c_n(t)}{c_{\bh,n}^p(t)} \cdot \frac{|c_{\bh,n}^p(t)|}{|c_{n}(t)|}\right)\right|,
\end{equation}
where~$c_n$ and~$c_{\bh,n}^p$ denotes the~$n$-th complex Fourier coefficients of the exact solution and the numerical one, respectively. Note that the largest coefficients (in magnitude) are $c_1,c_2,c_3,c_5$ for the datum \eqref{eq:tent} and $c_1,c_2,c_3,c_4$ for \eqref{eq:bump}. As with what was observed in \cite{FraschiniLoliMoiolaSangalli2023} for the stabilized method, here too
for $p>1$ the phase error grows moderately with time, and no particular differences from that method are noticed.

\tikzstyle{Linea1}=[thick,dashed, mark options={solid}]
\tikzstyle{Linea2}=[thick,mark options={solid}]
\tikzstyle{Linea3}=[thick,dotted,mark options={solid}]

\pgfplotscreateplotcyclelist{Lista1}{
	{Linea2,Red,mark=*},
	{Linea1,Green,mark=triangle*},
        {Linea3,Cyan,mark=diamond*},
	{Linea2,Green,mark=triangle*},
	{Linea1,Green,mark=triangle*},
	{Linea2,Cyan,mark=square*},
	{Linea1,Cyan,mark=square*},
	{Linea2,Violet,mark=diamond*},
	{Linea1,Violet,mark=diamond*},
}

\def \DATAFILEA {tabErrorDispersionTent.csv}
\def \DATAFILEB {tabErrorDispersionBump.csv}

\begin{figure}[htbp]
	\centering
	\begin{tikzpicture}
		\begin{groupplot}[
			group style={
				group name=myPlots,
				group size=2 by 1,
				xlabels at=edge bottom,
				ylabels at=edge left,
				horizontal sep=3cm,
				vertical sep=2cm,
			},
			width=\linewidth
			]
			
			\nextgroupplot[ymode=log,
				cycle list name=Lista1,
				width=0.45\linewidth,
				height=0.45\linewidth,
				xlabel={$p$},
				ymin=1e-6,
				ymax=5e-1,
				xminorticks=false,
				yminorticks=false,
				ylabel={$\| U(\cdot,T) - U_h(\cdot,T) \|_{L^2(\Omega)}/\| U(\cdot,T) \|_{L^2(\Omega)}$},
				xmajorgrids=true,
				ymajorgrids=true,
				xtick={1,2,3,4}]
				\addplot table [x=p, y=L2fo, col sep=comma] {\DATAFILEA};
				\addplot table [x=p, y=L2fp, col sep=comma] {\DATAFILEA};
                \addplot table [x=p, y=L2Is, col sep=comma] {\DATAFILEA};
			
			\nextgroupplot[ymode=log,
				cycle list name=Lista1,
				width=0.45\linewidth,
				height=0.45\linewidth,
				xlabel={$p$},
				ymin=1e-3,
				ymax=5e-0,
				xminorticks=false,
				yminorticks=false,
				ylabel={$| U(\cdot,T) - U_{\bh}(\cdot,T) |_{H^1(\Omega)}/| U(\cdot,T) |_{H^1(\Omega)}$},
				xmajorgrids=true,
				ymajorgrids=true,
				xtick={1,2,3,4},	
				legend to name=Legend7,
				legend columns=1, 
				legend entries={\susm,\fusm,\ssm}]
				\addplot table [x=p, y=H1sfo, col sep=comma] {\DATAFILEA};
				\addplot table [x=p, y=H1sfp, col sep=comma] {\DATAFILEA};
                \addplot table [x=p, y=H1sIs, col sep=comma] {\DATAFILEA};
                \end{groupplot}
  \node (myLegend7) at ($(myPlots c1r1.north)!0.5!(myPlots c2r1.north)$)[above=0.5cm]{\ref{Legend7}};
	\end{tikzpicture}
\\
\begin{tikzpicture}
		\begin{groupplot}[
			group style={
				group name=myPlots,
				group size=2 by 1,
				xlabels at=edge bottom,
				ylabels at=edge left,
				horizontal sep=3cm,
				vertical sep=2cm,
			},
			width=\linewidth
			]
			
			\nextgroupplot[ymode=log,
				cycle list name=Lista1,
				width=0.45\linewidth,
				height=0.45\linewidth,
				xlabel={$p$},
				ymin=1e-6,
				ymax=5e-1,
				xminorticks=false,
				yminorticks=false,
				ylabel={$\| U(\cdot,T) - U_{\bh}(\cdot,T) \|_{L^2(\Omega)}/\| U(\cdot,T) \|_{L^2(\Omega)}$},
				xmajorgrids=true,
				ymajorgrids=true,
				xtick={1,2,3,4}]
				\addplot table [x=p, y=L2fo, col sep=comma] {\DATAFILEB};
				\addplot table [x=p, y=L2fp, col sep=comma] {\DATAFILEB};
                \addplot table [x=p, y=L2Is, col sep=comma] {\DATAFILEB};
			
			\nextgroupplot[ymode=log,
				cycle list name=Lista1,
				width=0.45\linewidth,
				height=0.45\linewidth,
				xlabel={$p$},
				ymin=1e-3,
				ymax=5e-1,
				xminorticks=false,
				yminorticks=false,
				ylabel={$| U(\cdot,T) - U_{\bh}(\cdot,T) |_{H^1(\Omega)}/| U(\cdot,T) |_{H^1(\Omega)}$},
				xmajorgrids=true,
				ymajorgrids=true,
				xtick={1,2,3,4},	
				legend to name=Legend3,
				legend columns=1, 
				legend entries={\susm,\fusm,\ssm}]
				\addplot table [x=p, y=H1sfo, col sep=comma] {\DATAFILEB};
				\addplot table [x=p, y=H1sfp, col sep=comma] {\DATAFILEB};
                \addplot table [x=p, y=H1sIs, col sep=comma] {\DATAFILEB};
                \end{groupplot}
	\end{tikzpicture}

\caption{Example 6. Comparison between the relative errors at final time of all the methods, for the periodic problem of Section~\ref{subsec6.7} with initial data~\eqref{eq:tent} (first row), and with initial data~\eqref{eq:bump} (second row). For all the methods and all the spline degrees, $N_{\mathrm{dof}}=17\,424$ and $h_t \approx 2 h_x$.}
\label{fig:dispersionTentErrP}
\end{figure}
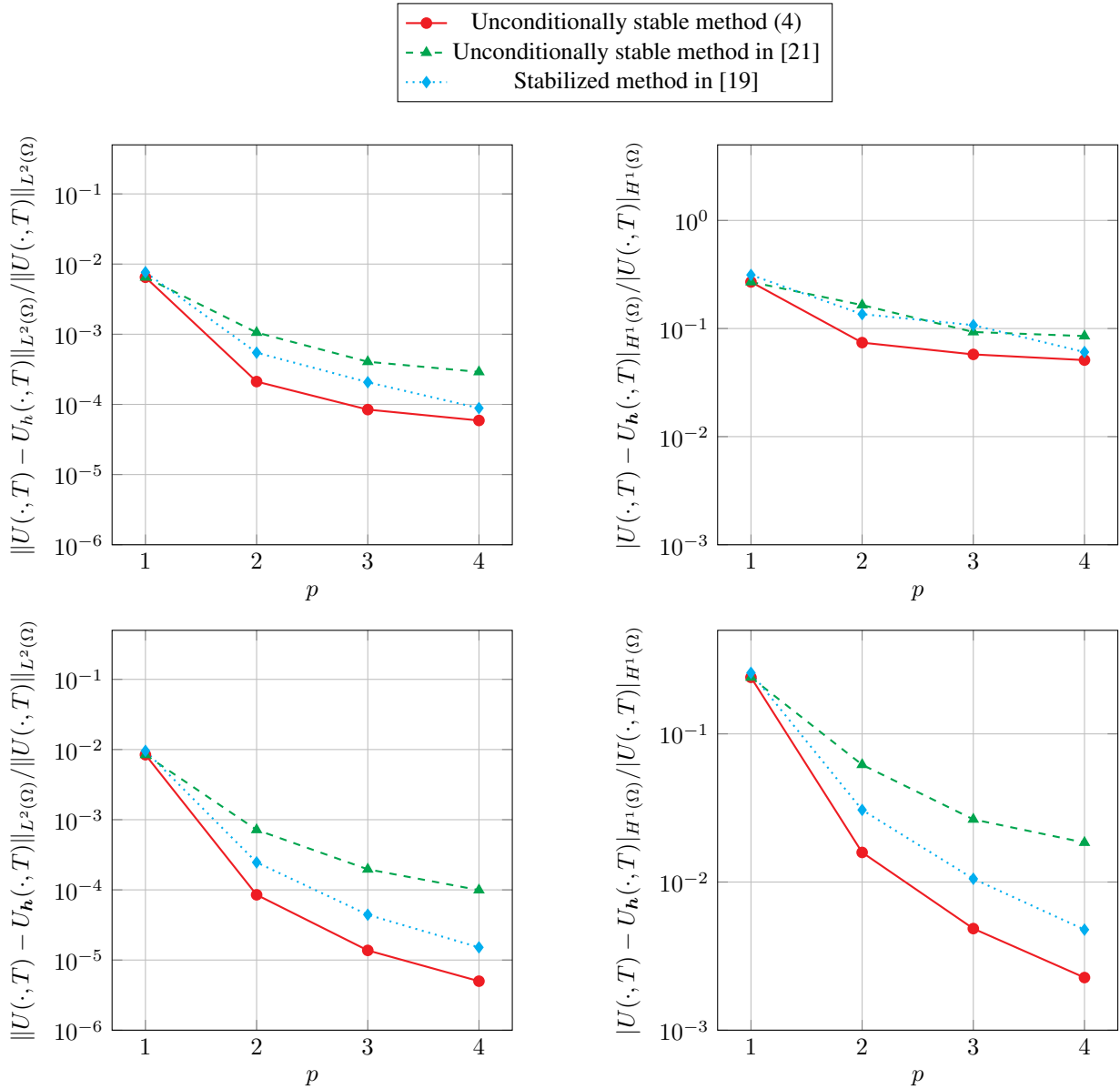

\tikzstyle{Linea1}=[thick]
\tikzstyle{Linea2}=[thick,mark=*]

\pgfplotscreateplotcyclelist{Lista1}{
	{Linea1,Red},
	{Linea1,Green},
	{Linea1,Cyan},
	{Linea1,Violet},
	{Linea1,BurntOrange},
	{Linea1,Brown},
	{Linea1,Magenta},
	{Linea1,Blue},
}

\def\DATAFILEFoaB{CnTableBumpFo_p1.csv}
\def\DATAFILEFobB{CnTableBumpFo_p2.csv}
\def\DATAFILEFocB{CnTableBumpFo_p3.csv}
\def\DATAFILEFodB{CnTableBumpFo_p4.csv}
\def\DATAFILEFoaT{CnTableTentFo_p1.csv}
\def\DATAFILEFobT{CnTableTentFo_p2.csv}
\def\DATAFILEFocT{CnTableTentFo_p3.csv}
\def\DATAFILEFodT{CnTableTentFo_p4.csv}

\def\HeiFac{0.4}

\begin{figure}[htbp]
	\centering
\begin{tikzpicture}
\begin{groupplot}[
	group style={
	group name=plots1,
	group size=2 by 4,
	xlabels at=edge bottom,
	ylabels at=edge left,
	horizontal sep=2cm,
	vertical sep=0.5cm,
	},
	width=\linewidth
	]
	
	\nextgroupplot[
	ymode=log,
	xmin=0,	
	xmax=2,
	ymin=0.00000001, 
	ymax=15,
	cycle list name=Lista1,
	width=0.475\linewidth,
	height=0.19\linewidth,
	xticklabels={\empty},
	ylabel={\fbox{\rotatebox{270}{$\Phi_{\bh,1}^p$}}},
	grid=both]
	\addplot table [x=phyT, y=c1, col sep=comma] {\DATAFILEFoaT};
	\addplot table [x=phyT, y=c1, col sep=comma] {\DATAFILEFobT};
	\addplot table [x=phyT, y=c1, col sep=comma] {\DATAFILEFocT};
	\addplot table [x=phyT, y=c1, col sep=comma] {\DATAFILEFodT};
	
	\nextgroupplot[
	ymode=log,
	xmin=0,	
	xmax=2,
	ymin=0.00000001, 
	ymax=15,
	ylabel={\fbox{\rotatebox{270}{$\Phi_{\bh,1}^p$}}},
        yticklabels={\empty},
	cycle list name=Lista1,
	width=0.475\linewidth,
	height=0.19\linewidth,
	xticklabels={\empty},
	grid=both]
	\addplot table [x=phyT, y=c1, col sep=comma] {\DATAFILEFoaB};
	\addplot table [x=phyT, y=c1, col sep=comma] {\DATAFILEFobB};
	\addplot table [x=phyT, y=c1, col sep=comma] {\DATAFILEFocB};
	\addplot table [x=phyT, y=c1, col sep=comma] {\DATAFILEFodB};
	
	\nextgroupplot[
	ymode=log,
	xmin=0,	
	xmax=2,
	ymin=0.00000001, 
	ymax=15,
	cycle list name=Lista1,
	width=0.475\linewidth,
	height=0.19\linewidth,
	ylabel={\fbox{\rotatebox{270}{$\Phi_{\bh,2}^p$}}},
	xticklabels={\empty},
	grid=both]
	\addplot table [x=phyT, y=c2, col sep=comma] {\DATAFILEFoaT};
	\addplot table [x=phyT, y=c2, col sep=comma] {\DATAFILEFobT};
	\addplot table [x=phyT, y=c2, col sep=comma] {\DATAFILEFocT};
	\addplot table [x=phyT, y=c2, col sep=comma] {\DATAFILEFodT};
	
	\nextgroupplot[
	ymode=log,
	xmin=0,	
	xmax=2,
	ymin=0.00000001, 
	ymax=15,
	yticklabels={\empty},
	cycle list name=Lista1,
	width=0.475\linewidth,
	height=0.19\linewidth,
	ylabel={\fbox{\rotatebox{270}{$\Phi_{\bh,2}^p$}}},
	grid=both]
	\addplot table [x=phyT, y=c2, col sep=comma] {\DATAFILEFoaB};
	\addplot table [x=phyT, y=c2, col sep=comma] {\DATAFILEFobB};
	\addplot table [x=phyT, y=c2, col sep=comma] {\DATAFILEFocB};
	\addplot table [x=phyT, y=c2, col sep=comma] {\DATAFILEFodB};
	
	\nextgroupplot[
	ymode=log,
	xmin=0,	
	xmax=2,
	ymin=0.00000001, 
	ymax=15,
	cycle list name=Lista1,
	width=0.475\linewidth,
	height=0.19\linewidth,
	xticklabels={\empty},
	ylabel={\fbox{\rotatebox{270}{$\Phi_{\bh,3}^p$}}},
	grid=both]
	\addplot table [x=phyT, y=c3, col sep=comma] {\DATAFILEFoaT};
	\addplot table [x=phyT, y=c3, col sep=comma] {\DATAFILEFobT};
	\addplot table [x=phyT, y=c3, col sep=comma] {\DATAFILEFocT};
	\addplot table [x=phyT, y=c3, col sep=comma] {\DATAFILEFodT};
	
	\nextgroupplot[
	ymode=log,
	xmin=0,	
	xmax=2,
	ymin=0.00000001, 
	ymax=15,
	yticklabels={\empty},
        ylabel={\fbox{\rotatebox{270}{$\Phi_{\bh,3}^p$}}},
	cycle list name=Lista1,
	width=0.475\linewidth,
	height=0.19\linewidth,
	xticklabels={\fbox{\rotatebox{270}{$\Phi_{\bh,3}^p$}}},
	grid=both]
	\addplot table [x=phyT, y=c3, col sep=comma] {\DATAFILEFoaB};
	\addplot table [x=phyT, y=c3, col sep=comma] {\DATAFILEFobB};
	\addplot table [x=phyT, y=c3, col sep=comma] {\DATAFILEFobB};
	\addplot table [x=phyT, y=c3, col sep=comma] {\DATAFILEFodB};
	
	\nextgroupplot[
	ymode=log,
	xlabel={$t$},
	xmin=0,	
	xmax=2,
	ymin=0.00000001, 
	ymax=15,
	cycle list name=Lista1,
	width=0.475\linewidth,
	height=0.19\linewidth,
	ylabel={\fbox{\rotatebox{270}{$\Phi_{\bh,5}^p$}}},
	grid=both]
	\addplot table [x=phyT, y=c5, col sep=comma] {\DATAFILEFoaT};
	\addplot table [x=phyT, y=c5, col sep=comma] {\DATAFILEFobT};
	\addplot table [x=phyT, y=c5, col sep=comma] {\DATAFILEFocT};
	\addplot table [x=phyT, y=c5, col sep=comma] {\DATAFILEFodT};
	
	\nextgroupplot[
	ymode=log,
	xlabel={$t$},
	xmin=0,	
	xmax=2,
	ymin=0.00000001, 
	ymax=15,
	ylabel={\fbox{\rotatebox{270}{$\Phi_{\bh,4}^p$}}},
        yticklabels={\empty},
	cycle list name=Lista1,
	width=0.475\linewidth,
	height=0.19\linewidth,
	legend to name=myLegendDisp,
         legend columns=4,
        legend entries={$p=1$,$p=2$,$p=3$,$p=4$},
	grid=both]
	\addplot table [x=phyT, y=c4, col sep=comma] {\DATAFILEFoaB};
	\addplot table [x=phyT, y=c4, col sep=comma] {\DATAFILEFobB};
	\addplot table [x=phyT, y=c4, col sep=comma] {\DATAFILEFocB};
	\addplot table [x=phyT, y=c4, col sep=comma] {\DATAFILEFodB};
	\end{groupplot}
 \node (legendDisp) at ($(plots1 c1r1.north)!0.5!(plots1 c2r1.north)$)[above=0.5cm]{\ref{myLegendDisp}};
\end{tikzpicture}
	
\caption{Example 6. Phase errors $\Phi_{\bh,i}^p$ (defined in \eqref{phase_error}) of the largest 4 Fourier coefficients for the periodic problem with initial conditions~\eqref{eq:tent} (first column) and~\eqref{eq:bump} (second column), approximated with $N_{\mathrm{dof}}=17\,424$ and $h_t \approx 2h_x$.}
\label{fig:dispersionBumpFourier}
\end{figure}
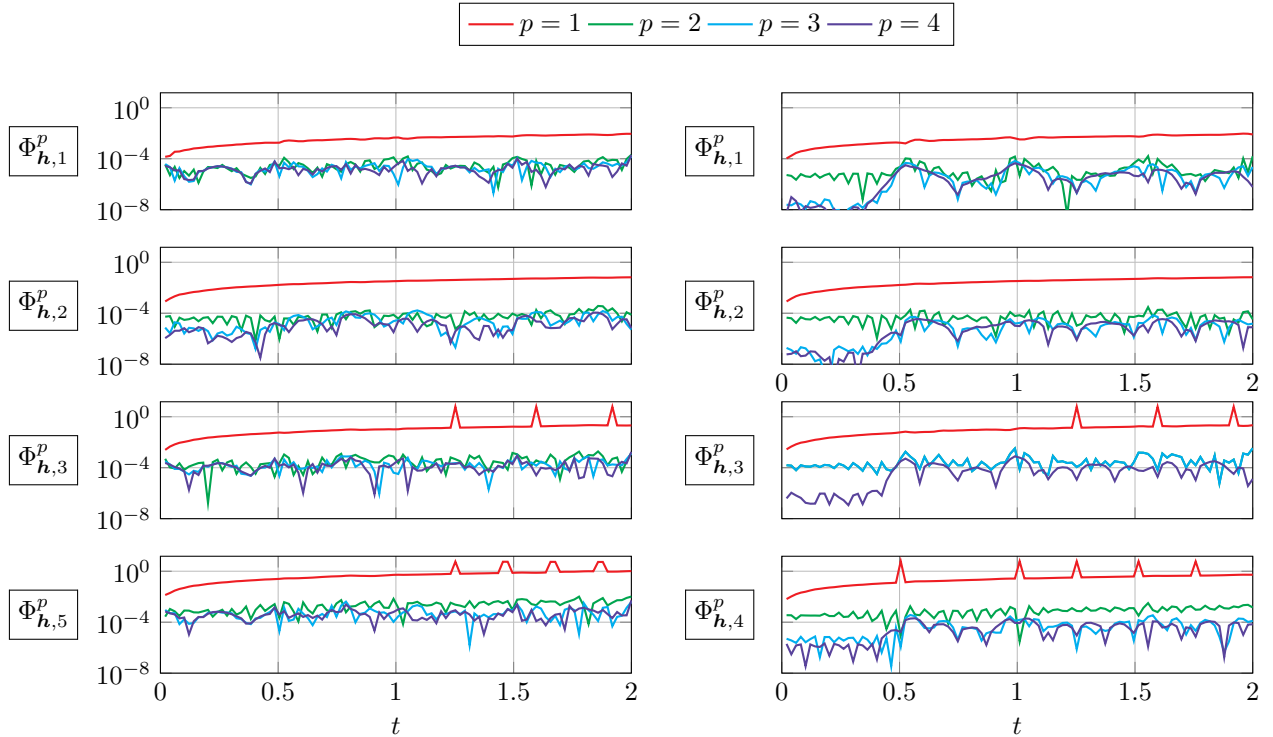

\subsubsection{Example 7. Two-dimensional wave propagation with non-constant wave speed} \label{subsec6.8}

In this last experiment, we test the effectiveness of our discretization method~\eqref{eq:matrix_system} to solve the two-dimensional wave problem in a heterogeneous material as in~\cite[Section 6.7]{PerugiaSchoberlStockerWintersteiger2020}. Let~$\Omega =(0,2)^2$. The problem under consideration is~\eqref{eq:2} (homogeneous Dirichlet boundary conditions imposed on~$\Gamma_D = \partial \Omega$) with piecewise constant wave velocity
\begin{equation*}
    c(\bx,t)= \begin{cases}
    1 & 0\le x_1 \le 1.2, \\
    3 & 1.2 < x_1 \le 2,
    \end{cases}
    \quad \quad \text{for } (\bx,t) \in Q_T = \Omega \times (0,1),
\end{equation*}
and zero source term.
The initial conditions are 
\begin{equation*}
    U_0(\bx) = e^{-\| \bx - \bx_0 \|^2/\delta^2}, \qquad  V_0(\bx) = 0, \quad \text{with}~ \bx_0 = (1,1)^\top \text{~and~} \delta = 0.01.
\end{equation*}
Figure~\ref{fig:HuygensWave} shows the numerical solution, computed with~$p=4$ and~$h_t = h_{\bx} = 0.0078$, at different time instants. As observed, the initial wave propagates through the left part of the spatial domain (with~$c=1$) until it reaches the interface between the two materials at~$ t=0.2$. By~$t=0.3$, part of the original wave and its reflection from the interface can be seen traveling to the left, while the transmitted part of the original wave moves to the right, with~$c=3$. In the final snapshot, at~$t=0.4$, the Huygens wave, which initially traveled parallel to the interface, begins to move back toward the left. These frames are similar to those obtained in \cite[Figure 10]{PerugiaSchoberlStockerWintersteiger2020}.\\
For a more quantitative comparison of our solution with the one presented in~\cite[Section 6.7]{PerugiaSchoberlStockerWintersteiger2020}, Figure~\ref{fig:HuygensEvo} shows the time evolution of the quantity
\begin{equation} \label{eq:UC}
    U_C(t) := \| U_{\bh}^p(\cdot,t) \|_{L^1(\Omega_C)},
\end{equation}
measured in~$\Omega_C := [1 - \eps_C, 1 + \eps_C,] \times [0.25 - \eps_C, 0.25 + \eps_C,]$, with~$\eps_C = 2^{-7}$. Also for this quantity a similar behaviour with that presented in \cite{PerugiaSchoberlStockerWintersteiger2020} is observed.

\begin{figure}[htbp]
\centering 
\includegraphics[width=0.9\linewidth]{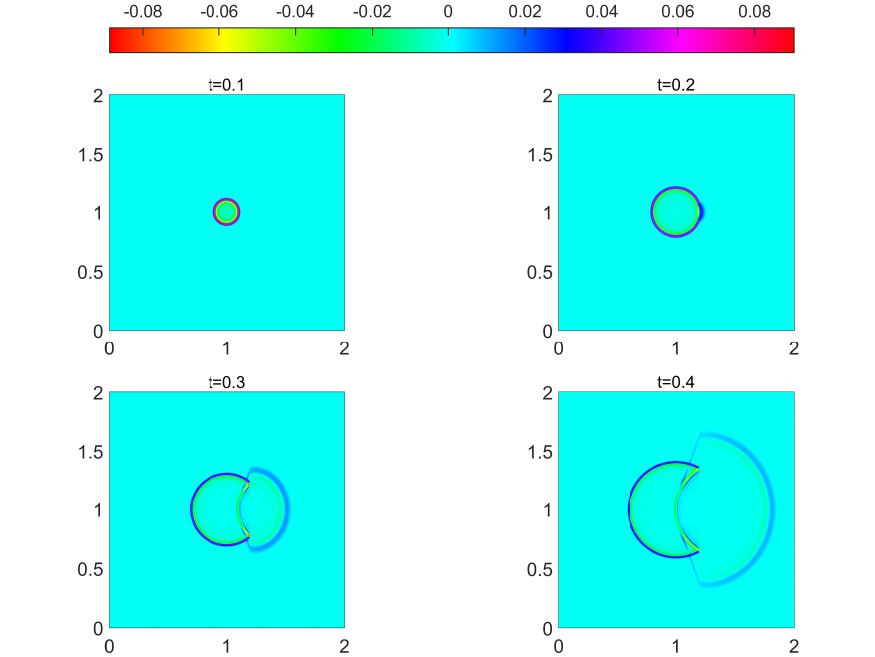}
\caption{Example 7. Snapshots of solution obtained with~$p=4$,~$h_{\bx}=h_t=0.0078$.}
\label{fig:HuygensWave}
\end{figure}

\tikzstyle{Linea1}=[thick]

\pgfplotscreateplotcyclelist{Lista1}{
	{Linea1,Red},
	{Linea1,Green},
	{Linea1,Cyan},
	{Linea1,Violet},
	{Linea2,Red},
	{Linea2,Green},
	{Linea2,Cyan},
	{Linea2,Violet},
	{Linea3,Red},
	{Linea3,Green},
	{Linea3,Cyan},
	{Linea3,Violet},
}

\def \DATAFILE {tabHuygens.csv}

\begin{figure}[htbp]
	\centering 
	\begin{tikzpicture}
		\begin{axis}[cycle list name=Lista1,
			width=0.5\linewidth,
			height=0.25\linewidth,
			xlabel={$t$},
			xmin=0,
			xmax=1,
			xminorticks=false,
			yminorticks=false,
			ylabel={$U_C$},
			grid=major,
   ]
			\addplot table [x=t, y=Uc, col sep=comma] {\DATAFILE};
		\end{axis}
	\end{tikzpicture}
	\caption{Example 7. Time evolution of $U_C$ defined in \eqref{eq:UC} obtained with~$p=4$,~$h_{\bx}=h_t=0.0078$.}
	\label{fig:HuygensEvo}
\end{figure}
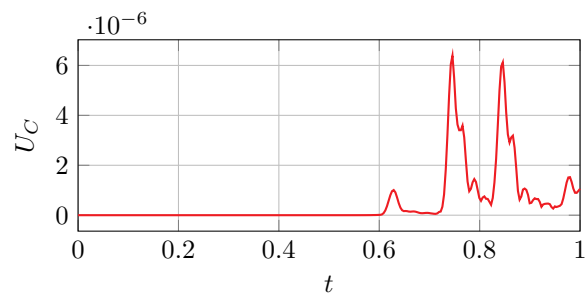

\section{Conclusion}
In this paper, we proposed an unconditionally stable conforming space--time method for the wave equation using splines of maximal regularity for the discretization in time. The method relies on a first-order-in-time formulation of the wave equation. We examined the conditioning behaviour of some families of matrices related to the temporal part of the scheme, and proved that they are weakly well-conditioned when the test space consists of splines of exactly one degree less than the trial space. It turns out that no CFL condition is required in this case. Our analysis is based on results from numerical linear algebra and properties of symbols associated with spline discretizations. We have also shown that the use of maximal regularity splines of the same degree for test and trial functions in the temporal discretization leads to schemes that are only conditionally stable. We also presented numerical tests on the full space--time formulation of the wave equation, using isogeometric discretization also in space, which validate the method and confirm the theoretical results.

\section{Acknowledegments}
This research was supported by the Austrian Science Fund (FWF) project \href{https://doi.org/10.55776/F65}{10.55776/F65} (SF, IP) and project \href{https://doi.org/10.55776/P33477}{10.55776/P33477} (MF, SF, IP). SF was also supported by the Vienna School of Mathematics. GL is member of the Gruppo Nazionale Calcolo Scientifico - Istituto
Nazionale di Alta Matematica (GNCS-INDAM). The research has received financial support from ICSC - Italian Research Center on High Performance Computing, Big Data and Quantum Computing, funded by European Union - NextGenerationEU.
\vspace{-0.4cm}
\begin{figure}[H]
	{\centering
		\hfill\includegraphics[scale=0.075]{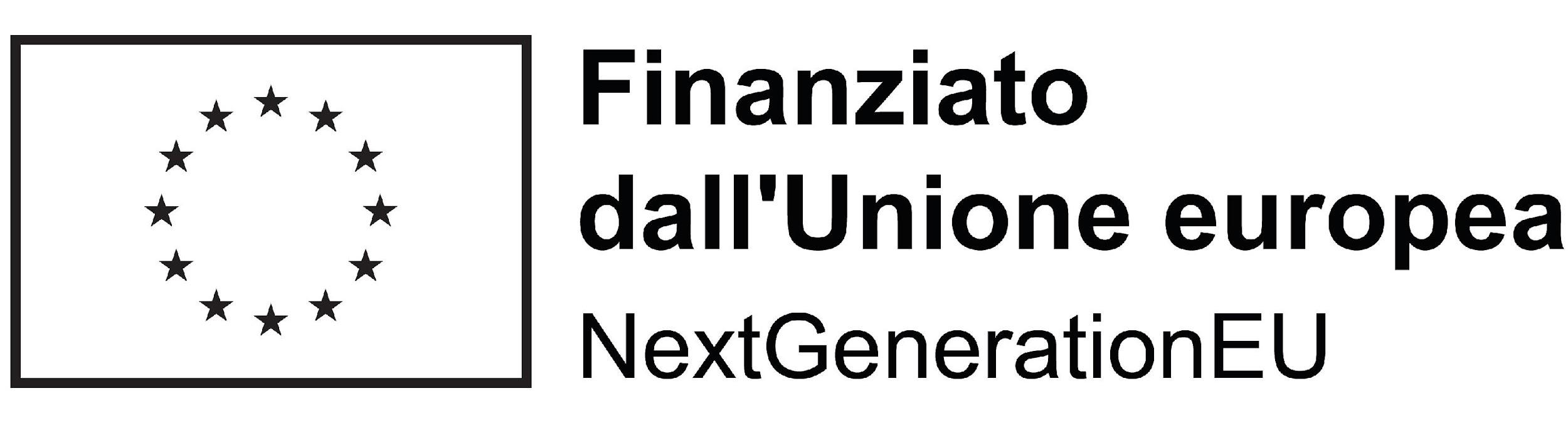}
		\hfill\includegraphics[scale=0.075]{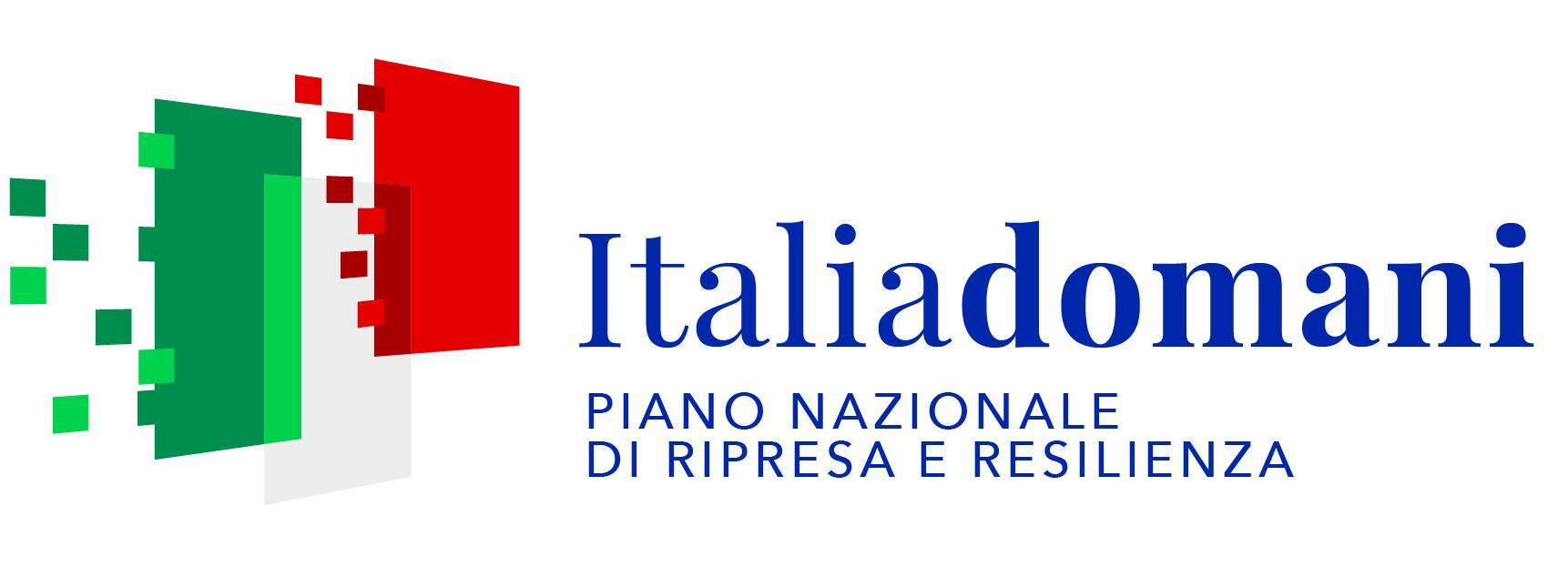}
		\hfill\includegraphics[scale=0.075]{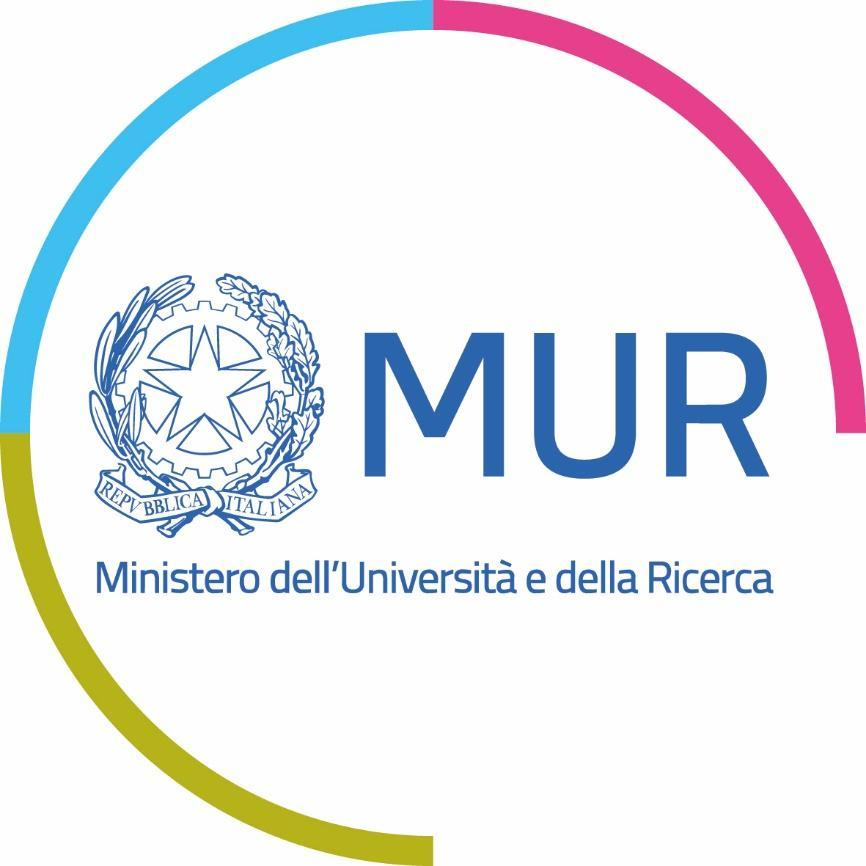}\hfill\mbox{}}
\end{figure}

\begin{appendices}
\section{An auxiliary inequality} \label{app:A}
In this appendix, we prove an inequality needed in the proof of Proposition~\ref{prop:321} and of Lemma~\ref{lem:B2} below.

\begin{lemma} \label{lem:A1}
For all~$p \in \N$ and for all~$\theta \in (0,\pi)$, we have
\begin{equation*}
    \sum_{j \in \Z} \frac{1}{(\theta + 2j\pi)^{2p+1}} < \sum_{j \in \Z} \frac{\theta}{(\theta + 2j\pi)^{2p+2}}.
\end{equation*}
\end{lemma}
\begin{proof}
The statement of the lemma is equivalent to 
\begin{equation*}
    \sum_{j \in \N} \frac{1}{(\theta - 2j\pi)^{2p+1}} + \frac{1}{\theta^{2p+1}} + \sum_{j \in \N} \frac{1}{(\theta + 2j\pi)^{2p+1}} < \sum_{j \in \N} \frac{\theta}{(\theta - 2j\pi)^{2p+2}} + \frac{1}{\theta^{2p+1}} + \sum_{j \in \N} \frac{\theta}{(\theta + 2j\pi)^{2p+2}}. 
\end{equation*}
We show that, for all~$j \ge 1$, it holds
\begin{equation} \label{eq:74}
  \frac{1}{(\theta - 2j\pi)^{2p+1}} + \frac{1}{(\theta + 2j\pi)^{2p+1}} < \frac{\theta}{(\theta - 2j\pi)^{2p+2}} + \frac{\theta}{(\theta + 2j\pi)^{2p+2}}. 
\end{equation}
With some manipulations, we obtain that~\eqref{eq:74} is equivalent to
\begin{equation*}
  (2j\pi- \theta)(\theta + 2j\pi)^{2p+2} - (\theta+2j\pi)(2j\pi-\theta)^{2p+2} + \theta\left[(2j\pi-\theta)^{2p+2} + (\theta+2j\pi)^{2p+2}\right] > 0,
\end{equation*}
or also to
\begin{equation*}
  2j\pi (2j\pi + \theta)^{2p+2} - 2j\pi(2j\pi-\theta)^{2p+2} > 0.
\end{equation*}
The latter is clearly true for all~$\theta \in (0,\pi)$, and~$j,p \ge 1$.
\end{proof}

\section{Justification for property~\eqref{eq:59}}\label{app:B}
In this appendix, we discuss property~\eqref{eq:59} of the function~$W_p(\theta,\rho)$ defined in~\eqref{eq:54}. We first prove two auxiliary results.
\begin{lemma} \label{lem:B1}
For all~$p \in \N$ and for all~$\theta \in (0,\pi)$, we have
\begin{equation*}
    \sum_{j \in \Z} \frac{1}{(\theta + 2j\pi)^{2p+1}} > \sum_{j \in \Z} \frac{\theta \sin \theta}{(\theta + 2j\pi)^{2p+3}}.
\end{equation*}
\end{lemma}
\begin{proof}
For~$j=0$, we readily obtain
\begin{equation*}
    \frac{1}{\theta^{2p+1}} > \frac{\theta \sin \theta}{\theta^{2p+3}},
\end{equation*}
since~$\sin \theta < \theta$ for all~$\theta \in (0,\pi)$. We show that for all~$j \ge 1$, the pairs of terms in the sum with opposite indices satisfy
\begin{equation} \label{eq:75}
  \frac{1}{(\theta - 2j\pi)^{2p+1}} + \frac{1}{(\theta + 2j\pi)^{2p+1}} < \frac{\theta \sin \theta}{(\theta - 2j\pi)^{2p+3}} + \frac{\theta \sin \theta}{(\theta + 2j\pi)^{2p+3}}. 
\end{equation}
With some manipulations, we obtain that~\eqref{eq:75} is equivalent to
\begin{equation*}
  (\theta-2j\pi)^2(\theta + 2j\pi)^{2p+3} + (\theta+2j\pi)^2 (\theta-2j\pi)^{2p+3} > \theta \sin \theta \left[(\theta+2j\pi)^{2p+3} + (\theta-2j\pi)^{2p+3}\right],
\end{equation*}
or also to
\begin{equation*}
  (\theta+2j\pi)^{p+3}\left[(\theta - 2j\pi)^2 - \theta \sin \theta\right] + (\theta-2j\pi)^{p+3} \left[(\theta + 2j\pi)^2 - \theta \sin \theta\right] > 0.
\end{equation*}
The latter is clearly true for all~$\theta \in (0,\pi)$, and~$j,p \ge 1$, since
\begin{equation*}
    \frac{(\theta - 2j\pi)^2 - \theta \sin \theta}{(\theta + 2j\pi)^2 - \theta \sin \theta} < 1.
\end{equation*} \qedhere
\end{proof}
\begin{lemma} \label{lem:B2}
Let~$C_p$ and~$M_p$ be defined in~\eqref{eq:30} and~\eqref{eq:33}, respectively. Then, for all~$p \ge 1$, we have
\begin{equation} \label{eq:76}
    \left( \frac{C_p'(\theta)}{M_p(\theta)} \right)' > 0 \quad \text{for all~} \theta \in (0,\pi).
\end{equation}
\end{lemma}
\begin{proof}
From~\eqref{eq:37}, we obtain
\begin{align*}
    \frac{C_p'(\theta)}{M_p(\theta)} = (p+1) \frac{\sin \theta}{1-\cos \theta} \frac{C_p(\theta)}{M_p(\theta)} + 2p+1 = (p+1) \frac{\sin \theta}{1-\cos \theta} \frac{\widehat{C}_p(\theta)}{\widehat{M}_p(\theta)} + 2p+1,
\end{align*}
with~$\widehat{C}_p$ and~$\widehat{M}_p$ defined in~\eqref{eq:47}. We deduce
\begin{align*}
    \left(\frac{C_p'(\theta)}{M_p(\theta)} \right)' & = - \frac{p+1}{1-\cos \theta} \frac{\widehat{C}_p(\theta)}{\widehat{M}_p(\theta)} + (p+1)\frac{\sin \theta}{1-\cos \theta} \frac{(\widehat{C}_p'(\theta) \widehat{M}_p(\theta) - \widehat{C}_p(\theta) \widehat{M}_p'(\theta))}{\widehat{M}_p^2(\theta)}
    \\ & = \frac{p+1}{1-\cos \theta} \frac{1}{\widehat{M}_p(\theta)} \left(- \widehat{C}_p(\theta) + \sin \theta \widehat{C}_p'(\theta) - \sin \theta \widehat{C}_p(\theta) \frac{\widehat{M}_p'(\theta)}{\widehat{M}_p(\theta)}\right).
\end{align*}
Recalling that~$-\widehat{C}_p(\theta) < \theta \widehat{M}_p(\theta)$, i.e., Lemma~\ref{lem:A1}, and that~$\widehat{M}_p'(\theta)<0$, we deduce that~\eqref{eq:76} is satisfied if
\begin{align*}
    - \widehat{C}_p(\theta) + \sin \theta \widehat{C}_p'(\theta) + \theta \sin \theta \widehat{M}'_p(\theta) > 0 \quad \text{for all~} \theta \in (0,\pi),
\end{align*}
or, equivalently, if
\begin{align*}
    - \widehat{C}_p(\theta) + \sin \theta (2p+1)\widehat{M}_p(\theta) + \theta \sin \theta \widehat{M}'_p(\theta) > 0 \quad \text{for all~} \theta \in (0,\pi).
\end{align*}
Employing~\eqref{eq:48}, we get
\begin{align*}
    - \widehat{C}_p(\theta) + \sin \theta(2p+1) \widehat{M}_p(\theta) + \theta \sin \theta \widehat{M}'_p(\theta) > - \widehat{C}_p(\theta) + \theta \sin \theta \frac{1}{2p+2} \widehat{M}_p'(\theta), \quad \text{for all~} \theta \in (0,\pi),
\end{align*}
and we conclude with Lemma~\ref{lem:B1}.
\end{proof}
In order to establish property~\eqref{eq:59}, consider the ratio~$\partial_{\theta} W_p(\theta,\rho)/(M_p(\theta) M_p'(\theta))$. This is well-defined in~$(0,\pi)$ since~$M_p(\theta)M_p'(\theta)<0$ for all~$\theta \in (0,\pi)$, and its zeros coincide with those of~$\partial_{\theta} W_p(\theta,\rho)$. We evaluate 
\begin{equation*}
    \lim_{\theta \to 0} \frac{\partial_\theta W_p(\theta,\rho)}{M_p(\theta) M_p'(\theta)} = 2 \rho -2\lim_{\theta \to 0}
    \frac{C_p(\theta) C_p'(\theta)}{M_p(\theta) M_p'(\theta)}
    =2\rho + 2\lim_{\theta \to 0}\frac{C_p(\theta)}{M_p'(\theta)}
    =2\rho + 2\frac{6}{p+1}>0,  
\end{equation*}
where the second identity follows from~\eqref{eq:36} and~\eqref{eq:38}, and the limit in the last identity is obtained directly from the definition of~$C_p$ and the expression of~$M_p'$ as
\begin{equation*}
    M_p'(\theta) = \frac{p+1}{1-\cos \theta} (\sin \theta M_p(\theta) + C_{p+1}(\theta)).
\end{equation*}
Furthermore, using~\eqref{eq:38}, we get
\begin{equation*}
    \lim_{\theta \to \pi} \frac{\partial_\theta W_p(\theta,\rho)}{M_p(\theta)M_p'(\theta)} 
    =2\rho - 2 \lim_{\theta \to \pi}\frac{C_p(\theta) C_p'(\theta)}{M_p(\theta) M_p'(\theta)} =2\rho-2(2p+1)\lim_{\theta\to\pi}\frac{C_p(\theta)}{M_p'(\theta)}. 
\end{equation*}
Computing the limit on the right-hand side by the l'H\^opital rule, from~$C_p'(\pi)=(2p+1)M(\pi)$ (see~\eqref{eq:38}) and
\begin{equation*}
    M_p''(\pi) = \frac{p+1}{2} \big((2p+3)M_{p+1}(\pi) - M_p(\pi)\big),
\end{equation*}
we obtain
\begin{equation*}
    \lim_{\theta \to \pi} \frac{\partial_\theta W_p(\theta,\rho)}{M_p(\theta)M_p'(\theta)} 
    =  2\rho - 2\,\frac{2(2p+1)^2}{p+1} \frac{M_p(\pi)}{(2p+3)M_{p+1}(\pi)-M_p(\pi)} =: 2\rho - 2E_p. 
\end{equation*}
Note that~$E_p >0$. Indeed, using~\cite[Proposition 5.6]{FerrariFraschini2024} and~\cite[Theorem 1.1]{Qi2019} (see also~\cite[Equation 2.1]{Qi2019}), we compute
\begin{equation*}
    \frac{M_{p+1}(\pi)}{M_p(\pi)} = \frac{1}{\pi^2}\frac{2^{2(p+2)}-1}{2^{2(p+1)}-1} \frac{\zeta(2(p+2))}{\zeta(2(p+1))} > \frac{4}{\pi^2} \frac{(2^{2p+4}-1)(2^{2p+1}-1)}{(2^{2p+3}-1)(2^{2p+2}-1)} > \frac{588}{155\pi^2} > \frac{1}{5} \ge \frac{1}{2p+3}\qquad \text{for all~$p \ge 1$}.
\end{equation*}
In addition, for all~$p\ge 1$, we have observed numerically, and postpone the proof to the work~\cite{Ferrari2025} in preparation, that 
\begin{equation} \label{eq:77}
    \frac{\partial}{\partial \theta} \left( \frac{\partial_\theta W_p(\theta,\rho)}{M_p(\theta) M_p'(\theta)} \right) = \frac{\partial}{\partial \theta} \left(2 \rho - 2\frac{C_p(\theta)C_p'(\theta)}{M_p(\theta)M_p'(\theta)} \right) = - 2 \left(\frac{C_p(\theta)C_p'(\theta)}{M_p(\theta)M_p'(\theta)} \right)' < 0 \quad \text{for all~} \theta \in (0,\pi).
\end{equation}
Thus, for any fixed~$0<\rho < E_p$, we expect that the function~$\partial_\theta W_p(\theta,\rho)$ has exactly one zero in~$(0,\pi)$. 
\begin{remark}
Note that $\rho_p < E_p$ is always satisfied. Indeed, from~\eqref{eq:77}, we deduce that, for all~$\rho >0$,
\begin{equation*}
    \inf_{\theta \in (0,\pi)} \frac{\partial_\theta W_p(\theta,\rho)}{M_p(\theta) M_p'(\theta)} = \lim_{\theta \to \pi} \frac{\partial_\theta W_p(\theta,\rho)}{M_p(\theta) M_p'(\theta)} = 2\rho - 2 E_p.
\end{equation*}
In particular, for all~$\theta \in (0,\pi)$, we obtain 
\begin{equation*}
     \frac{\partial_\theta W_p(\theta,\rho_p)}{M_p(\theta) M_p'(\theta)} = 2 \rho_p - 2\frac{C_p(\theta) C_p'(\theta)}{M_p(\theta) M_p'(\theta)} > 2\rho_p - 2 E_p,
\end{equation*}
from which we conclude taking~$\theta = \theta_p$ and using the second equation in~\eqref{eq:60}.
\end{remark}
\end{appendices}

\bibliography{mybibliography}{}
\bibliographystyle{plain}

\end{document}